\documentclass[numbers=noenddot]{scrartcl}

\usepackage [latin1]{inputenc}
\usepackage{amsthm,amsmath,amssymb}
\usepackage[colorlinks=true,citecolor=blue,linkcolor=blue,urlcolor=blue]{hyperref}
\usepackage[dvipsnames]{xcolor}

\theoremstyle{plain}
\newtheorem{theorem}{Theorem}[section]
\newtheorem{lemma}[theorem]{Lemma}
\newtheorem{corollary}[theorem]{Corollary}

\theoremstyle{definition}

\newtheorem{example}[theorem]{Example}

\theoremstyle{remark}
\newtheorem{remark}[theorem]{Remark}

\usepackage[dvipsnames]{xcolor}
\usepackage[ulem=normalem,draft,authormarkup=none,deletedmarkup=uwave]{changes}
\definecolor{deletedgray}{rgb}{0.95, 0.95, 0.96}
\setdeletedmarkup{{\color{deletedgray}#1}}

\definechangesauthor[name=Katharina, color=RubineRed]{KL}
\definechangesauthor[name=Thomas, color=NavyBlue]{TA}
\definechangesauthor[name=Jonny, color=ForestGreen]{JP}
\setcommentmarkup{\mpar{\IfIsColored{\color{authorcolor}}\tiny \arabic{authorcommentcount}: #1}}

\newcommand{\R}{\mathbb{R}}
\renewcommand{\div}{\mathrm{div}}
\newcommand{\mpar}[1]{\marginpar{\scriptsize\raggedright\sloppy #1}}
\numberwithin{equation}{section}
\renewcommand{\u}{u}
\newcommand{\U}{U}
\newcommand{\y}{y}
\newcommand{\Ih}{\pi}

\newcommand{\Uad}{\U^{\mathrm{ad}}}

\newcommand{\abs}[1]{\lvert#1\rvert}
\newcommand{\Abs}[1]{\left\lvert#1\right\rvert}
\newcommand{\norm}[1]{\lVert#1\rVert}

\begin{document}
\title{Numerical analysis for the
  Stokes problem with non-homogeneous \\ Dirichlet boundary condition}
\author{Thomas Apel\thanks{\texttt{thomas.apel@unibw.de}, Universit\"at der
    Bundeswehr M\"unchen, Institute of Mathematics and Com\-pu\-ter-Based
    Simulation, 85577 Neubiberg, Germany.} \and 
Katharina Lorenz\thanks{\texttt{katharina.lorenz@unibw.de}, Universit\"at der
    Bundeswehr M\"unchen, Institute of Mathematics and Com\-pu\-ter-Based
    Simulation, 85577 Neubiberg, Germany.} \and 
Johannes Pfefferer\thanks{\texttt{johannes.pfefferer@unibw.de}, Universit\"at der
    Bundeswehr M\"unchen, Department of Electrical Power Systems and Information Technology, 85577 Neubiberg, Germany.}}
\maketitle

  \setlength{\parindent}{0pt}\setlength{\parskip}{1ex plus 0.5ex minus 0.5ex}
  \textbf{Abstract:} The Stokes problem with non-homogeneous Dirichlet boundary condition is solved numerically using conforming discretizations and an approximation of the boundary datum in the corresponding trace space.   Optimal discretization error estimates in various norms are derived.  
In the case of a homogeneous differential equation, all estimates of the velocity error are uniformly valid in the viscosity parameter for all conforming inf-sup stable discretizations.
The theory accounts for the influence of corner singularities in the case of a non-convex domain.   Several variants of the boundary data approximation are discussed.
  Moreover, the case of  boundary data with very low regularity is studied, where a weak solution does not exist. The well-posedness of the very weak solution is investigated, and optimal discretization error  estimates are derived.  Numerical tests confirm the theory.

  The compatibility condition for the boundary data is not necessary for well-posedness of the weak and very weak formulations but it ensures that the solution satisfies the continuity equation in the distributional sense. In the same spirit, the compatibility condition is not necessary for the approximating boundary data; a good approximation of the original boundary data is important.
  

\newpage
\section{Introduction}

The investigation of the Stokes problem with non-homogeneous boundary conditions is motivated by the numerical solution of the Dirichlet boundary control problem 
\begin{align*} 
	\min_{(\y,\u)\in  Y\times \Uad} J(\y,\u) :\!&=
	\frac{1}{2}\|\y-\y_d\|_{L^2(\Omega)^d}^2 + \frac{\sigma}{2}
	\|\u\|_{L^2(\Gamma)^d}^2 \\ \text{subject to}\quad \y &= S\u \end{align*} with the control-to-state mapping $S:\Uad\subset L^2(\Gamma)^d\to  Y$, given by $\u\mapsto \y=S\u$, 
    \begin{subequations}\label{eq:stokesclassical}
        \begin{align}
        -\nu\Delta \y+\nabla p &= 0 \quad\text{in }\Omega\subset\R^d, \quad d=2,3, \label{eq:stokesclassical1}\\
        \nabla\cdot \y&=0\quad \text{in }\Omega, \label{eq:stokesclassical2}\\
        \y&=\u\quad\text{on }\Gamma=\partial\Omega\label{eq:stokesclassical3},
        \end{align}
    \end{subequations}
where $\sigma$ and $\nu$ are positive contants, $y_d$ is the prescribed desired state and $\Uad$ is the admissible set of controls. Hence we need to treat non-homogeneous Dirichlet data $\u\in H^t(\Gamma)^d$, $d=2,3$. Since the state and the control variable are denoted by $y$ and $u$ in large parts of the optimal control community we adopt this notation here, although optimal control will be the subject of a future paper only. 

This paper deals with the finite element solution of the Stokes problem \eqref{eq:stokesclassical}.
Since the problem is linear and since the case of homogeneous boundary conditions but non-ho\-mo\-geneous equations is widely studied in the literature, we simplify the presentation by considering homogeneous equations.   We distinguish between two cases for the non-homogeneous Dirichlet data $\u\in H^t(\Gamma)^d$: $t \geq \frac{1}{2}$, where a weak solution $(\y, p) \in H^1(\Omega)^d \times L^2_0(\Omega)$ exists, and $t<\frac{1}{2}$, where the existence of a weak solution cannot generally be expected. The latter case is important in our motivational application but also in other well-known examples like the lid-driven cavity problem. We allow the polygonal or polyhedral domain $\Omega\subset\R^d$ to be non-convex which introduces additional particularities.

The investigation starts in Subsection \ref{sec:2.1} with the weak formulation of the Stokes problem: Find 
$\y\in Y_*:=\{v\in H^1(\Omega)^d: v=\u\text{ a.\,e.~on }\Gamma\}$ and $p\in Q:=L^2_0(\Omega):=\left\{v\in L^2(\Omega):(v,1)=0\right\}$ such that
\begin{subequations}\label{eq:weakneui}
\begin{align}
  \nu(\nabla \y,\nabla v)-(\nabla\cdot v,p) &= 0 \quad\forall v\in  Y_0=H^1_0(\Omega)^d, \label{eq:weaki1}\\
  (\nabla\cdot \y,q) &=0 \quad\forall q\in Q. \label{eq:weaki2}
\end{align}
\end{subequations}
For $\u\in H^{\frac12}(\Gamma)^d$, there is a unique weak solution which solves \eqref{eq:stokesclassical1}  in the sense of distributions, see Lemma \ref{lem:2.1}. 
A recurring topic in the paper is the investigation of the role of the compatibility condition. 
One has to realize that \eqref{eq:weaki2} corresponds not only to \eqref{eq:stokesclassical2} but also to $\nabla\cdot \y=c=\mathrm{const}.$, and the weak solution solves this equation almost everywhere with $c=|\Omega|^{-1}\langle\u,n\rangle_\Gamma$. The compatibility condition
\begin{equation}\label{eq:compatibility}
\langle\u,n\rangle_\Gamma=0
\end{equation}
is not necessary for the existence of the weak solution $(\y,p)$ but for the property that $\y$ solves \eqref{eq:stokesclassical2} almost everywhere.

The discretization of problem \eqref{eq:weakneui} is based on conforming approximation spaces $Y_h\subset H^1(\Omega)^d$ and $Q_h\subset L^2_0(\Omega)$ such that the pair $(Y_{0h},Q_h)$ with $Y_{0h}=\{v\in Y_h:v|_\Gamma=0\}$ form an inf-sup stable pair, e.\,g.\ the Taylor--Hood element, the MINI element, the SMALL element \cite{BoffiBrezziFortin:13}, or the Scott--Vogelius element. The method is to approximate the boundary data $\u$ by $\u_h$ in the trace space $Y_h^\partial:=\mathrm{tr}\, Y_h$ and to seek the approximate solution $(\y_h,p_h)\in Y_{*h}\times Q_h$, where $Y_{*h}=\{v\in Y_h:v|_\Gamma=\u_h\}$, such that
\begin{align}\label{eq:v_h_i}
   \nu(\nabla \y_h,\nabla v_h)-(\nabla\cdot v_h,p_h)-(\nabla\cdot \y_h,q_h) & =0 \quad\forall
   (v_h,q_h)\in Y_{0h}\times Q_h .
\end{align}
For the approximation $\u_h$ of $\u$ we discuss the $L^2(\Gamma)^d$-projection and (quasi-)inter\-polants in Subsection \ref{sec:reg_error_neu}.  Our main message is that a sufficiently good approximation order of the discretization error $\u-\u_h$ is important while it is not important whether or not the discrete version of the compatibility condition \eqref{eq:compatibility} is satisfied. Nevertheless, we show how the condition $\langle\u_h,n\rangle_\Gamma=0$ can be enforced.

Section \ref{sec:NumerikWeak} is devoted to the analysis of the error between the discrete solution $(\y_h,p_h)$ and the weak solution $(\y,p)$ under general smoothness assumptions for the solution $(\y,p)$ and the boundary data $\u$. We start with the error estimates for $\norm{\nabla(\y-\y_h)}_{L^2(\Omega)^{d\times d}}$ and $\norm{p-p_h}_{L^2(\Omega)}$, see Theorem \ref{th:approxweak}, and proceed with weaker norms $\norm{\y-\y_h}_{H^{1-s}(\Omega)^d}$ and $\norm{p-p_h}_{(H^{s}(\Omega))'}$, $s\in(\frac12,1]$,  see Theorem \ref{lem:L2velocity}. The approximation order depends on the choice of the approximation $\u_h$ of $\u$, the available regularity of the data $\u$, and regularity issues due to corner singularities. 
A particular result is that all estimates for the velocity error are valid uniformly in the viscosity parameter $\nu$ for all conforming inf-sup stable discretizations. This property is achieved in problems with non-homogeneous equation \eqref{eq:stokesclassical1} only with pressure robust discretizations, see, e.\,g., \cite{EickmannScottTscherpel:25,LedererLinkeMerdonSchoeberl:17,LinkeMerdon:16}. Therefore we call our estimates pressure robust as well.

Section \ref{sec:veryweakformulation} addresses the  very weak formulation of problem \eqref{eq:stokesclassical}, presenting existence and regularity results in Theorem \ref{thm:vwf_wellposed} and Corollary \ref{cor:vwf_wellposed}. The formulation is defined by the method of transposition \cite{Lions1968}, i.\,e.\  by an additional integration by parts: Find $(\y,p) \in \mathcal{Y} $  such that
\begin{align}\label{eq:bilinearform_i}
	a((\y,p),(v,q)):=\langle \nu\y,-\Delta v+\nabla q\rangle-\langle \nabla\cdot v,p\rangle =\langle \nu\u,qn-\partial_n v\rangle_\Gamma \quad\forall (v,q) \in \mathcal{V}
\end{align}
with appropriate spaces $\mathcal{Y}$ and $\mathcal{V}$. A key part is to prove inf-sup conditions for the bilinear form in Lemma \ref{lem:infsup}. Using the Babu\v{s}ka--Lax--Milgram theorem we infer the existence, uniqueness, and regularity of the solution, see Theorem \ref{thm:vwf_wellposed}. Depending on the regularity of $\u\in H^t(\Gamma)^d$, the method of proof is different for $t\in(-\frac12,0)$, $t=0$ and $t\in(0,\frac12]$. Since $t=\frac12$ is included we get in particular that the weak solution is also a very weak solution. Again, the compatibility condition \eqref{eq:compatibility} is not necessary for these results, only for $y$ solving \eqref{eq:stokesclassical2} in the distributional sense, see Lemma \ref{lem:4.6}. 

For the numerical solution of \eqref{eq:bilinearform_i} we assume that the boundary datum $\u$ is at least in $L^2(\Gamma)^d$ such that we can obtain an approximation $u_h$ by the $L^2(\Gamma)^d$-projection or a quasi-interpolant, and proceed with the fomulation \eqref{eq:v_h_i}. For the error analysis in Section \ref{sec:NAveryweak} we have to use a weak norm such that $(\y,p)$ is in the corresponding space. We obtain optimal estimates in Theorem \ref{th:erresty-y_h} which are confirmed by numerical tests in Section \ref{sec:test}. Again, the estimates for the velocity error are valid uniformly in the viscosity parameter $\nu$.

Sections \ref{sec:NumerikWeak}--\ref{sec:NAveryweak} contain in each first subsection the main results with examples and extensive discussion. The proofs are postponed to further subsections.

\begin{remark}[conforming discretization]
  An alternative numerical method could be to discretize \eqref{eq:bilinearform_i} by using spaces $\mathcal{ Y}_h\subset\mathcal{Y}$ and $\mathcal{V}_h\subset\mathcal{V}$ but this discretization requires $C^1$-elements for $\mathcal{V}_h$ and does not improve accuracy since our error estimates are optimal.\qed
\end{remark}

\bigskip 

As written above, our interest in non-homogeneous Dirichlet boundary conditions arose from a boundary control problem with the particularity that the data are non-smooth, in particular $\u\not\in H^{\frac12}(\Gamma)^d$. By scanning standard textbooks we found that even the case $\u\in H^t(\Gamma)^d$, $t\ge\frac12$, is not well treated in the literature. The typical approach is to homogenize the conditions on the continuous level, i.\,e., to find an extension $E\u\in H^1(\Omega)^d$ and to seek $\y_0=\y-E\u$ with homogeneous boundary conditions, see e.\,g.\ \cite[Rem. 8.2.2]{BoffiBrezziFortin:13} or \cite[Re. 4.9]{John:16}. The discussion of non-homogeneous Dirichlet boundary conditions was revitalized by authors considering boundary optimal control problems, see the references below.

Our treatment of non-smooth Dirichlet boundary conditions is inspired by our investigation of the Dirichlet problem for the Poisson equation with $L^2(\Gamma)$ data and non-convex polygonal domains, see \cite{ApelNicaisePfefferer:16,ApelNicaisePfefferer:17}, and the papers \cite{DuranGastaldiLombardi:20,HamoudaTemamZhang:17,GongHuMateos:20,GongMateosSinglerZhang:22,ZhouGong:22}. 
Apel, Nicaise, and Pfefferer \cite{ApelNicaisePfefferer:16} investigated the method of transposition in the proper function spaces to cast non-convex domains, a novel regularization approach and compared it with another formulation due to Berggren \cite{berggren:04}. They also investigated the finite element approximation including mesh grading or dual singular functions.
Hamouda, Temam, and Zhang \cite{HamoudaTemamZhang:17} introduce the transposition approach for the Stokes problem but consider convex or smooth domains only and do not consider discretization aspects. 
Duran, Gastaldi, and Lombardi \cite{DuranGastaldiLombardi:20} consider the regularization approach and discretization aspects assuming a convex domain, but their discussion of the condition $\int_\Gamma \u^h\cdot n=0$ is not completely correct, see Remark \ref{rem:approximationvariants}. 
Gong, Mateos, Singler, and Zhang \cite{GongMateosSinglerZhang:22} introduce a proper formulation based on the transposition approach but with function spaces other than ours.  Moreover, they show a non-optimal discretization error estimate.
An almost optimal error estimate was derived by Zhou and Gong \cite{ZhouGong:22} but assuming convexity of the domain and condition \eqref{eq:compatibility}. The most recent contribution is by Eickmann, Scott, and Tscherpel \cite{EickmannScottTscherpel:25} who analyze the discretization of the weak formulation. Since they allow for a non-homogeneous equation \eqref{eq:stokesclassical1}, their general estimate cannot be pressure robust. In order to establish a pressure robust method, they analyze then the Scott--Vogelius method but, different to us, they stress the compatibility condition.

Gunzburger and Manservisi \cite{GunzburgerManservisi:00} consider a boundary control problem for the non-stationary Navier--Stokes equations. Their setting is quite smooth; they assume a $C^2$ domain and seek controls at least with regularity such that a weak solution exists.
John and Wachsmuth \cite{JohnWachsmuth:09} consider the Stokes- and Navier--Stokes problem  in a $C^2$-domain as well. A very weak solution of the Stokes problem is introduced where the existence is traced back to a paper of Amrouche and Girault \cite{AmroucheGirault:91}. The discretization is not analyzed but examples are calculated with Taylor--Hood elements.
Due to the assumptions on the domain, the very weak solution was defined with test functions $(v,q)\in H^2(\Omega)\times H^1(\Omega)$ \cite{AmroucheGirault:91,JohnWachsmuth:09}. In all three of the mentioned papers, the divergence constraint is included in the solution and/or test space, i.\,e., solenoidal spaces $\{v\in H^2(\Omega)^d:\nabla\cdot v=0,v|_\Gamma=0\}$ \cite{JohnWachsmuth:09} or $\{v\in H^1_0(\Omega):\nabla\cdot v=0\}$ \cite{GunzburgerManservisi:00} or $\{v\in L^r(\Omega):\nabla\cdot v=0\}$ \cite{JohnWachsmuth:09} are considered. The compatibility condition $\langle\u,n\rangle_\Gamma=0$ is always required, in the case of optimal boundary control directly in the space where the control is seeked, namely in $\{\u\in H^1(\Gamma)^d: \langle\u,n\rangle_\Gamma=0\}$ \cite{GunzburgerManservisi:00} or $\{\u\in L^2(\Gamma)^d: \langle\u,n\rangle_\Gamma=0\}$ \cite{JohnWachsmuth:09}.

In the whole paper, the letter $c$ will denote a generic constant that may change from line to line.

\section{\label{sec:basic}Basic ideas}
\subsection{\label{sec:2.1}Weak formulation}

We start with the Stokes problem \eqref{eq:stokesclassical} in weak formulation. 
Let $\Omega\in\R^d$, $d=2,3$, be a polygonal or polyhedral domain.
Convexity is not assumed. We use the standard Sobolev spaces
$H^s(\Omega)$, $H^s_0(\Omega)$, and $H^s(\Gamma)$ and denote 
by $(.,.)$ the $L^2(\Omega)$ and the $L^2(\Omega)^d$ scalar product, by
$(.,.)_\Gamma$ the $L^2(\Gamma)$ and the $L^2(\Gamma)^d$ scalar product.
Similarly, we denote by $\langle.,.\rangle$ and $\langle.,.\rangle_\Gamma$ duality products.
Furthermore,  we use the spaces
\begin{align*}
  Y_0&=H^1_0(\Omega)^d, \quad Y_*=\{v\in H^1(\Omega)^d: v=\u\text{ a.\,e.~on }\Gamma\},\\
  \quad Q&=L^2_0(\Omega):=\left\{v\in L^2(\Omega):(v,1)=0\right\}.
\end{align*}
The weak formulation is: Find $\y\in Y_*$ and $p\in Q$ such that
\begin{subequations}\label{eq:weakneu}
\begin{align}
  \nu(\nabla \y,\nabla v)-(\nabla\cdot v,p) &= 0 \quad\forall v\in Y_0, \label{eq:weak1}\\
  (\nabla\cdot \y,q) &=0 \quad\forall q\in Q. \label{eq:weak2}
\end{align}
\end{subequations}
\begin{lemma}[existence of weak solution]\label{lem:2.1}
There exists a unique solution $(\y,p)\in Y_*\times Q$ of the system \eqref{eq:weakneu} for
$\u\in H^{\frac12}(\Gamma)^d$ satisfying
\[
    \nu\norm{\y}_{H^1(\Omega)^d}+\norm{p}_{L^2(\Omega)}\le c\nu\norm{\u}_{H^{\frac12}(\Gamma)^d}.
\]
\end{lemma}
\begin{proof}
  To show the assertion we use a classical lifting argument. There is a linear and continuous operator $E\colon H^{\frac12}(\Gamma)^d\to H^1(\Omega)^d$ (the extension operator) such that for $\u\in H^{\frac12}(\Gamma)^d$ there holds $E\u|_\Gamma=\u$, see e.\,g.\ \cite{Ding:96}. Let us define $f\in Y_0'$ by
  \[
    \langle f,v\rangle:=-(\nu\nabla E\u,\nabla v)
  \]
  and $g\in Q$ by
  \[
    g:=-\nu\nabla\cdot E\u+\delta, \quad \delta=\abs{\Omega}^{-1}(\nu\nabla\cdot E\u,1),
  \]
  such that
  \[
    (g,q)=(-\nu\nabla\cdot E\u,q)+\left(\delta,q\right)=(-\nu\nabla\cdot E\u,q)\quad\forall q\in Q.
  \]
  Obviously, there holds $(g,1)=0$.
  As a consequence, according to \cite[Thm. 8.2.1]{BoffiBrezziFortin:13}, see also \cite[Rem. 4.9]{John:16}, there exists a unique pair $(\y_0,p)\in Y_0\times Q$ satisfying
  \begin{align*}
    \nu(\nabla \y_0,\nabla v)-(\nabla\cdot v,p) &= \langle f,v\rangle \quad\forall v\in Y_0, \\
    \nu(\nabla\cdot \y_0,q) &=(g,q) \quad\forall q\in Q,
  \end{align*}
  and
  \begin{equation}\label{eq:apriory0}
    \nu\norm{\y_0}_{ Y_0}+\norm{p}_{Q}\le c \left(\norm{f}_{ Y_0'}+\norm{g}_{Q}\right)\le c\nu\norm{\nabla E\u}_{L^2(\Omega)^{d\times d}}\le c\nu\norm{\u}_{H^{\frac12}(\Gamma)^d},
  \end{equation}
  where we used the continuity of the extension operator $E$ in the last step.
  Let us now define $\y:=E\u+\y_0\in Y_*$. By rearranging terms in the previous system, we obtain that the pair $(\y,p)$ solves the system \eqref{eq:weakneu}. Moreover, there holds the a priori estimate of the assertion due to \eqref{eq:apriory0} and the continuity of the extension operator $E$.
\end{proof}
\begin{remark}[on the viscosity parameter $\nu$]\label{rem:vissol}
     Iff $(\bar y, \bar p)$ solves problem~\eqref{eq:weakneu} with $\nu=1$ then $(y, p)=(\bar y, \nu\bar p)$ obviously solves problem~\eqref{eq:weakneu} with general $\nu>0$.
\qed
\end{remark}
\begin{remark}[on solving the original system \eqref{eq:stokesclassical}]\label{remark:origsys}
Integration by parts shows that the solution of \eqref{eq:weakneu} satisfies the momentum equation \eqref{eq:stokesclassical1}
  in the distributional sense, since $C_0^\infty(\Omega)$ is dense in $H^1_0(\Omega)$. In order to investigate the divergence equation \eqref{eq:stokesclassical2}, it is necessary that equation \eqref{eq:weak2} holds for all $q\in
L^2(\Omega)=Q\oplus\mathcal{P}_0(\Omega)$, where $\mathcal{P}_0(\Omega)$ denotes the space of constant functions in $\Omega$. However, testing the left hand side of \eqref{eq:weak2} with $q\in\mathcal{P}_0(\Omega)$ and integrating by parts shows that
\[
    (\nabla\cdot \y,q) = \left(\u_\Gamma,q\right)  \quad\forall q\in \mathcal{P}_0(\Omega)
\]
with
\[
    \u_\Gamma:=\abs{\Omega}^{-1}\langle\u,n\rangle_\Gamma\in \mathcal{P}_0(\Omega),
\]
where $n$ denotes the outward normal to $\Gamma$. This, together with \eqref{eq:weak2} and the definition of $Q$, leads to
\[
    (\nabla\cdot \y,q) = \left(\u_\Gamma,q\right)  \quad\forall q\in L^2(\Omega)
\]
and hence to
\[
  \nabla\cdot\y=\u_\Gamma \quad\text{a.\,e. in } \Omega.
\]
Therefore, only if the additional compatibility condition
\begin{align*}
  \u_\Gamma = 0\quad\Leftrightarrow\quad\langle\u,n\rangle_\Gamma=0
\end{align*}
holds, the divergence equation \eqref{eq:stokesclassical2} is satisfied almost everywhere. The boundary condition \eqref{eq:stokesclassical3} is automatically fulfilled as it is incorporated into the affine space $Y_*$. \qed
\end{remark}
\begin{remark}[on higher regularity]
    Depending on the domain and on the regularity of the Dirichlet boundary datum $u$, it is also possible to deduce higher regularity for the solution $(\y,p)$. \qed
\end{remark}

\subsection{\label{sec:regularity}Regularity of a dual problem with homogeneous boundary condition}

Later on, we will
need a detailed regularity statement of a dual problem. It also motivates the function spaces which we use, and limitations on some parameters. To this end, introduce the space
\begin{align}\label{def:V02beta}
  V^{0,2}_\beta(\Omega)=\{v: r^\beta v\in L^2(\Omega)\}, \quad\beta\in\R,
 \end{align}
where $r=r(x)$ is the distance of $x\in\Omega$ to the set
$A_0(\Omega)$ of vertices of $\Omega$ in the two-dimensional case, and
the distance of $x\in\Omega$ to the set $A_1(\Omega)$ of edges of
$\Omega$ in the three-dimensional case.

Consider the weak solution $(v,q)$ to the Stokes problem
\begin{subequations}\label{eq:stokes_dual_neu}
  \begin{align}
    -\Delta v +\nabla q &= f\quad\text{in }\Omega,\label{eq:stokes_dual_neu1}\\
    \nabla \cdot v &=g\quad\text{in }\Omega,\label{eq:stokes_dual_neu2}\\
    v&=0\quad\text{on }\Gamma,\label{eq:stokes_dual_neu3}
  \end{align}
\end{subequations}
which is defined as follows: For $f\in Y_0'$ and $g\in Q$ find $(v,q)\in Y_0\times Q$ satisfying
\begin{subequations}\label{eq:stokes_dual_neu_weak}
\begin{align}
(\nabla v,\nabla \varphi)-(\nabla\cdot \varphi,q) &= \langle f,\varphi\rangle \quad\forall \varphi\in Y_0, \\
(\nabla\cdot v,\psi) &=(g,\psi) \quad\forall \psi\in Q.
\end{align}
\end{subequations}
The existence of a unique solution is granted according to \cite[Thm. 8.2.1]{BoffiBrezziFortin:13} or \cite[Thm. 2.4, Rem. 2.5]{temam:79}, see also \cite[Rem. 4.9]{John:16}. Similarly to Remark~\ref{remark:origsys}, this weak solution solves the original system \eqref{eq:stokes_dual_neu} in the distributional sense only if the compatibility condition $(g,1)=0$ is fulfilled, which is satisfied due to $ g \in Q$. If the compatibility condition $(g,1)=0$ is ignored, there is still a unique weak solution, which, however, only satisfies $\nabla\cdot v = g-\abs{\Omega}^{-1}(g,1)$ almost everywhere in $\Omega$.

We now review higher regularity results of Kellogg and Osborn \cite{KelloggOsborn:76}, Dauge \cite{dauge:89}, as well as Maz'ya and Rossmann \cite{MazyaRossmann:07, MazyaRossmann:10}. In general, the regularity of the
solution of \eqref{eq:stokes_dual_neu}  is  limited by corner and possibly edge singularities, which are characterized by the parameter $\xi\in\R_+$, see Remark \ref{rem:xi}. Important for us is that $\xi>1$ for convex domains and
$\xi\in(\frac12,1)$ for non-convex domains.
Let \[ s\in[0,\xi)\cap [0,1],\quad s\not=\frac12.\]   Note that it means $s\in [0,1]$, $s\not=\frac12$, for
convex domains and $s\in[0,\xi)$, $s\not=\frac12$, for non-convex domains. With this we are able to state the  regularity result: 
If $(v,q)\in Y_0\times Q$ is the solution of the system \eqref{eq:stokes_dual_neu} (possibly under the assumption $(g,1)=0$) and if the data $f$ and $g$ satisfy $f\in H^{s-1}(\Omega)^d$ and $g\in H^{s}(\Omega)$ and in case of $s=1$ additionally $g\in V_{-1}^{0,2}(\Omega)$, then there holds $(v,q)\in H^{1+s}(\Omega)^d\times H^s(\Omega)$, see
\cite[Theorems 5.5 and 9.20]{dauge:89} for $s\in(0,1]$, $s\neq\frac12$, and also
\cite[Theorem 2]{KelloggOsborn:76} for $s=1$ and $d=2$ and
\cite[Section 5.5]{MazyaRossmann:07} in combination with \cite[(3.2.1) on page 90]{MazyaRossmann:10} for $s=1$ and $d=3$. Moreover, there is the estimate
\begin{align}
  \norm{v&}_{H^{1+s}(\Omega)^d}+\norm{q}_{H^s(\Omega)} \notag
  \\&\le c
  \begin{cases}
  \norm{f}_{H^{s-1}(\Omega)^d}+\norm{g}_{H^{s}(\Omega)}
  &\text{for }s\in[0,\xi)\cap [0,1),\ s\ne\frac12,\\
  \norm{f}_{L^2(\Omega)^d}+\norm{g}_{H^{1}(\Omega)}
  +\norm{g}_{V_{-1}^{0,2}(\Omega)}
  &\text{for }s=1 \text{ and }\xi>1.
  \end{cases}
  \label{eq:aprioridauge}
\end{align}
If $s\in(0,1)$ this estimate is a consequence of the Fredholm property of the Stokes operator, which is outlined in \cite[Theorem 3.3]{dauge:89}. For $s=0$ it follows from e.\,g.\ \cite[Thm. 8.2.1]{BoffiBrezziFortin:13} or \cite[Thm. 2.4, Rem. 2.5]{temam:79}. For $s=1$ we distinguish between $d=2$ and $d=3$. In the first case, the a priori estimate is stated in \cite[Theorem 2]{KelloggOsborn:76}. In the second case, the estimate can be deduced from \cite[Theorem 3.13]{MazyaRossmann:07}.

\begin{remark}[more information on $\xi$] \label{rem:xi}
  In the two-dimensional case, denote
  $\xi_1(\omega)=\min_k\mathrm{Re}(\lambda_k(\omega))$, where
  $\lambda_k(\omega)\in\mathbb{C}$ are the roots with
  $\mathrm{Re}(\lambda_k(\omega))\ge0$ of
  \[
    \lambda^{-2}(\lambda-1)^{-2}\left(\sin^2(\lambda\omega)-\lambda^2\sin^2\omega\right)=0.
  \]
  The meaning is, in simplified words, that the solution of
  \eqref{eq:stokes_dual_neu} behaves near a vertex with opening angle
  $\omega$ like $v\sim r^{\xi_1(\omega)}$, $q\sim
  r^{\xi_1(\omega)-1}$.  To get an impression, let us state that
  $\xi_1(\omega)>\frac\pi\omega$ if $\omega\in(0,\pi)$ and
  $\xi_1(\omega)\in(\frac12,\frac\pi\omega)$ if $\omega\in(\pi,2\pi)$,
  see \cite[Sect. 5.1]{dauge:89}. Denote by $\omega_x$ the interior
  angle of $\Omega$ at the vertex $x\in A_0(\Omega)$, then
  $\xi=\min\{\xi_1(\omega_x):x\in A_0(\Omega)\}$.  

  In the three-dimensional case, the minimum of $\xi_1(\omega)$ and
  $\pi/\omega$ describes the regularity near an edge with opening
  angle $\omega$, and there is another number $s_x\in\R$ describing
  the regularity near a vertex $x\in A_0(\Omega)$. There holds $s_x>1$
  if $\Omega$ is convex and $s_x>\frac12$ in the general case, see
  \cite[Def.~9.19]{dauge:89} and the subsequent text. The
  characteristic value is now defined by $\xi=\min\big\{ 
  \min\{\xi_1(\omega_x), \pi/\omega_x :x\in A_1(\Omega)\}, \min\{s_x:x\in A_0(\Omega)\}
  \big\}$, see \cite[Thm.~9.20]{dauge:89}. \qed
\end{remark}

Before we close this section, we introduce and discuss for later usage in Section~\ref{sec:veryweakformulation} the domain of the Stokes operator, i.\,e.\ the differential operator associated with~\eqref{eq:stokes_dual_neu_weak}. Using
\begin{equation}\label{eq:Y_s}
  \mathcal{ Y}_s:=
  \begin{cases}
    H^{1-s}_0(\Omega)^d\times (H^{s}(\Omega)\cap L_0^2(\Omega))'&\text{if }s\in[0,1),\, s\neq \frac12,\\
    L^2(\Omega)^d\times (H^{1}(\Omega)\cap L_0^2(\Omega)\cap V_{-1}^{0,2}(\Omega))'&\text{if } s=1,
  \end{cases}
\end{equation}
such that 
\begin{equation}\label{eq:V_s}
  \mathcal{ Y}_s'=
  \begin{cases}
    H^{s-1}(\Omega)^d\times (H^{s}(\Omega)\cap L_0^2(\Omega))&\text{if }s\in[0,1),\, s\neq \frac12,\\
    L^2(\Omega)^d\times (H^{1}(\Omega)\cap L_0^2(\Omega)\cap V_{-1}^{0,2}(\Omega))&\text{if } s=1,
  \end{cases}
\end{equation}
the domain of the Stokes operator in $\mathcal{Y}_s'$ is defined by
\begin{align*}
    \mathcal{V}_s &:=\big\{(v,q)\in H^1_0(\Omega)^d\times L^2_0(\Omega)\colon 
  (-\Delta v+\nabla q ,\nabla \cdot v)\in \mathcal{Y}_s'\big\}
\end{align*}
for $s\in[0,1]\setminus\{\frac12\}$, where $-\Delta v+\nabla q$ is understood in the distributional sense. This means, $\mathcal{V}_s$ with $s\in[0,1]\setminus\{\frac12\}$ contains all solutions $(v,q)\in H^1_0(\Omega)^d\times L^2_0(\Omega)$ of \eqref{eq:stokes_dual_neu_weak} with right hand side $(f,g)\in \mathcal{Y}_s'$. In particular, there holds
\begin{equation}\label{eq:dual_distri}
\langle -\Delta v + \nabla q,\varphi \rangle +\langle \nabla \cdot v, \psi \rangle = \langle f,\varphi\rangle +\langle g, \psi \rangle \quad\forall (\varphi,\psi)\in \mathcal{ Y}_s.
\end{equation}
Corresponding norms are defined by
\begin{equation}\label{eq:normYs}
    \norm{(\y,p)}_{\mathcal{ Y}_s}
    :=
    \begin{cases}
        \norm{\y}_{H^{1-s}(\Omega)^d}+ \norm{p}_{(H^s(\Omega)\cap L^2_0(\Omega))'}&\text{if }s\in[0,1),\, s\neq \frac12,\\
        \norm{\y}_{L^2(\Omega)^d}+\norm{p}_{(H^1(\Omega)\cap L^2_0(\Omega)\cap V_{-1}^{0,2}(\Omega))'}&\text{if } s=1,
    \end{cases}
\end{equation}
\[
    \norm{(\y,p)}_{\mathcal{ Y}_s'}
    :=
    \begin{cases}
        \norm{\y}_{H^{s-1}(\Omega)^d}+ \norm{p}_{H^s(\Omega)}&\text{if }s\in[0,1),\, s\neq \frac12,\\
        \norm{\y}_{L^2(\Omega)^d}+\norm{p}_{H^1(\Omega)}+\norm{p}_{V_{-1}^{0,2}(\Omega)}&\text{if } s=1,
    \end{cases}
\]
and 
\[
    \norm{(v,q)}_{\mathcal{V}_s}=
  \norm{v}_{H^1_0(\Omega)^d}+\norm{q}_{L^2(\Omega)} +\norm{(-\Delta v +\nabla q,\nabla\cdot v)}_{\mathcal{ Y}_s'}
\]
(the graph norm) for all $s\in[0,1]\setminus\{\frac12\}$. By the regularity result \eqref{eq:aprioridauge} with $s=0$ there holds 
\[
  \norm{v}_{H^1_0(\Omega)^d}+\norm{q}_{L^2(\Omega)}\le c \norm{(-\Delta v +\nabla q,\nabla\cdot v)}_{\mathcal{ Y}_0'}\le c \norm{(-\Delta v +\nabla q,\nabla\cdot v)}_{\mathcal{ Y}_s'}
\]
for all $(v,q)\in \mathcal{V}_s$ with $s\in[0,1]\setminus\{\frac12\}$. As a direct consequence we obtain
\begin{align}\label{eq:normVs}
  \norm{(v,q)}_{\mathcal{V}_s}\sim\norm{(-\Delta v +\nabla q,\nabla\cdot v)}_{\mathcal{ Y}_s'}
\end{align}
for all $(v,q)\in \mathcal{V}_s$ with $s\in[0,1]\setminus\{\frac12\}$. Hence, for all $(v,q)\in \mathcal{V}_s$ with $s\in[0,\xi)\cap [0,1]\setminus\{\frac12\}$ there holds according to~\eqref{eq:aprioridauge}
\begin{equation}
  \norm{v}_{H^{1+s}(\Omega)^d}+\norm{q}_{H^s(\Omega)}\le c \norm{(-\Delta v +\nabla q,\nabla\cdot v)}_{\mathcal{ Y}_s'}\le c \norm{(v,q)}_{\mathcal{V}_s}. \label{eq:apriori}
\end{equation}

\subsection{\label{sec:discr}Discretization}

We consider finite-dimensional subspaces $Y_h\subset H^1(\Omega)^d$ and $Q_h\subset Q = L^2_0(\Omega)$ with parameter $h>0$, and set
\begin{align*}
  Y_{0h}&:=\{v_h\in Y_h:v_h|_\Gamma=0\},\quad  Y_h^\partial:=Y_h|_\Gamma, \quad
  \tilde Y_h^\partial:=\{v_h\in Y_h^\partial: \langle v_h,n\rangle_\Gamma=0\},\\
  \tilde Q_h&:=Q_h\oplus\mathcal{P}_0(\Omega).
\end{align*}
The pair $(Y_{0h},Q_h)$ is assumed to satisfy the discrete inf-sup
condition,
\begin{align}\label{eq:inf-sup-discrete}
  \exists \alpha>0: \quad
  \sup_{v_h\in Y_{0h}\setminus\{0\}}\frac{(\nabla\cdot v_h,q_h)}{\|v_h\|_{H^1_0(\Omega)^d}}\ge 
  \alpha \|q_h\|_{L^2(\Omega)}\quad\forall q_h\in Q_h.
\end{align}
Concerning the approximation properties of the spaces $Y_h$ and $\tilde Q_h$ we assume that there is a number $k\in\mathbb{N}$ such that for all $r\in\R$, $r\ge0$, there holds
\begin{subequations}\label{eq:approxYhQh}
\begin{align}
  \inf_{v_h\in Y_h}\norm{v-v_h}_{H^1(\Omega)^{d}}&\le 
  ch^{\min(k,r)} \norm{v}_{H^{r+1}(\Omega)^d}&&\forall v\in H^{r+1}(\Omega)^d, \label{eq:approxYh} \\
  \inf_{\tilde q_h\in \tilde Q_h}\norm{q-\tilde q_h}_{L^2(\Omega)}&\le 
  ch^{\min(k,r)} \norm{q}_{H^r(\Omega)} &&\forall q\in H^r(\Omega). \label{eq:approxQh}
\end{align}
\end{subequations}
\begin{example}[value of $k$ for several elements]\label{ex:spaces}
We have $k=1$ for the MINI element and the Bernardi--Raugel element (SMALL element), as well as $k=2$ for the Taylor--Hood $\mathcal{P}_2/\mathcal{P}_1$ element. Moreover, the pairs $(Y_{0h},Q_h)$ constructed by these elements fulfill the discrete inf-sup condition, at least when quasi-uniform meshes are used. Further elements with corresponding $k$ are listed in \cite[Cor. 4.30]{John:16}.\qed
\end{example}
Furthermore, the approximating space $Y_h$ is assumed to satisfy a discrete extension theorem,
i.\,e., we assume that there is a constant $c_1>0$ such that
\begin{align}
    \inf_{\varphi_h\in Y_h\colon \varphi_h|_\Gamma=\u_h}\norm{\nabla \varphi_h}_{L^2(\Omega)^{d\times d}}&\le c_1 \norm{\u_h}_{H^{\frac12}(\Gamma)^d} \quad\forall \u_h\in Y_h^\partial. \label{eq:assum_boundary}
\end{align}

\begin{example}[Assumption \eqref{eq:assum_boundary}]\label{ex:assum_boundary}
Assumption \eqref{eq:assum_boundary} holds, for instance, for the finite elements considered in Example \ref{ex:spaces}, which can be seen by the following arguments: Let $I_h:H^1(\Omega)^d\to Y_h$ be the Scott-Zhang interpolation operator, which is stable in $H^1(\Omega)$ and preserves the Dirichlet boundary conditions, see \cite{ScottZhang:90}, and let $S:H^{\frac12}(\Gamma)^d\to H^1(\Omega)^d$ be the harmonic extension operator,
\begin{align*}
    (\nabla S\u,\nabla v)=0\quad\forall v\in H^1_0(\Omega)^d,\qquad S\u=\u\quad\text{on }\Gamma,
\end{align*}
which is well known to be linear and continuous. Thus, there holds
\begin{align*}
  \norm{\nabla I_hS\u_h}_{L^2(\Omega)^{d\times d}}\le c  \norm{S\u_h}_{H^1(\Omega)^{d}}\le c
  \norm{\u_h}_{H^{\frac12}(\Gamma)^d}\quad\forall\u_h\in Y_h^\partial.
\end{align*}\qed
\end{example}

\begin{remark}[Assumption \eqref{eq:assum_boundary}]\label{rem:assum_boundary}
    Using properties of the harmonic extension operator and the trace theorem one can directly show that
    \[
         \inf_{\varphi\in Y\colon \varphi|_\Gamma=\u}\norm{\nabla \varphi}_{L^2(\Omega)^{d\times d}}\le \norm{\nabla S\u}_{L^2(\Omega)^{d\times d} } \le c\norm{\u}_{H^{\frac12}(\Gamma)^d} \quad\forall \u\in Y|_\Gamma=H^{\frac12}(\Gamma)^d.
    \]
    On the discrete level we cannot infer this from the previous assumptions directly. \qed
\end{remark}

\begin{remark}[meshes]
Note that the assumptions \eqref{eq:inf-sup-discrete}, \eqref{eq:approxYhQh}, and \eqref{eq:assum_boundary} are written in a form which is independent of the finite element mesh. We consider quasi-uniform meshes in the examples in the current paper. To transfer the results to graded or anisotropic meshes, one just has to check the assumptions. 
This will be elaborated in a further paper. \qed
\end{remark}

As a consequence of \eqref{eq:assum_boundary} we can show a corresponding stability result for the discrete harmonic extension operator.
\begin{lemma}[stability of the discrete harmonic extension]\label{lem:discreteharmonicstabil}
    Let $S_h:Y_h^\partial\to Y_h$ be the discrete harmonic extension operator,
    \begin{align}\label{eq:discrharmextension}
    (\nabla S_h\u_h,\nabla v_h)=0\quad\forall v_h\in Y_{0h},\qquad S_h\u_h=\u_h\quad\text{on }\Gamma,
    \end{align}
and let Assumption \eqref{eq:assum_boundary} be satisfied. Then there holds
    \[
        \norm{\nabla S_h \u_h}_{L^2(\Omega)^{d\times d}}\le c_1\norm{\u_h}_{H^{\frac12}(\Gamma)^d}\quad \forall \u_h\in Y_h^\partial.
    \]
\end{lemma}

\begin{proof}
    Due to the definition of the discrete harmonic extension operator, there holds
    \[
        \norm{\nabla S_h \u_h}_{L^2(\Omega)^{d\times d}}^2
        =(\nabla S_h \u_h,\nabla \varphi_h)
        \le \norm{\nabla S_h \u_h}_{L^2(\Omega)^{d\times d}}\norm{\nabla \varphi_h}_{L^2(\Omega)^{d\times d}},
    \]
    where $\varphi_h$ is an arbitrary function from $Y_h$ fulfilling $\varphi_h|_\Gamma=\u_h$. Hence, the stability results of the assertion follows from \eqref{eq:assum_boundary}.
\end{proof}

Assumption \eqref{eq:assum_boundary} and as a consequence Lemma \ref{lem:discreteharmonicstabil} have special relevance for the approximation properties of the approximation space $Y_{0h}$. 
\begin{corollary}[best approximation in $Y_{0h}$]
    Let assumptions \eqref{eq:approxYh} and \eqref{eq:assum_boundary} be satisfied. There holds
  \begin{equation}\label{eq:approxYh0}
\begin{split}
\inf_{v_{0h}\in Y_{0h}}\norm{\nabla (v-v_{0h})}_{L^2(\Omega)^{d\times d}}
&\le c h^{\min(k,r)}\norm{v}_{H^{r+1}(\Omega)^d} \\
&\quad \forall v\in H^1_0(\Omega)^d\cap H^{r+1}(\Omega)^d.
\end{split}
\end{equation}
\end{corollary}
\begin{proof}
     Let $v\in H^1_0(\Omega)^d\cap H^{r+1}(\Omega)^d$ with $r$ from \eqref{eq:approxYh} and $v_h\in Y_h$ be arbitrary functions. Moreover, let $\tilde v_h$ be the discrete harmonic extension of $v_h|_\Gamma$ and let $\tilde v_{0h}\in Y_{0h}$ be chosen such that $v_h=\tilde v_h + \tilde v_{0h}$. There holds
    \[
        \norm{\nabla (v-\tilde v_{0h})}_{L^2(\Omega)^{d\times d}}\le \norm{\nabla (v-v_h)}_{L^2(\Omega)^{d\times d}}+\norm{\nabla \tilde v_h}_{L^2(\Omega)^{d\times d}}.
    \]
    According to Lemma \ref{lem:discreteharmonicstabil} the second term can be bounded by
    \[
        \norm{\nabla \tilde v_h}_{L^2(\Omega)}\le c\norm{v_h}_{H^{\frac12}(\Gamma)^d}\le c\norm{v-v_h}_{H^{\frac12}(\Gamma)^d}\le c \norm{v-v_h}_{H^1(\Omega)^d},
    \]
    where in the last two steps we used $v\in Y_0$ and the trace theorem. 
    Hence we have shown that for any $v_h\in Y_h$ there is a $\tilde v_{0h}\in Y_{0h}$ such that
    \[
        \norm{\nabla (v-\tilde v_{0h})}_{L^2(\Omega)^{d\times d}}\le c\norm{v-v_h}_{H^1(\Omega)^d}.
    \]
    As $v_h$ is arbitrarily chosen in $Y_h$, we get by using \eqref{eq:approxYh}
    \[
        \inf_{v_{0h}\in Y_{0h}}\norm{\nabla (v-v_{0h})}_{L^2(\Omega)^{d\times d}} \le \inf_{v_h\in Y_h}\norm{v-v_h}_{H^1(\Omega)^d}\le c h^{\min(k,r)}\norm{v}_{H^{r+1}(\Omega)^d},
    \]
    which is the assertion.
\end{proof}

In order to be able to treat the non-homogeneous Dirichlet boundary condition numerically, we approximate the function $\u$ by a
function $\u_h$ from the finite element space 
$Y_h^\partial=Y_h|_\Gamma$ 
 or $\tilde Y_h^\partial\subset Y_h^\partial$,
 see Subsection
  \ref{sec:reg_error_neu} for possibilities.  Set
\[ 
  Y_{*h}=\{v_h\in Y_h:v_h|_\Gamma=\u_h\}
\]
and define $(\y_h,p_h)\in Y_{*h}\times Q_h$ via
\begin{subequations}\label{eq:u_hp_h}
\begin{align}\label{eq:u_h}
  \nu(\nabla \y_h,\nabla v_h)-(\nabla\cdot v_h,p_h) &=0 \quad\forall v_h\in Y_{0h},\\
  \label{eq:p_h} (\nabla\cdot \y_h,q_h)&=0 \quad\forall q_h\in Q_h. 
\end{align}
\end{subequations}

\begin{remark}[compatibility condition]\label{rem:tilde}
  The solution of \eqref{eq:u_hp_h} is well defined for
  all $\u_h\in Y_h^\partial$; there is no need 
  that the boundary datum satisfies the compatibility condition $\langle\u_h,n\rangle_\Gamma=0$.
  (Use a homogenization argument as in the proof of
    Lemma \ref{lem:2.1} but in discrete spaces, and apply \cite[Prop.~8.2.1]{BoffiBrezziFortin:13}.)
  However, in implementations of the method, the space $Q_h$ is often replaced by $\tilde Q_h=Q_h\oplus\mathcal{P}_0(\Omega)$, see \cite[Rem. 4.70]{John:16} and Remark \ref{rem:implementation}. The normalization condition on the pressure is handled by a Lagrange multiplier, 
  and equation \eqref{eq:p_h} is demanded for $q_h\in\tilde Q_h$, i.\,e.\ additionally for $q_h\equiv1$. 
  In the case that $\u_h\in \tilde Y_h^\partial$, i.\,e., that
  the approximated boundary datum $\u_h$ satisfies $\langle\u_h,n\rangle_\Gamma=0$, 
  this additional equation is fulfilled.
  We get in particular $(\nabla\cdot\y_h,1)=0$ which may be of practical
  importance. If the compatibility condition 
  $\langle\u_h,n\rangle_\Gamma=0$ is not requested one has to replace
  \eqref{eq:p_h} by 
    \begin{align}\label{eq:delta_h}
      (\nabla\cdot \y_h,q_h)=(\delta_h,q_h) \quad\forall q_h\in \tilde Q_h, \quad \delta_h:=\abs{\Omega}^{-1}\langle\u_h,n\rangle_\Gamma,
    \end{align}
    for solvability of this augmented system.
    Note that the right hand side of \eqref{eq:delta_h} vanishes for $q_h\in Q_h$ such that we recover \eqref{eq:p_h}. For $q_h\equiv1$ we get the compatibility condition in the form $(\nabla\cdot\y_h,1)=\langle\u_h,n\rangle_\Gamma$ (Gauss divergence theorem). Finally, we note that for $\u$ satisfying the compatibility condition $\langle\u,n\rangle_\Gamma=0$ there holds
    \[
        \abs{\delta_h}=\abs{\Omega}^{-1}\abs{\langle\u_h-\u,n\rangle_\Gamma}\le c \norm{\u_h-\u}_{H^{-\frac12+\varepsilon}(\Gamma)}
    \]
    for any $\varepsilon>0$ small enough, as the normal vector $n$ is bounded in $H^{\frac12-\varepsilon}(\Gamma)$. Hence, for $\u_h\in Y_h^\partial$ we do not have $\delta_h=0$ in general, but $\delta_h$ is tending to zero for $h$ tending to zero, if $\u_h$ converges to $\u$ in $H^{-\frac12+\varepsilon}(\Gamma)$. Moreover, by a suitable choice of the approximations $\u_h$, one can derive an estimate for $\delta_h$, which is independent of $\varepsilon$, see Remark~\ref{rem:restriction}.
    \qed
\end{remark} 

\begin{remark}[implementational aspects]\label{rem:implementation}
    Replacing $Q_h$ by $\tilde Q_h$ for simpler implementation, as mentioned in Remark \ref{rem:tilde}, means that one has to find $(\y_h,p_h,\delta_h)\in Y_{*h}\times \tilde Q_h\times \R$ such that
    \begin{subequations}\label{eq:u_hp_hi}
    \begin{align}\label{eq:u_hi}
        \nu(\nabla \y_h,\nabla v_h)-(\nabla\cdot v_h,p_h) &=0 \quad\forall v_h\in Y_{0h}\\
        \label{eq:p_hi} (\nabla\cdot \y_h,q_h)-(\delta_h,q_h)&=0 \quad\forall q_h\in \tilde Q_h,\\
        (\mu_h,p_h) &=0 \quad\forall\mu_h\in\R.
    \end{align}
    \end{subequations}
    Since the boundary conditions $y_h=u_h$ are known (or computed in advance), one typically sets $y_h=y_{0h}+E_hu_h$, $E_h:Y_h^\partial\to Y_h$ being a discrete extension operator, such that the system \eqref{eq:u_hp_hi} can be reformulated as the symmetric formulation to find $(\y_{0h},p_h,\delta_h)\in Y_{0h}\times \tilde Q_h\times \R$ such that
    \begin{subequations}\label{eq:u_hp_hi0}
    \begin{align}\label{eq:u_hi0}
        \nu(\nabla \y_{0h},\nabla v_h)-(\nabla\cdot v_h,p_h) &=-\nu(\nabla E_hu_h,\nabla v_h) \quad\forall v_h\in Y_{0h}\\
        \label{eq:p_hi0} -(\nabla\cdot \y_{0h},q_h)+(\delta_h,q_h)&=(\nabla\cdot E_hu_h,q_h) \quad\forall q_h\in \tilde Q_h,\\
        (\mu_h,p_h) &=0 \quad\forall\mu_h\in\R.
    \end{align}
    \end{subequations}
    This system has the structure
    \begin{align*}
        \begin{pmatrix}
            0&\mathbf{o}^T&\mathbf{s}^T\\\mathbf{o}&\mathbf{A}&\mathbf{B}^T\\ \mathbf{s}&\mathbf{B}&\mathbf{O}
        \end{pmatrix}\begin{pmatrix}
            \mathbf{\delta}\\\mathbf{y}_0\\\mathbf{p}
        \end{pmatrix}=
        \begin{pmatrix}
            0\\\mathbf{f}\\\mathbf{g}
        \end{pmatrix}.
    \end{align*}
    Its matrix is symmetric and has saddle point structure with a singular upper left block
    $\left(\begin{smallmatrix}
        0&\mathbf{o}^T\\\mathbf{o}&\mathbf{A}
    \end{smallmatrix}\right)$.
    Note that $\delta=\delta_h$ can be computed independently via \eqref{eq:delta_h} or algebraically via $\mathbf{\delta}=\mathbf{e}^T(\mathbf{g}-\mathbf{B}\mathbf{y}_0)/\mathbf{e}^T\mathbf{s}$ where $\mathbf{e}$ is the vector of ones. Hence one could modify the first equation such that the system is
    \begin{align*}
        \begin{pmatrix}
            \alpha&\mathbf{o}^T&\mathbf{s}^T\\\mathbf{o}&\mathbf{A}&\mathbf{B}^T\\ \mathbf{s}&\mathbf{B}&\mathbf{O}
        \end{pmatrix}\begin{pmatrix}
            \mathbf{\delta}\\\mathbf{y}_0\\\mathbf{p}
        \end{pmatrix}=
        \begin{pmatrix}
            \alpha\delta\\\mathbf{f}\\\mathbf{g}
        \end{pmatrix}
    \end{align*}
    with a positive definite upper left block. The positive real number $\alpha$ should be chosen appropriately.
    \qed
\end{remark}

We analyze the discretization error of this method in
  Section \ref{sec:NumerikWeak} for the case  $\y \in H^{t+\frac{1}{2}}(\Omega)$,  $ t \ge\frac12$, where a
  weak solution of the problem exists. For  $t\in[0,\frac12)$, we will introduce in Section
\ref{sec:veryweakformulation} a very weak
solution $(\y,p)$ and analyze the numerical solution for this case in
Section \ref{sec:NAveryweak}. This analysis is different from
that of Section \ref{sec:NumerikWeak} since the concept of a weak solution cannot be
used, hence the error can only be estimated in very weak norms. Since the
approximation of the boundary datum in
$Y_h^\partial$ is common for both the weak and the very weak
solution we discuss this in the next subsection.

\subsection{\label{sec:reg_error_neu}Regularization of the boundary datum} 
Bartels, Carstensen, and Dolzmann considered the Poisson problem with nonhomogeneous Dirichlet boundary data and encountered the same task of approximating $\u$ in $Y_h^\partial$, in that case a scalar valued $\u$, \cite{BCD:04}. To this end, they analyzed the
$L^2(\Gamma)$-projection and the Lagrange interpolation, each based on the $\mathcal{P}_1$ element, for constructing $\u_h=\Ih_h\u\in Y_h^\partial$ and subsequently derived optimal finite element error estimates in $L^2(\Omega)$ and $H^1(\Omega)$.
Apel, Nicaise, and Pfefferer treated the same problem in \cite{ApelNicaisePfefferer:16}, however, with focus on
non-smooth Dirichlet boundary data, which do not belong to the trace space of $H^1(\Omega)$. For approximating $\u$ by $\u_h=\Ih_h\u$ in $Y_h^\partial$, they analyzed the $L^2(\Gamma)$-projection and the Carstensen interpolant
\cite{Yserentant:90,carstensen:99}, each based on the $\mathcal{P}_1$ element, and showed optimal finite element error estimates in $L^2(\Omega)$. Estimates in $H^1(\Omega)$ cannot be expected in this case due to the non-smoothness of the Dirichlet boundary data. 

Finite element error estimates in $L^2(\Omega)$ and $H^1(\Omega)$ generally require different properties of the approximations $\u_h=\pi_h \u$. For instance, estimates in $H^1(\Omega)$ can be achieved if the approximation error $\u-\u_h$ can be appropriately bounded in $H^{1/2}(\Gamma)$, whereas the best estimates in $L^2(\Omega)$ can be obtained if estimates for the approximation error are additionally available in $H^{s}(\Gamma)$ with $s\le0$ sufficiently small, for details see below.

The idea of approximating $\u$ by  $\u_h = \Ih_h \u\in Y_h^\partial$,
can also be followed in the Stokes case
component-wise. In order to obtain convergence orders for the approximation error, we require a certain regularity for $\u$. Typically, this regularity is measured in the scale of fractional order Sobolev spaces $H^t(\Gamma)$ on the boundary. However, the definition of these spaces is limited to the range $t\in[0,1]$ due to the fact that the boundary is only of class $\mathcal{C}^{0,1}$. For that reason, we introduce the space $\tilde H^t(\Gamma)$ by
\[
    \tilde H^t(\Gamma)=
    \begin{cases}
        H^t(\Gamma) & \text{if }t\in[0,1],\\
        C(\Gamma)\cap \Pi_{j=1}^d H^t(\Gamma_j) & \text{if }t>1,
    \end{cases}
\]
where $\Gamma_j$, $j=1,\ldots,N$, are the plane facets of $\Gamma=\bigcup_{j=1}^N \bar \Gamma_j$. A norm in this space is given by
\[
    \norm{v}_{\tilde H^t(\Gamma)}:=\sum_{j=1}^N \norm{v}_{H^t(\Gamma_j)}.
\]
For large $t$ this space might be larger than the trace space of $H^{t+\frac12}(\Omega)$, i.\,e., the extension $Ev$ of $v\in \tilde H^t(\Gamma)$ might not exist in $H^{t+\frac12}(\Omega)$, but it is sufficient for our needs. 
Forthcoming, we assume that there are intervals $I_s$ and $I_t$ satisfying $[0,\tfrac12]\subset I_s\subset [-1,1]$ and $I_t\subset[0,\infty)$, as well as an operator $\Ih_h: \tilde H^{t}(\Gamma)^d\to Y_h^\partial$ with $t\in I_t$ such that 
\begin{equation}\label{eq:assumptionregularizer}
    \norm{\u-\Ih_h\u}_{H^{s}(\Gamma)^d}\le ch^{t-s}\|\u\|_{\tilde H^t(\Gamma)^d}
\end{equation}
holds for all $s\in I_s$ and $t\in I_t \cap [s,\infty)$.
Examples for the operator $\Ih_h$, which we consider here, are the $L^2(\Gamma)$-projection, the Carstensen interpolation operator \cite{Yserentant:90,carstensen:99} or the Lagrange interpolation operator, each with different approximation properties, i.\,e.\ intervals $I_s$ and $I_t$, depending on the underlying space $Y_h^\partial$, see below.
In order to show estimate \eqref{eq:assumptionregularizer} for $s\le0$ one has to show the inequality
\begin{align}
  \langle \u-\Ih_h\u,\varphi\rangle_\Gamma&\le ch^{t-s}
  \|\u\|_{\tilde H^t(\Gamma)^d}\|\varphi\|_{H^{-s}(\Gamma)^d}.
  \label{eq:regerrortested}
\end{align}
Note that, for $s<0$, this estimate is typically based on orthogonality properties of $\pi_h$, which are satisfied for the $L^2(\Gamma)$-projection and the Carstensen interpolant but in general not for other approximants like the
Lagrange, Cl\'ement or the Scott--Zhang interpolant. However, Tantardini and Veeser constructed a modification of the Scott--Zhang interpolant satisfying \eqref{eq:regerrortested} in domains, see \cite[Section 5]{TantardiniVeeser:16}.

\begin{example}[$I_s$ and $I_t$ for \eqref{eq:assumptionregularizer}]\label{ex:prox}
We exemplarily discuss estimate \eqref{eq:assumptionregularizer} for the $L^2(\Gamma)$-projection, the Carstensen interpolation operator \cite{Yserentant:90,carstensen:99} and the Lagrange interpolation operator as elements of $Y_h^\partial$ being the trace space of $Y_h$ based on the MINI element and the Bernardi--Raugel element (SMALL element), as well as the Taylor--Hood $\mathcal{P}_2/\mathcal{P}_1$ element. For simplicity, we assume that the underlying triangulation is quasi-uniform. 
For the MINI element, the space of piecewise linear and continuous functions over the triangulation of the boundary coincides with $Y_h^\partial$, while for the Bernardi--Raugel element, it is a subspace of $Y_h^\partial$. 
As a consequence, we obtain that \eqref{eq:assumptionregularizer} holds for
\begin{itemize}
  \item $I_s=[-1,1]$, $I_t=[0,2]$ in case of the $L^2(\Gamma)$-projection,
  \item $I_s=[-1,1]$, $I_t=[0,1]$ in case of the Carstensen interpolant, or
  \item $I_s=[0,1]$, $I_t=(\frac12(d-1), 2]$ in case of the Lagrange interpolant.
\end{itemize}
For the Taylor--Hood $\mathcal{P}_2/\mathcal{P}_1$ element, the space of piecewise quadratic and continuous functions over the triangulation of the boundary coincides with $Y_h^\partial$. Hence, \eqref{eq:assumptionregularizer} holds for
\begin{itemize}
  \item $I_s=[-1,1]$, $I_t=[0,3]$ in case of the $L^2(\Gamma)$-projection,
  \item $I_s=[-1,1]$, $I_t=[0,1]$ in case of the Carstensen interpolant, or
  \item $I_s=[0,1]$, $I_t=(\frac12(d-1), 3]$ in case of the Lagrange interpolant.
\end{itemize}
In case that $s\in[-1,0]$ the corresponding proofs for the $L^2(\Gamma)$-projection and the Carstensen interpolant, respectively, follow the same lines as the one in \cite[Lemma 2.14]{ApelNicaisePfefferer:16}, see also \eqref{eq:regerrortested} and the subsequent explanations. In case that $s\in[0,1]$ one first shows estimates in $L^2(\Gamma)$ and $H^1(\Gamma)$ in an appropriate way.
These estimates are well known for quasi-uniform meshes. The estimate in $H^s(\Gamma)$ then follows by real interpolation in Sobolev spaces. The estimate for the Lagrange interpolant is a standard one using real interpolation in Sobolev spaces. \qed
\end{example}

A property of the approximation $\Ih_h \u$ is that in general
\[
  \langle \Ih_h\u,n\rangle_\Gamma\not=0
\]
even if $\langle\u,n\rangle_\Gamma=0$. Examples can be constructed easily, see Example \ref{bsp:counterexampleDGL} in the appendix.
Our remedy is the enforcement of 
\begin{equation}\label{eq:intghn=0}
  \langle\u_h,n\rangle_\Gamma=0,
\end{equation}
by correcting the map
$\Ih_h: \tilde H^t(\Gamma)^d\to Y_h^\partial$ to become $\tilde \Ih_h:\tilde H^t(\Gamma)^d\to\tilde Y_h^\partial$. The modification uses a function $w_h\in Y_h^\partial$, where we only assume for the moment that
\begin{equation*}
    \langle w_h,n\rangle_\Gamma\neq0.
\end{equation*}
It is defined by
\begin{align}\label{eq:lambda_h1}
  \tilde \Ih_h \u=\Ih_h \u-\lambda_h w_h \quad\text{with}\quad \lambda_h=
  \frac{\langle \Ih_h\u,n\rangle_\Gamma}{\langle w_h,n\rangle_\Gamma}.
\end{align}
Note that the choice of $\lambda_h$ ensures
  $\langle \tilde\Ih_h\u,n\rangle_\Gamma=0$ such that $\tilde
  \Ih_h\u\in\tilde Y_h^\partial$.
The function $w_h$ can be chosen, e.\,g., as the $L^2(\Gamma)$-projection or the Carstensen interpolant of the normal vector $n$ on a component-wise basis, see Remark \ref{rem:approximationvariants}. The Lagrange interpolant of $n$ is not suited since the normal vector exhibits jumps in corners in 2D and in edges in 3D, respectively, of the domain. However, other choices are possible as well, where $w_h$ does not approach $n$ at all. For instance, in two space dimensions one can choose $w_h=\y_0|_\Gamma\in Y_h^\partial$ with $\y_0=\frac12\binom{x_1-\bar x_1}{x_2-\bar x_2}$ and
appropriate $(\bar x_1,\bar x_2)^T\in\R^2$.

\begin{remark}[$L^2(\Gamma)$-projection onto $\tilde Y_h^\partial$]
  If $\Ih_h$ is the $L^2(\Gamma)$-projection onto $Y_h^\partial$ and
  $w_h=\Ih_h n$, then $\tilde \Ih_h$ is exactly the $L^2(\Gamma)$-projection
  onto $\tilde Y_h^\partial$. To see this, one has to confirm that $\tilde \Ih_h u\in \tilde Y_h^\partial$ and that $(\u-\tilde \Ih_h\u,v_h)_\Gamma=0$ for all $v_h\in \tilde
  Y_h^\partial$. The former is true due to the choice of $\lambda_h$. The latter follows from
  \begin{align*}
    (\u-\tilde \Ih_h\u,v_h)_\Gamma=(\u-\Ih_h \u,v_h)_\Gamma+\lambda_h(\Ih_h n,v_h)_\Gamma,
  \end{align*}
  where the first term vanishes due to the projection property of $\Ih_h$. The second term vanishes
  since we have for any $v_h\in \tilde
  Y_h^\partial\subset Y_h^\partial$ that  $(\Ih_h n,v_h)_\Gamma=(n,v_h)_\Gamma=0$. \qed
\end{remark}

\begin{remark}[approximation variants] \label{rem:approximationvariants}
  Let us mention approaches in the literature:
Hamouda, Temam, and Zhang, \cite{HamoudaTemamZhang:17}, just use the normal vector $n$ itself for $w_h$, which works since they consider domains with $C^2$ boundary
  only, otherwise $\tilde \Ih_h \u$ would not be a continuous function.

Zhou and Gong, \cite{ZhouGong:22}, use the
  $L^2(\Gamma)$-projection for $\Ih_h$ and set $w_h=\Ih_h n$. They show 
  \[
    \|\u-\tilde \Ih_h \u\|_{L^2(\Gamma)^d} \le C h^t \|\u\|_{\tilde H^t(\Gamma)^d}, \quad
    t\in[0,\tfrac32].
  \]
 
Duran, Gastaldi, and Lombardi, \cite{DuranGastaldiLombardi:20}, use
  in Section 2 the $L^2(\Gamma)$-projection and the
  Carstensen interpolant for $\Ih_h$ and claim that
    $\langle \Ih_h\u,n\rangle_\Gamma=0$ which is not true, see Example
    \ref{bsp:counterexampleDGL} in the Appendix.  In
    Section 3, they define $\tilde \Ih_h$ by a local correction
  (only on two edges of size $O(h)$) of a variant of
    the Lagrange interpolation operator for piecewise smooth
    functions. The
  advantage of the locality competes with the fact that $\tilde \Ih_h\u$
  is uniformly bounded in $H^s(\Gamma)^d$ for $s<\frac12$ only.
  In the analysis of the problem with regularized
    boundary datum this leads to a logarithmic factor which is sharp
    for problems with piecewise smooth boundary datum.

    Heister, Rebholz and Xiao, \cite{HeisterRebholzXiao:16}, assume continuous boundary data such that Lagrange interpolation can be used. In order to satisfy the compatibility condition, they modify the nodal interpolant at one interior node per facet and analyze the interpolation error.  Eickmann, Scott and Tscherpel, \cite{EickmannScottTscherpel:25}, argue that the modification in one single boundary facet is enough, and apply an idea similar to \eqref{eq:lambda_h1} by using $w_h=b_fn_f$ with $b_f$ being the face bubble function of one boundary facet with normal $n_f$.
    
    Jessberger and Kaltenbach, \cite{JessbergerKaltenbach:24}, enforce the compatibility condition by extending the boundary data to the interior, interpolate it and restrict it to the boundary. The interpolation is chosen such that the divergence is preserved in the $\tilde Q_h'$ sense.
\qed 
\end{remark}

\begin{lemma}[error estimates for $\tilde\Ih_h$] \label{lem:tildeI_h}
Let the operator $\Ih_h$ satisfy \eqref{eq:assumptionregularizer} for $s\in I_s\cap(-\frac12,1]$ and $t\in I_t\cap [s,\infty)$. Moreover, let the operator $\tilde\Ih_h$ be defined by \eqref{eq:lambda_h1}, where we additionally assume that there is an $s'\in I_s\cap (-\tfrac12,s]$ such that
\begin{equation}\label{eq:assumptionPh}
   \frac{\norm{w_h}_{H^{s}(\Gamma)^d}}{\left|\langle w_h,n\rangle_\Gamma\right|}\le c h^{s'-s}.
\end{equation}
  Then the estimate 
  \begin{align}\label{eq:assumptionregularizermod}
    \norm{\u-\tilde \Ih_h\u}_{H^{s}(\Gamma)^d}\le
    ch^{t-s}\|\u\|_{\tilde H^t(\Gamma)^d}
  \end{align}
  holds for all $\u\in\tilde H^t(\Gamma)^d$ satisfying $\langle u,n\rangle_{\Gamma}=0$.
\end{lemma}

\begin{proof}
Using
\begin{equation}\label{eq:tildehelp1}
    \langle \Ih_h\u,n\rangle_\Gamma=
    \langle \Ih_h\u-\u, n\rangle_\Gamma,
\end{equation}
\eqref{eq:assumptionPh}, and the boundedness of $n$ in $H^{-s'}(\Gamma)^d$, we obtain
\begin{align*}
    \norm{\u-\tilde \Ih_h\u}_{H^{s}(\Gamma)^d}&\le \norm{\u-\Ih_h\u}_{H^{s}(\Gamma)^d}+\frac{|\langle \Ih_h\u-\u, n\rangle_\Gamma|}{|\langle w_h,n\rangle|}\norm{w_h}_{H^{s}(\Gamma)^d}\\
    &\le \norm{\u-\Ih_h\u}_{H^{s}(\Gamma)^d}+ch^{s'-s}|\langle \Ih_h\u-\u, n\rangle_\Gamma| \\
    &\le \norm{\u- \Ih_h\u}_{H^{s}(\Gamma)^d}+ch^{s'-s}\norm{n}_{H^{-s'}(\Gamma)^d}\norm{\u- \Ih_h\u}_{H^{s'}(\Gamma)^d}\\
    &\le \norm{\u- \Ih_h\u}_{H^{s}(\Gamma)^d}+ch^{s'-s}\norm{\u- \Ih_h\u}_{H^{s'}(\Gamma)^d}.
\end{align*}
With \eqref{eq:assumptionregularizer} we deduce \eqref{eq:assumptionregularizermod}.
\end{proof}

\begin{example}[discussion of assumption \eqref{eq:assumptionPh}]\label{ex:proxmod}
Provided that \eqref{eq:assumptionPh} is satisfied, the approximation quality of $\pi_h$ and $\tilde \pi_h$ coincide for $s\in I_s\cap(-\frac12,1]$ and $t\in I_t\cap [s,\infty)$, which is enough for our purposes. We exemplarily discuss different choices of $w_h$. 

If $w_h=\y_0|_\Gamma\in Y_h^\partial$ with $\y_0=\frac12\binom{x_1-\bar x_1}{x_2-\bar x_2}$ and appropriate $(\bar x_1,\bar x_2)\in\R^2$, we get by direct calculations
\[
    \frac{\norm{w_h}_{H^{s}(\Gamma)}}{\left|\langle w_h,n\rangle_\Gamma\right|}\le c.
\]
Consequently, \eqref{eq:assumptionPh} is satisfied, e.\,g., for $s'=s$. 

As a second example, let us consider $w_h$ being the $L^2(\Gamma)$-projection of the normal vector~$n$ using the elements from before on quasi-uniform meshes. By using approximation properties of the $L^2(\Gamma)$-projection and real interpolation in Sobolev spaces, there holds
\[
    \frac{\norm{w_h}_{H^{s}(\Gamma)^d}}{\left|\langle w_h,n\rangle_\Gamma\right|} =\frac{\norm{w_h}_{H^{s}(\Gamma)^d}}{\norm{w_h}_{L^2(\Gamma)^d}^2}\le c\frac{\norm{w_h}_{H^{\max(s,0)}(\Gamma)^d}}{\norm{w_h}_{L^2(\Gamma)^d}^2}\le c\frac{\norm{n}_{H^{\max(s,0)}(\Gamma)^d}}{\norm{w_h}_{L^2(\Gamma)^d}^2}\le c
\]
for $h$ small enough and for all $s<1/2$, as the normal vector $n$ is bounded in $H^s(\Gamma)$ for $s<1/2$. Hence, if $s<1/2$, \eqref{eq:assumptionPh} is satisfied, e.\,g., for $s'=s$. For $s\ge1/2>s'$ we get by means of an inverse inequality and the foregoing estimate
\[
    \frac{\norm{w_h}_{H^{s}(\Gamma)}}{\left|\langle w_h,n\rangle_\Gamma\right|} \le ch^{s'-s}\frac{\norm{w_h}_{H^{s'}(\Gamma)}}{\left|\langle w_h,n\rangle_\Gamma\right|}\le c h^{s'-s}.
\]
As a consequence, if $s\ge1/2$, \eqref{eq:assumptionPh} is satisfied with $s'<1/2$. \qed
\end{example}

\begin{remark}[on the restriction $s'>-\frac12$]\label{rem:restriction}
    The restriction $s'>-\frac12$ in Lemma~\ref{lem:tildeI_h} comes from the fact that the normal vector $n$ does not belong to $H^{\frac12}(\Gamma)^d$ but to $H^{\frac12-\varepsilon}(\Gamma)^d$ for any $\varepsilon>0$. However, in certain situations it is possible to improve the result of Lemma~\ref{lem:tildeI_h} to $s'=-\frac12$.
    As an example, let $\pi_h$ be the $L^2(\Gamma)$-projection as considered in Example~\ref{ex:proxmod}. Moreover, let $P_h$ denote the Carstensen interpolation operator. Then, by using
    \begin{align*}
        \langle \Ih_h\u-\u, n\rangle_\Gamma &= \langle \Ih_h\u-\u, n-P_h n\rangle_\Gamma\\
        &\le \norm{\Ih_h\u-\u}_{L^2(\Gamma)}\norm{n-P_h n}_{L^2(\Gamma)}\\
        &\le c h^t \|\u\|_{\tilde H^t(\Gamma)^d} \norm{n-P_h n}_{L^2(\Gamma)}
    \end{align*}
    we obtain, instead of \eqref{eq:assumptionregularizermod},
    \[
        \norm{\u-\tilde \Ih_h\u}_{H^{s}(\Gamma)^d}\le
        c\left(h^{t-s}\|\u\|_{\tilde H^t(\Gamma)^d} + 
        ch^{t-s+s'}\|\u\|_{\tilde H^t(\Gamma)^d}\norm{n-P_h n}_{L^2(\Gamma)}\right).
    \]
    By direct calculations, we find that the Carstensen interpolant of $n$ approximates $n$ with order $\tfrac12$, i.\,e.,
    \begin{equation}\label{eq:Carstensen_h12}
        \norm{n-P_h n}_{L^2(\Gamma)} \le c h^{\frac12}.
    \end{equation}
    This is due to the fact that the normal vector is piecewise constant. As a result, the normal vector and its Carstensen interpolant coincide except in an $O(h)$-neighborhood of the corners in 2D, or of the edges in 3D, respectively. Hence, we have shown the result of Lemma~\ref{lem:tildeI_h} with $s'=-\frac12$. Of course, then we also require \eqref{eq:assumptionPh} for $s'=-\frac12$, which is however possible according to Example~\ref{ex:proxmod}. Finally, let us remark that a similiar discussion is also possible for $\pi_h$ being the Carstensen interpolant.    \qed
\end{remark}

\section{\label{sec:NumerikWeak}Numerical analysis for the weak solution}

\subsection{\label{sec:3.1}Main results}

Consider the weak solution $(\y,p)\in Y_*\times Q$ according to \eqref{eq:weakneu} and the discrete solution $(\y_h,p_h)\in Y_{*h}\times Q_h$ according to \eqref{eq:u_hp_h}. Moreover, let $(\bar\y,\bar p)\in Y_*\times Q$ be the weak solution according to \eqref{eq:weakneu} with $\nu=1$. As already discussed in Remark~\ref{rem:vissol} there holds
\begin{equation}\label{eq:vissol}
    (y, p)=(\bar y, \nu\bar p).
\end{equation}
In this section we derive approximation error estimates between $(\y,p)$ and $(\y_h,p_h)$ in the energy norm and in weaker norms. Due to \eqref{eq:vissol}, we will obtain that all estimates are pressure robust. In addition,
we formulate the results in a way that neither $\langle\u,n\rangle_\Gamma=0$ nor $\langle\u_h,n\rangle_\Gamma=0$ are required. Nevertheless, we comment in Remark~\ref{rem:approx-tilde} on the respective advantages and disadvantages of whether the compatibility conditions are fulfilled or not.

\begin{theorem}[best approximation in $Y_h$ and $\tilde Q_h$]\label{th:bestapprox}
  Let the assumptions \eqref{eq:inf-sup-discrete} and \eqref{eq:assum_boundary} be satisfied. Then the discretization error satisfies
  \begin{align}\label{eq:bestapproximationyneu} \nonumber
    \nu\norm{\nabla(\y&-\y_h)}_{L^2(\Omega)^{d\times d}} + \norm{p-p_h}_{L^2(\Omega)} \\
    &\le c\nu\left(
    \norm{\bar \y-v_h}_{H^1(\Omega)^{d}} + \norm{\u-\u_h}_{H^{\frac12}(\Gamma)^d} + 
    \norm{\bar p-\tilde q_h}_{L^2(\Omega)} \right) 
  \end{align}
  for all $v_h\in Y_{h}$ and all $\tilde q_h\in \tilde Q_h$. If in addition
  \[
    \nabla \cdot v_h\equiv 0 \quad \forall v_h\in Y_{0h}^\div:=\{v_h^\div\in Y_{0h}:(\nabla\cdot v_h^\div,q_h)=0\ \forall q_h\in Q_h \}
  \]
  then there even holds
  \begin{equation}\label{eq:bestapproximationyneuplus}
    \norm{\nabla(\y-\y_h)}_{L^2(\Omega)^{d\times d}} 
    \le c\left(
    \norm{\bar \y-v_h}_{H^1(\Omega)^{d}} + \norm{\u-\u_h}_{H^{\frac12}(\Gamma)^d}
    \right)
  \end{equation}
  for all $v_h\in Y_{h}$.
\end{theorem}

The proof of this result is postponed to Subsection \ref{sec:proofapprox}, where even the dependence on the inf-sup constant $\alpha$, see \eqref{eq:inf-sup-discrete}, is traced. Note also that we do not need assumption~\eqref{eq:approxYhQh}.

\begin{remark}[pressure robustness in the energy norm]\label{rem:pressurerobust}
    In particular, we get from \eqref{eq:bestapproximationyneu} that the estimate
    \[
        \norm{\nabla(\y-\y_h)}_{L^2(\Omega)^{d\times d}}
    \le c\left(
    \norm{\bar \y-v_h}_{H^1(\Omega)^{d}} + \norm{\u-\u_h}_{H^{\frac12}(\Gamma)^d} + 
    \norm{\bar p-\tilde q_h}_{L^2(\Omega)} \right)
    \]
    is satisfied for all $v_h\in Y_{h}$ and all $\tilde q_h\in \tilde Q_h$, which does not depend on the viscosity parameter $\nu$. Hence, the estimate is pressure robust independent of a particular choice of the finite elements. If we use finite elements, such as the Scott--Vogelius element, which imply $\nabla \cdot v_h\equiv 0$ for all $v_h\in Y_{0h}^\div$, we additionally get that the estimate for the velocity is independent of the pressure, see \eqref{eq:bestapproximationyneuplus}. However, let us again emphasize that the pressure robustness is valid for all finite elements, which satisfy assumptions \eqref{eq:inf-sup-discrete} and \eqref{eq:assum_boundary}. \qed
\end{remark}

\begin{remark}[pressure robustness for the general Stokes problem]
    If we consider the general Stokes problem
    \begin{align*}
    -\nu\Delta \y +\nabla p &= f\quad\text{in }\Omega,\\
    \nabla \cdot \y &=g\quad\text{in }\Omega,\\
    \y&=\u\quad\text{on }\Gamma
  \end{align*}
  with non-homogeneous differential equations, pressure robust estimates cannot be expected for all finite elements and, hence, depend on the choice of the finite elements. However, due to the linearity of the differential operator, we can always split the problem into two. The first problem has exactly the structure as considered here (see \eqref{eq:stokesclassical}), where we obtain pressure robust estimates independent of the choice of the finite elements. The second problem possesses homogeneous Dirichlet boundary conditions and non-homogeneous differential equations. It requires special finite elements, such as the Scott--Vogelius element, or adapted methods as in \cite{LedererLinkeMerdonSchoeberl:17} to ensure pressure robustness. However, this has extensively been considered in the literature. \qed
\end{remark}

For deriving the concrete approximation orders, we will rely on natural regularity assumptions for $\bar \y$, $\bar p$ and $\u$.
A study of which regularity can be expected in dependence of the data, is beyond the scope of the paper.

\begin{theorem}[error estimates in the energy norm]\label{th:approxweak}
  Let the assumptions \eqref{eq:inf-sup-discrete}, \eqref{eq:approxYhQh} and \eqref{eq:assum_boundary} be satisfied. Assume that $\bar\y\in H^{t+\frac12}(\Omega)^d$, $\bar p\in
  H^{t-\frac12}(\Omega)$, and $\u\in \tilde H^t(\Gamma)$ with  some $t\ge\frac12$. If the approximate boundary condition satisfies
  \begin{align}\label{eq:xxx4}
    \norm{\u-\u_h}_{H^{\frac12}(\Gamma)^d} \le
    ch^{\min(t-\frac12,k)} \norm{\u}_{\tilde H^t(\Gamma)^d}
  \end{align}
  with $k$ from Subsection \ref{sec:discr}
  then the error estimate
  \begin{align}\notag
    \nu\norm{\nabla(\y-\y_h)}_{L^2(\Omega)^{d\times d}}&+\norm{p-p_h}_{L^2(\Omega)} \\ &\le \label{eq:est_energynorm}
    c\nu h^{\min(t-\frac12,k)} \left(\norm{\bar\y}_{H^{t+\frac12}(\Omega)^d}+
    \norm{\bar p}_{H^{t-\frac12}(\Omega)}+\norm{\u}_{\tilde H^t(\Gamma)^d}\right)  \end{align}
      holds. If $\u-\u_h\equiv 0$ then the estimate does not explicitly depend on $\norm{\u}_{\tilde H^t(\Gamma)^d}$.
\end{theorem}

\begin{proof}
The assertion follows from Theorem \ref{th:bestapprox} in combination with assumptions \eqref{eq:approxYh}, \eqref{eq:approxQh} and \eqref{eq:xxx4}. If $\u-\u_h\equiv 0$ then $\norm{\u-\u_h}_{H^{\frac12}(\Gamma)^d}=0$ and there is no need to estimate this term further.
\end{proof}

\begin{example}[approximation orders in the energy norm]\label{ex:approxH1}
 In Example \ref{ex:spaces} and Example \ref{ex:assum_boundary}, we considered the MINI element and the SMALL element which correspond to $k=1$, as well as the Taylor--Hood $\mathcal{P}_2/\mathcal{P}_1$ element which corresponds to $k=2$. There we observed that \eqref{eq:inf-sup-discrete}, \eqref{eq:approxYhQh} and \eqref{eq:assum_boundary} are satisfied for these elements. It remains to show \eqref{eq:xxx4}. For simplicity, we assume that the underlying triangulation is quasi-uniform. If we then construct $u_h$, e.\,g., by means of the $L^2(\Gamma)^d$-projection, the Lagrange interpolation or by means of their appropriately modified versions according to \eqref{eq:lambda_h1}, we get that \eqref{eq:xxx4} is satisfied. More precisely, we can even show
 \[
    \norm{\u-\u_h}_{H^{\frac12}(\Gamma)^d} \le
    ch^{\min(t-\frac12,k+\frac12)} \norm{\u}_{\tilde H^t(\Gamma)^d}\le ch^{\min(t-\frac12,k)} \norm{\u}_{\tilde H^t(\Gamma)^d}
 \]
 for all $t\ge\tfrac12$ in case of the (normal or modified) $L^2(\Gamma)^d$-projection and for all $t>\frac12(d-1)$ in case of the (normal or modified) Lagrange interpolation, respectively, see Example~\ref{ex:prox} and Example~\ref{ex:proxmod}. Of course, for the modfied versions, we require $\langle u,n\rangle_{\Gamma}=0$, see Lemma~\ref{lem:tildeI_h}. Hence, we get the order $k$ if $t=k+\frac12$, that is if $\bar \y\in H^{k+1}(\Omega)^d$, $\bar p\in H^k(\Omega)$, and $\u\in \tilde H^{k+\frac12}(\Gamma)^d$. Note that the Carstensen interpolant is not interesting in this case since the approximation order in $H^{\frac12}(\Gamma)^d$ is only less than or equal to $\frac12$ by construction.\qed
\end{example}
\begin{remark}[approximation orders of the modified version]\label{rem:approx-tilde}
      As already discussed in Remark \ref{rem:tilde} and Example \ref{ex:approxH1}, we again highlight that using the modified version $\tilde \Ih_h$, i.\,e., choosing $\u_h = \tilde \Ih_h \u $ instead of $\u_h = \Ih_h \u $, does not affect the approximation orders from above. This modification requires additional effort in the construction, and it has the advantage that the discrete velocity satisfies $\langle\nabla\cdot\y_h,1\rangle=0$, since the modified approximation satisfies $\langle \u_h,n\rangle_{\Gamma} = 0 $ by construction. On the contrary, if this property is not requested, one can directly proceed with $u_h=\Ih_h u$ and obtains $\langle \u_h,n\rangle_{\Gamma} =\langle\nabla\cdot\y_h,1\rangle= 0 $ in the limit $h\to0$ provided that $\langle \u,n\rangle_{\Gamma} = 0 $, see the discussion of Remark \ref{rem:tilde}. \qed
\end{remark}

In some cases, depending on the approximation $u_h$, we can get rid of the data approximation term in \eqref{eq:bestapproximationyneu}.

\begin{theorem}[best approximation independent of data error]\label{th:bestapproxindependent}
    Let the assumptions \eqref{eq:inf-sup-discrete} and \eqref{eq:assum_boundary} be satisfied. If the approximate boundary condition satisfies
  \begin{align}\label{eq:bestassum}
    \norm{\u-\u_h}_{H^{\frac12}(\Gamma)^d} \le c \norm{\u-v_h}_{H^{\frac12}(\Gamma)^d}
  \end{align}
  for all $v_h\in Y_h^\partial$ then the error estimate
    \begin{align*}
    \nu\norm{\nabla(\y&-\y_h)}_{L^2(\Omega)^{d\times d}} + \norm{p-p_h}_{L^2(\Omega)} 
    \le c\nu\left(
    \norm{\bar \y-v_h}_{H^1(\Omega)^{d}} + 
    \norm{\bar p-\tilde q_h}_{L^2(\Omega)} \right) 
  \end{align*}
  is fulfilled for all $v_h\in Y_{h}$ and all $\tilde q_h\in \tilde Q_h$. If there additionally holds \eqref{eq:approxYhQh}, $\bar\y\in H^{t+\frac12}(\Omega)^d$ and $\bar p\in
  H^{t-\frac12}(\Omega)$ with  some $t\ge\frac12$
  then there is the error estimate
  \begin{align}
    \nu\norm{\nabla(\y-\y_h)}_{L^2(\Omega)^{d\times d}}&+\norm{p-p_h}_{L^2(\Omega)} \nonumber\\ &\le
    c\nu h^{\min(t-\frac12,k)} \left(\norm{\bar \y}_{H^{t+\frac12}(\Omega)^d}+
    \norm{\bar p}_{H^{t-\frac12}(\Omega)}\right).\label{eq:bestapproxestimate}
    \end{align}
\end{theorem}

\begin{proof}
    We directly obtain the first part of this result from \eqref{eq:bestassum} and Theorem \ref{th:bestapprox} by applying the trace theorem, i.\,e., there holds
    \[
        \norm{\u-\u_h}_{H^{\frac12}(\Gamma)^d}\le c \norm{\u-v_h}_{H^{\frac12}(\Gamma)^d}\le c \norm{\bar\y-v_h}_{H^{1}(\Omega)^d}
    \]
    for all $v_h\in Y_h$. The second part follows from \eqref{eq:approxYhQh}.
\end{proof}

\begin{example}
    Let us consider the MINI element, the SMALL element as well as the Taylor--Hood $\mathcal{P}_2/\mathcal{P}_1$ element, which satisfy assumptions \eqref{eq:inf-sup-discrete} and \eqref{eq:assum_boundary} as discussed in Example \ref{ex:spaces} and Example \ref{ex:assum_boundary}. If we construct $u_h$, e.\,g., by means of the $L^2(\Gamma)^d$-projection or by means of its appropriately modified versions according to \eqref{eq:lambda_h1}, we next show that \eqref{eq:bestassum} is satisfied. In case of the modified variant, we additionally require
    \[
        \frac{\norm{w_h}_{H^{\frac12}(\Gamma)^d}}{\abs{\langle w_h,n\rangle_\Gamma}}\le ch^{-1},
    \]
    see Example~\ref{ex:proxmod} for different choices of $w_h$, and $\langle u,n\rangle_{\Gamma}=0$ to be satisfied, which is natural in this case. For simplicity, we assume that the underlying triangulation is quasi-uniform. Now, let us first consider the case $u_h=\pi_h u$ with $\pi_h$ being the $L^2(\Gamma)^d$-projection. There holds
    \[
        v_h=\pi_hv_h
    \]
    for all $v_h\in Y_h^\partial$ and
    \[
        \norm{\pi_hv}_{H^{\frac12}(\Gamma)^d}\le c\norm{v}_{H^{\frac12}(\Gamma)^d}
    \]
    for all $v\in H^{\frac12}(\Gamma)$, which is a standard result on quasi-uniform meshes.
    Hence, we obtain
    \begin{align}
        \norm{v_h-\u_h}_{H^{\frac12}(\Gamma)^d}&=\norm{\pi_h(v_h-\u)}_{H^{\frac12}(\Gamma)^d}
        \le c\norm{\u-v_h}_{H^{\frac12}(\Gamma)^d},\label{eq:stabH12}
    \end{align}
    and as a consequence
    \begin{align*}
        \norm{\u-u_h}_{H^{\frac12}(\Gamma)^d} &\le
        \norm{\u-v_h}_{H^{\frac12}(\Gamma)^d} + \norm{v_h-\u_h}_{H^{\frac12}(\Gamma)^d} \le
        c\norm{\u-v_h}_{H^{\frac12}(\Gamma)^d},
    \end{align*}
    which is \eqref{eq:bestassum}.
    In the second case, where $u_h=\tilde \pi_h u$ with $\tilde \pi_h$ being the modified version of the $L^2(\Gamma)^d$-projection according \eqref{eq:lambda_h1}, we observe that
    \[
        \tilde \pi_h u=\Ih_h \u-\lambda_h w_h = \Ih_h \u-\frac{\langle \Ih_h\u,n\rangle_\Gamma}{\langle w_h,n\rangle_\Gamma} w_h=\Ih_h \u-\frac{\langle \Ih_h\u-u,n\rangle_\Gamma}{\langle w_h,n\rangle_\Gamma} w_h
    \]
    due to $\langle u,n\rangle_{\Gamma}=0$. Hence, as $v_h=\pi_hv_h$ for all $v_h\in Y_h^\partial$ and $\langle u-\pi_h u,v_h\rangle_\Gamma=0$ for all $v_h\in Y_h^\partial$, we get
    \begin{align*}
        v_h-\tilde \pi_h u&=\Ih_h(v_h- \u)-\frac{\langle \Ih_h(v_h-\u)-(v_h-u),n\rangle_\Gamma}{\langle w_h,n\rangle_\Gamma} w_h\\
        &=\Ih_h(v_h- \u)-\frac{\langle\pi_h(v_h-u)-(v_h-u),n-P_h n\rangle_\Gamma}{\langle w_h,n\rangle_\Gamma} w_h
    \end{align*}
    with $P_h n$ denoting the Carstensen interpolant of $n$. Subsequently, we deduce by means of \eqref{eq:stabH12}, a standard estimate for the $L^2(\Gamma)^d$-projection and estimate \eqref{eq:Carstensen_h12}
    \begin{align*}
        &\norm{v_h-\tilde \pi_h u}_{H^{\frac12}(\Gamma)^d}\\
        &\le \norm{\pi_h(v_h-\u)}_{H^{\frac12}(\Gamma)^d} + \frac{\abs{\langle\pi_h(v_h-u)-(v_h-u),n-P_h n\rangle_\Gamma}}{\abs{\langle w_h,n\rangle_\Gamma}}\norm{w_h}_{H^{\frac12}(\Gamma)^d}\\
        &\le c\norm{v_h-\u}_{H^{\frac12}(\Gamma)^d} +\norm{\pi_h(v_h-u)-(v_h-u)}_{L^2(\Gamma)^d}\frac{\norm{n-P_h n}_{L^2(\Gamma)^d}}{\abs{\langle w_h,n\rangle_\Gamma}}\norm{w_h}_{H^{\frac12}(\Gamma)^d}\\
        &\le c\norm{v_h-\u}_{H^{\frac12}(\Gamma)^d}\left(1 +h\frac{\norm{w_h}_{H^{\frac12}(\Gamma)^d}}{\abs{\langle w_h,n\rangle_\Gamma}}\right).
    \end{align*}
    The second term is bounded by assumption. Hence, \eqref{eq:bestassum} is also shown in case of the modified variant of the $L^2(\Gamma)^d$-projection. \qed 
\end{example}

We do not only get sharp discretization error estimates in the energy norm but also in weaker norms.

\begin{theorem}[error estimates in weaker norms]\label{lem:L2velocity}
  Let the assumptions \eqref{eq:inf-sup-discrete}, \eqref{eq:approxYhQh} and \eqref{eq:assum_boundary} be satisfied. Assume that $\bar\y\in H^{t+\frac12}(\Omega)^d$, $\bar p\in H^{t-\frac12}(\Omega)$, and $u\in \tilde H^{t'}(\Gamma)$ with some $t'\ge t\ge \frac12$. Let $s\in(\frac12,\xi)\cap(\frac12,1]$ with $\xi\in\R$ from Subsection \ref{sec:regularity} and $s'\in[\frac12,s]\cap[\frac12,1)$. Assume that the approximate boundary condition satisfies \eqref{eq:xxx4} or \eqref{eq:bestassum}, respectively, and
  \begin{align}\label{eq:xxx5}
    \norm{\u-\u_h}_{H^{-s'+\frac12}(\Gamma)^d} \le
    ch^{s+\min(t-\frac12,k)} \norm{\u}_{\tilde H^{t'}(\Gamma)^d}
  \end{align}
  with $k$ from Subsection \ref{sec:discr}.
  Then the error estimate
  \begin{align}\notag
     \nu\norm{\y-\y_h}_{H^{1-s}(\Omega)^d} &+  \norm{p-p_h}_{H^{s}(\Omega)'} \\
     &\le c\nu h^{s+\min(t-\frac12,k)} \left(\norm{\bar\y}_{H^{t+\frac12}(\Omega)^d}+
    \norm{\bar p}_{H^{t-\frac12}(\Omega)}+\norm{\u}_{\tilde H^{t'}(\Gamma)^d}\right) \label{eq:++}
  \end{align}
  holds for $s\in (\frac12,\min(\xi,1))$. If $\xi>1$ and $s=1$, there holds
  \begin{align}\notag
    \nu\norm{\y-\y_h}_{L^2(\Omega)^d} &+ \norm{p-p_h}_{(H^{1}(\Omega)\cap L^2_0(\Omega)\cap V_{-1}^{0,2}(\Omega))'}\\
    &\le c\nu h^{1+\min(t-\frac12,k)} \left(\norm{\bar\y}_{H^{t+\frac12}(\Omega)^d}+
    \norm{\bar p}_{H^{t-\frac12}(\Omega)}+\norm{\u}_{\tilde H^{t'}(\Gamma)^d}\right).\label{eq:++a}
  \end{align}
  If $\u-\u_h\equiv 0$ then both estimates do not explicitly depend on $\norm{\u}_{\tilde H^{t'}(\Gamma)^d}$.
\end{theorem}

The proof is postponed to Subsection \ref{sec:proofH1-s}.

\begin{remark}[estimates in $L^2\times(H^1)'$] \label{rem:erresty-y_h_kap3}
  Since $H^{1-s}(\Omega)^d\hookrightarrow L^2(\Omega)^d$ and
  $H^{s}(\Omega)'\hookrightarrow H^1(\Omega)'$, Theorem
  \ref{lem:L2velocity} implies in particular 
  \begin{align*}
    \nu\|\y-\y_h\|_{L^2(\Omega)^d}&+\|p-p_h\|_{(H^1(\Omega))'} \\&\le
    c\nu h^{s+\min(t-\frac12,k)} \left(\norm{\bar\y}_{H^{t+\frac12}(\Omega)^d}+
    \norm{\bar p}_{H^{t-\frac12}(\Omega)}+\norm{\u}_{\tilde H^{t'}(\Gamma)^d}\right)
  \end{align*}
  for $s\in (\frac12,\min(\xi,1))$. 
  In case of a non-convex domain, i.\,e. $\xi<1$, the error order  is $\xi+\min(t-\frac12,k)-\varepsilon$ for all $\varepsilon>0$, which is smaller than the error order for the best approximation of the solution, which is $1+\min(t-\frac12,k)$. However, if the error is estimated in a slightly stronger norm, $\nu\norm{\y-\y_h}_{H^{1-s}(\Omega)^d} +  \norm{p-p_h}_{H^{s}(\Omega)'}$, $s\in(\frac12,\xi)$, then the error order coincides with that of the best approximation. This observation seems to be new.\qed
\end{remark}

\begin{remark}[the range of $s$]
    Theorem \ref{lem:L2velocity} contains the approximation order $\alpha(s)=s+\min(t-\frac12,k)$ with $s\in(\frac12,\xi)\cap(\frac12,1)$ for $\nu\norm{\y-\y_h}_{H^{1-s}(\Omega)^d} +  \norm{p-p_h}_{H^{s}(\Omega)'}$. Since this result is also true for $s=0$ by combining Theorems \ref{th:approxweak} and \ref{lem:L2velocity}, the order  $\alpha(s)$ should also be correct for $s\in(0,\frac12]$ by using interpolation in Sobolev spaces. However, we did not elaborate the best assumptions for the error $\u-\u_h$. \qed
\end{remark}

\begin{remark}[role of $t'$] \label{rem:role_of_t'}
    In some cases, the regularity $\u\in \tilde H^t(\Gamma)^d$ does not lead to the approximation order $s+\min(t-\frac12,k)$ in assumption \eqref{eq:xxx5} and consequently in the corresponding finite element error estimates. A little more regularity, $\u\in \tilde H^{t'}(\Gamma)^d$ with $t'>t$, may then help, see Example~\ref{ex:approximation orders in weaker norms}. \qed
\end{remark}

\begin{remark}[pressure robustness in weaker norms]
    In particular, we get from Theorem~\ref{lem:L2velocity} that for $s\in (\frac12,\xi)\cap (\frac12,1]$ the estimate
    \begin{align*}
        \norm{\y-\y_h}_{H^{1-s}(\Omega)^d}
        &\le c h^{s+\min(t-\frac12,k)} \left(\norm{\bar\y}_{H^{t+\frac12}(\Omega)^d}+
        \norm{\bar p}_{H^{t-\frac12}(\Omega)}+\norm{\u}_{\tilde H^{t'}(\Gamma)^d}\right)
    \end{align*}
    is satisfied, which does not depend on the viscosity parameter $\nu$. Hence, the estimate is pressure robust independent of a particular choice of the finite elements. If we use finite elements, such as the Scott--Vogelius element, which imply $\nabla \cdot v_h\equiv 0$ for all $v_h\in Y_{0h}^\div$, we even get that the right hand side of this estimate does not depend on the pressure, i.\,e. on $\norm{\bar p}_{H^{t-\frac12}(\Omega)}$. However, let us emphasize that this is not important for the pressure robustness.\qed
\end{remark}

\begin{example}[verification of the assumptions of Theorem~\ref{lem:L2velocity}]\label{ex:veriassumuh}
We verify assumption \eqref{eq:xxx5} for selected finite element pairings and approximations $u_h$. In detail, we discuss the MINI and the Bernardi--Raugel (SMALL) element,  where $k=1$, and the Taylor--Hood $\mathcal{P}_2/\mathcal{P}_1$ element, where $k=2$, see Example~\ref{ex:spaces} and Example~\ref{ex:assum_boundary}. For simplicity, we assume that the underlying triangulation is quasi-uniform. To construct $u_h$, we use the $L^2(\Gamma)^d$-projection, the Lagrange interpolation (if $t'>\frac12(d-1)$) or their appropriately modified versions, see Example~\ref{ex:prox} and Example~\ref{ex:proxmod}. Of course, for the modfied versions, we require $\langle u,n\rangle_{\Gamma}=0$, see Lemma~\ref{lem:tildeI_h}. These approximations $u_h$ satisfy assumption \eqref{eq:xxx4} as already discussed in Example~\ref{ex:approxH1}. Assumption \eqref{eq:xxx5} is a bit more delicate. Indeed, this condition is satisfied for the $L^2(\Gamma)^d$-projection and the Lagrange interpolation and their modified version. However, we have to choose the values $s'$ and $t'$ for each approximation $u_h$ appropriately. The $L^2(\Gamma)^d$-projection satisfies \eqref{eq:xxx5} as follows: From Example~\ref{ex:prox} we deduce
\[
    \norm{\u-\u_h}_{H^{-s'+\frac12}(\Gamma)^d} \le
    ch^{s+t-\frac12} \norm{\u}_{\tilde H^{t+s-s'}(\Gamma)^d}
\]
for $-s'+\frac12\in[-1,1]$ and $t+s-s'\in[-s'+\frac12,k+1]$. The latter condition can be rewritten as $t-\frac12\in[-s,k+s'-s+\frac12]$. Hence, we get
\[
    \norm{\u-\u_h}_{H^{-s'+\frac12}(\Gamma)^d} \le
    ch^{s+\min(t-\frac12,k+s'-s+\frac12)} \norm{\u}_{\tilde H^{t+s-s'}(\Gamma)^d}
\]
for $s'\in[-\frac12,\frac32]$ and $t-\frac12\ge -s$. If we further restrict the ranges for the parameters to $s\in(\frac12,\xi)\cap(\frac12,1]$, $s'\in[\frac12,s]\cap[\frac12,1)$ and $t\ge\frac12$, we obtain $t':=t+s-s'\ge t\ge\frac12$ and $s'-s+\frac12\ge0$ such that
\[
    \norm{\u-\u_h}_{H^{-s'+\frac12}(\Gamma)^d} \le
    ch^{s+\min(t-\frac12,k)} \norm{\u}_{\tilde H^{t'}(\Gamma)^d}.
\]
In case of the Lagrange interpolation we obtain from Example~\ref{ex:prox}
\[
    \norm{\u-\u_h}_{H^{-s'+\frac12}(\Gamma)^d} \le
    ch^{s+t+s'-1} \norm{\u}_{\tilde H^{s+t-\frac12}(\Gamma)^d}
\]
for $-s'+\frac12\in[0,1]$ and $s+t-\frac12\in[\max(-s'+\frac12,\frac12(d-1)),k+1]$. The latter condition is equivalent to $t-\frac12 \in[\max(-s'+\frac12,\frac12(d-1))-s,k+1-s]$ and therefore
\[
    \norm{\u-\u_h}_{H^{-s'+\frac12}(\Gamma)^d} \le
    ch^{s+s'-\frac12+\min(t-\frac12,k+1-s)} \norm{\u}_{\tilde H^{s+t-\frac12}(\Gamma)^d}
\]
for $s'\in[-\frac12,\frac12]$ and $t-\frac12\ge\max(-s'+\frac12,\frac12(d-1))-s$. If we restrict the ranges for the parameters to $s\in(\frac12,\xi)\cap(\frac12,1]$, $s'=\frac12$ and $t>\frac12(d-1)$, we obtain $t':=s+t-\frac12\ge t\ge\frac12$ and $1-s\ge0$ such that
\[
    \norm{\u-\u_h}_{H^{-s'+\frac12}(\Gamma)^d} \le
    ch^{s+\min(t-\frac12,k)} \norm{\u}_{\tilde H^{t'}(\Gamma)^d}.
\]
Hence, assumption \eqref{eq:xxx5} is shown for the $L^2(\Gamma)^d$-projection and the Lagrange interpolation. It also holds for the appropriately modified versions according to Example~\ref{ex:proxmod}.

The Carstensen interpolant and its modified version are not relevant if we are interested in the full order of convergence in Theorem~\ref{lem:L2velocity}. This can be seen by the following arguments: From Example~\ref{ex:prox} we obtain
\[
    \norm{\u-\u_h}_{H^{-s'+\frac12}(\Gamma)^d} \le
    ch^{s+t-\frac12} \norm{\u}_{\tilde H^{t+s-s'}(\Gamma)^d}
\]
for $-s'+\frac12\in[-1,1]$ and $t+s-s'\in[-s'+\frac12,1]$. The latter condition can be rewritten as $t-\frac12\in[-s,s'-s+\frac12]$. Consequently, we get
\[
    \norm{\u-\u_h}_{H^{-s'+\frac12}(\Gamma)^d} \le
    ch^{s+\min(t-\frac12,s'-s+\frac12)} \norm{\u}_{\tilde H^{t+s-s'}(\Gamma)^d}
\]
for $s'\in[-\frac12,\frac32]$ and $t-\frac12\ge -s$.
If we further restrict the ranges for the parameters to $s\in(\frac12,\xi)\cap(\frac12,1]$, $s'\in[\frac12,s]\cap[\frac12,1)$ and $t\ge\frac12$, we obtain at best
\[
    \norm{\u-\u_h}_{H^{-s+\frac12}(\Gamma)^d} \le
    ch^{s+\min(t-\frac12,\frac12)} \norm{\u}_{\tilde H^{t}(\Gamma)^d}
\]
if $s\in(\frac12,\min(\xi,1))$ and
\[
    \norm{\u-\u_h}_{H^{-\frac12+\varepsilon}(\Gamma)^d} \le
    ch^{1+\min(t-\frac12,\frac12-\varepsilon)} \norm{\u}_{\tilde H^{t+\varepsilon}(\Gamma)^d}
\]
with $\varepsilon>0$ if $\xi>1$ and $s=1$.
Hence, it is not possible to achieve the full order of convergence in Theorem~\ref{lem:L2velocity} with the Carstensen interpolant. Therefore it is not further considered in Example~\ref{ex:approximation orders in weaker norms}. However, the Carstensen interpolant will be relevant in Section~\ref{sec:5.1} again.\qed
\end{example}

\begin{example}[approximation orders in weaker norms]\label{ex:approximation orders in weaker norms}
We elaborate the approximation orders in weaker norms for the finite element pairings from Example~\ref{ex:veriassumuh}. For simplicity, we assume that the underlying triangulation is quasi-uniform. To construct $u_h$, we use the $L^2(\Gamma)^d$-projection, the Lagrange interpolation (if $t'>\frac12(d-1)$) or their appropriately modified versions. These approximations $u_h$ satisfy \eqref{eq:xxx4} and \eqref{eq:xxx5} as discussed in Example~\ref{ex:veriassumuh}.
Hence, we have checked the assumptions of Theorem~\ref{lem:L2velocity}. Let us now discuss the approximations orders, which we can achieve. We distinguish between convex and non-nonconvex domains.

We start with the discussion of convex domains.
Then we have $\xi>1\ge s$. Consequently, there holds
\begin{align*} 
     \nu\norm{\y-\y_h}_{H^{1-s}(\Omega)^d} &+  \norm{p-p_h}_{H^{s}(\Omega)'} \\
     &\le c\nu h^{s + \min(t-\tfrac{1}{2}, k)} \left(\norm{\bar \y}_{H^{t+\frac12}(\Omega)^d}+
     \norm{\bar p}_{H^{t-\frac12}(\Omega)}+\norm{\u}_{\tilde H^{t'}(\Gamma)^d}\right)
\end{align*}
for all $s\in(\frac12,1)$ and
\begin{align*} 
     \nu\norm{\y-\y_h}_{L^2(\Omega)^d} &+ \norm{p-p_h}_{(H^{1}(\Omega)\cap L^2_0(\Omega)\cap V_{-1}^{0,2}(\Omega))'}\\
     &\le c\nu h^{1 + \min(t-\tfrac{1}{2}, k)} \left(\norm{\bar\y}_{H^{t+\frac12}(\Omega)^d}+
     \norm{\bar p}_{H^{t-\frac12}(\Omega)}+\norm{\u}_{\tilde H^{t'}(\Gamma)^d}\right).
\end{align*}
The required regularity for $\bar\y$ and $\bar p$ is $\bar\y\in H^{t+\frac12}(\Omega)^d$ and $\bar p\in H^{t-\frac12}(\Omega)$. The regularity of the Dirichlet boundary datum, which we need, is $\u\in \tilde H^{t'}(\Gamma)^d$ with $t'=t+\varepsilon$ for the (normal and modified) $L^2(\Gamma)^d$-projection, as we may only choose $s' = 1- \varepsilon$ with arbitrarily small $\varepsilon>0$, and $\u\in \tilde H^{t'}(\Gamma)^d$ with $t'=t+\frac12$ for the (normal and modified) Lagrange interpolation. Hence, we need more regular Dirichlet boundary data for the Lagrange interpolation than for the $L^2(\Gamma)^d$-projection, which is also known for the Poisson problem. 

If $\Omega$ is non-convex, then $\xi<1$ and $s<\xi$. Hence, there only holds
\begin{align*} 
     \nu\norm{\y-\y_h}_{H^{1-s}(\Omega)^d} &+  \norm{p-p_h}_{H^{s}(\Omega)'} \\
     &\le c\nu h^{s + \min(t-\tfrac{1}{2}, k)} \left(\norm{\bar\y}_{H^{t+\frac12}(\Omega)^d}+
     \norm{\bar p}_{H^{t-\frac12}(\Omega)}+\norm{\u}_{\tilde H^{t'}(\Gamma)^d}\right)
\end{align*}
for all $s\in(\frac12,\xi)$.
As in the convex case, the required regularity for $\bar\y$ and $\bar p$ is $\bar\y\in H^{t+\frac12}(\Omega)^d$ and $\bar p\in H^{t-\frac12}(\Omega)$. The regularity of the Dirichlet boundary datum is $\u\in \tilde H^{t'}(\Gamma)^d$ with $t'=t$ for the (normal and modified) $L^2(\Gamma)^d$-projection, as we may now choose $s' = s$, and $\u\in \tilde H^{t'}(\Gamma)^d$ with $t'=t+\frac12$ for the (normal and modified) Lagrange interpolation.\qed
\end{example}

\subsection{\label{sec:proofapprox}Proof of Theorem \ref{th:bestapprox}}
The first part of the proof follows the lines of \cite[\S\,4.2.1.1]{John:16} where the
case with homogeneous boundary conditions is analyzed.
We start the error analysis with a lemma similar to the Lemma of C\'ea
for coercive problems. To this end we let
\begin{align*}
  Y_{0h}^{\div}=\{v_h^\div\in Y_{0h}:
  (\nabla\cdot v_h^\div,q_h)=0\ \forall q_h\in Q_h \}, \\
  Y_{*h}^{\div}=\{v_h^\div\in Y_{*h}:
  (\nabla\cdot v_h^\div,q_h)=0\ \forall q_h\in Q_h \}.
\end{align*}

\begin{lemma}[best approximation in $Y_{*h}^\div$]
  The estimate
  \begin{align}\label{eq:error_with_div}
    \nu\norm{\nabla(\y-\y_h)}_{L^2(\Omega)^{d\times d}} \le 2\nu 
    \norm{\nabla(\y-w_h^\div)}_{L^2(\Omega)^{d\times d}} + \norm{p-q_h}_{L^2(\Omega)}
  \end{align}
  holds for all $w_h^\div\in Y_{*h}^\div$ and all $q_h\in Q_h$. If in addition $\nabla \cdot v_h\equiv 0$ for all $v_h\in Y_{0h}^\div$ there even holds 
  \begin{align*}
    \norm{\nabla(\y-\y_h)}_{L^2(\Omega)^{d\times d}} \le 
    \norm{\nabla(\y-w_h^\div)}_{L^2(\Omega)^{d\times d}}.
  \end{align*}
\end{lemma}

\begin{proof}
  We first assume that $\nabla \cdot v_h\equiv 0$ for all $v_h\in Y_{0h}^\div$ is not fulfilled and comment on the required modification afterwards if the statement is true. We set $v=v_h^\div\in Y_{0h}^\div\subset Y_{0h}$ in \eqref{eq:weak1} and
  $v_h=v_h^\div$ in \eqref{eq:u_h}, take the difference of the two
  equations and obtain
  \begin{align}  
    0&= \nu(\nabla (\y-\y_h),\nabla v_h^\div)-(\nabla\cdot v_h^\div,p-p_h) \nonumber\\ &=
    \nu(\nabla (\y-\y_h),\nabla v_h^\div)-(\nabla\cdot v_h^\div,p-q_h),\label{eq:div0}
  \end{align}
  where we used in the second step that $v_h^\div$ is discretely
  divergence free. Setting $v_h^\div=\y_h-w_h^\div\in Y_{0h}^{\div}$ leads to
  \begin{align*}  
    0&=\nu(\nabla (\y-\y_h),\nabla (\y_h-w_h^\div))-(\nabla\cdot(\y_h-w_h^\div),p-q_h).
  \end{align*}
  With this identity we conclude
  \begin{align*}
    \nu\norm{\nabla&(\y_h-w_h^\div)}_{L^2(\Omega)^{d\times d}}^2 =
    \nu(\nabla(\y_h-w_h^\div),\nabla(\y_h-w_h^\div)) \\ &=
    \nu(\nabla(\y-w_h^\div),\nabla(\y_h-w_h^\div))-\nu
    (\nabla(\y-\y_h),\nabla(\y_h-w_h^\div))  \\ &=
    \nu(\nabla(\y-w_h^\div),\nabla(\y_h-w_h^\div)) - 
    (\nabla\cdot(\y_h-w_h^\div),p-q_h) \\ &\le
    \norm{\nabla(\y_h-w_h^\div)}_{L^2(\Omega)^{d\times d}} \left(\nu
    \norm{\nabla(\y-w_h^\div)}_{L^2(\Omega)^{d\times d}}+
    \norm{p-q_h}_{L^2(\Omega)} \right)
  \end{align*}
  and as a consequence
  \begin{align*}
    \nu\norm{\nabla&(\y_h-w_h^\div)}_{L^2(\Omega)^{d\times d}} \le\nu
    \norm{\nabla(\y-w_h^\div)}_{L^2(\Omega)^{d\times d}} + \norm{p-q_h}_{L^2(\Omega)}.
  \end{align*}
  If $\nabla \cdot v_h=0$ for all $v_h\in Y_{0h}^\div$ then the second term in \eqref{eq:div0} vanishes and hence also in the previous estimate.
  The previous estimate and the triangle inequality in the form
  \begin{align*}
    \nu\norm{\nabla(\y-\y_h)}_{L^2(\Omega)^{d\times d}} \le
    \nu\norm{\nabla(\y-w_h^\div)}_{L^2(\Omega)^{d\times d}} + 
    \nu\norm{\nabla(\y_h-w_h^\div)}_{L^2(\Omega)^{d\times d}}
  \end{align*}
  lead to the asserted estimates.
\end{proof}

The estimate of $\norm{\nabla(\y-w_h^\div)}_{L^2(\Omega)^{d\times d}}$ in the
previous lemma will be simplified with the help of the following lemma
which is also inspired by \cite[\S\,4.2.1.1]{John:16}.

\begin{lemma}[from $Y_{*h}^\div$ to $Y_{*h}$]\label{lem:fromvhdivtovh}
  Let assumption \eqref{eq:inf-sup-discrete} be satisfied. For each $w_h\in Y_{*h}$ there is a $w_h^\div\in Y_{*h}^\div$ satisfying
  \begin{align*}
    \norm{\nabla(\y-w_h^\div)}_{L^2(\Omega)^{d\times d}}\le\left(1+\alpha^{-1}\right)
    \norm{\nabla(\y-w_h)}_{L^2(\Omega)^{d\times d}}
  \end{align*}
  where $\y\in Y$ satisfies \eqref{eq:weak2} and $\alpha$ is the
  constant in the inf-sup condition \eqref{eq:inf-sup-discrete}.
\end{lemma}

\begin{proof}
  For given $w_h\in Y_{*h}$ set $q_h=\Pi_h(\nabla\cdot w_h):=\mathrm{argmin}_{\phi_h\in Q_h}\norm{\nabla\cdot w_h-\phi_h}_{L^2(\Omega)}$. Hence,
  $\Pi_h:L^2(\Omega)\to Q_h$ is the $L^2(\Omega)$-projection onto $Q_h$. This
  function satisfies
  \begin{align}
    \nonumber \norm{q_h}_{L^2(\Omega)}^2 &=
    (\Pi_h(\nabla\cdot w_h),q_h)=(\nabla\cdot w_h,q_h).
  \end{align}
  Using $(\nabla\cdot \y,q_h)=0$ for all $q_h\in Q_h\subset Q$, which even holds without compatibility condition, we conclude
  \begin{align*}
    \nonumber \norm{q_h}_{L^2(\Omega)}^2 &=
    (\nabla\cdot (w_h-\y),q_h) \le
    \norm{\nabla\cdot (\y-w_h)}_{L^2(\Omega)} \norm{q_h}_{L^2(\Omega)}
  \end{align*}
  and as a consequence
  \begin{align}
    \norm{q_h}_{L^2(\Omega)}&\le\norm{\nabla(\y-w_h)}_{L^2(\Omega)^{d\times d}}.\label{eq:3}
  \end{align}
  By the inf-sup condition \eqref{eq:inf-sup-discrete}, there is a $v_h\in Y_{0h}$ such that 
  \begin{align*}
    \nabla\cdot v_h=q_h \quad\text{and}\quad 
    \norm{\nabla v_h}_{L^2(\Omega)^{d\times d}}\le
    \alpha^{-1}\norm{q_h}_{L^2(\Omega)},
  \end{align*}
  see \cite[Lemma 3.58]{John:16}, hence with \eqref{eq:3}
  \begin{align} \label{eq:2}
    \norm{\nabla v_h}_{L^2(\Omega)^{d\times d}} \le
    \alpha^{-1}\norm{\nabla (\y-w_h)}_{L^2(\Omega)^{d\times d}}.
  \end{align}
  Due to $\nabla\cdot v_h=q_h=\Pi_h(\nabla\cdot w_h)$ we have
  for all $r_h\in Q_h$
  \begin{align*}
    (\nabla\cdot(w_h-v_h),r_h)=
    (\nabla\cdot w_h -\Pi_h(\nabla\cdot w_h),r_h)=0.
  \end{align*}
  Hence, we can set $w_h^\div=w_h-v_h\in Y_{*h}^\div$ (note that $v_h=0$
  on $\Gamma$). The desired result now follows from the triangle inequality,
  \begin{align*}
    \norm{\nabla(\y-w_h^\div)}_{L^2(\Omega)^{d\times d}}&\le
    \norm{\nabla(\y-w_h)}_{L^2(\Omega)^{d\times d}} + 
    \norm{\nabla v_h}_{L^2(\Omega)^{d\times d}},
  \end{align*}
  together with \eqref{eq:2}.
\end{proof}

In the same spirit, we have the following result.

\begin{lemma}[from $Q_h$ to $\tilde Q_h$]\label{lem:intQh}
  For each $\tilde q_h\in \tilde Q_h:=Q_h\oplus\mathcal{P}_0(\Omega)$
  there is a $q_h\in Q_h$ satisfying
  \begin{align}\label{eq:bestapprox}
    \norm{p-q_h}_{L^2(\Omega)}\le 2 \norm{p-\tilde q_h}_{L^2(\Omega)}
  \end{align}
  for arbitrary $p\in L^2_0(\Omega)$.
\end{lemma}

\begin{proof}
  The function $q_h=\tilde q_h-\abs{\Omega}^{-1}\langle\tilde q_h,1\rangle$
  has vanishing mean value by construction and satisfies 
  \begin{align*}
    \norm{p-q_h}_{L^2(\Omega)}&\le\norm{p-\tilde q_h}_{L^2(\Omega)} + 
    \abs{\Omega}^{-\frac12}\Abs{\langle\tilde q_h,1\rangle}.
  \end{align*}
  With 
  \[\Abs{\langle\tilde q_h,1\rangle}=
  \Abs{\langle\tilde q_h-p,1\rangle}\le
  \abs{\Omega}^{\frac12} \norm{p-\tilde q_h}_{L^2(\Omega)}\] 
  we obtain the desired estimate.
\end{proof}

\begin{lemma}[error estimate for the pressure]
  Let assumption \eqref{eq:inf-sup-discrete} be satisfied. The estimate
  \begin{align*}
    \norm{p-p_h}_{L^2(\Omega)}\le 
    \nu\alpha^{-1} \norm{\nabla(\y-\y_h)}_{L^2(\Omega)^{d\times d}} + 
    \left(\alpha^{-1}+1\right) \norm{p-q_h}_{L^2(\Omega)} 
  \end{align*}
  holds for all $q_h\in Q_h$.
\end{lemma}

The proof is identical to that of \cite[Theorem 4.25]{John:16} since
the boundary condition for the velocity does not influence the error
estimate for the pressure. For completeness, we add the proof here.

\begin{proof}
  For arbitrary $v_h\in Y_{0h}$, we set $v=v_h$ in \eqref{eq:weak1} and
  subtract \eqref{eq:u_h} to obtain
  \begin{align*}
    \nu(\nabla(\y-\y_h),\nabla v_h)-(\nabla\cdot v_h,p-p_h)=0,
  \end{align*}
  and with some reformulations we obtain for arbitrary $q_h\in Q_h$
  \begin{align*}
    (\nabla\cdot v_h,p_h-q_h)&=
    -(\nabla\cdot v_h,p-p_h)+(\nabla\cdot v_h,p-q_h)\\ &=
    -\nu(\nabla(\y-\y_h),\nabla v_h)+(\nabla\cdot v_h,p-q_h)\\&\le
    \norm{\nabla v_h}_{L^2(\Omega)^{d\times d}} \left(\nu
    \norm{\nabla(\y-\y_h)}_{L^2(\Omega)^{d\times d}} +
    \norm{p-q_h}_{L^2(\Omega)} \right).
  \end{align*}
  By using the inf-sup condition \eqref{eq:inf-sup-discrete} we conclude 
  with $\|v_h\|_{H^1_0(\Omega)^d}=\norm{\nabla v_h}_{L^2(\Omega)^{d\times d}}$
  \begin{align*}
    \norm{p_h-q_h}_{L^2(\Omega)} &\le\frac1\alpha
    \sup_{v_h\in Y_{0h}\setminus\{0\}}\frac{(\nabla\cdot v_h,p_h-q_h)}{\|v_h\|_{H^1_0(\Omega)^d}} 
    \\ &\le \frac1\alpha \left(\nu
    \norm{\nabla(\y-\y_h)}_{L^2(\Omega)^{d\times d}} +
    \norm{p-q_h}_{L^2(\Omega)} \right).
  \end{align*}
  By using the triangle inequality, we get the desired estimate.
\end{proof}

To summarize until now, it follows from the lemmas of this section that
the discretization error satisfies
  \begin{align}\label{eq:bestapproximationy}
    \nu\norm{\nabla(\y-\y_h)}_{L^2(\Omega)^{d\times d}}&\le 2\nu\left(1+\alpha^{-1}\right)
    \norm{\nabla(\y-w_h)}_{L^2(\Omega)^{d\times d}} + 2\norm{p-\tilde q_h}_{L^2(\Omega)}, \\
    \norm{p-p_h}_{L^2(\Omega)}&\le \label{eq:bestapproximationp}
    \nu\alpha^{-1} \norm{\nabla(\y-\y_h)}_{L^2(\Omega)^{d\times d}} + 
    2\left(\alpha^{-1}+1\right) \norm{p-\tilde q_h}_{L^2(\Omega)} 
  \end{align}
  for all $w_h\in Y_{*h}$ and all $\tilde q_h\in \tilde Q_h$. The second term in \eqref{eq:bestapproximationy} even vanishes if in addition $\nabla \cdot v_h\equiv 0$ for all $v_h\in Y_{0h}^\div$.
  It remains to progress from $w_h\in Y_{*h}$ to general $v_h\in Y_h$.

\begin{lemma}[from $Y_{*h}$ to $Y_h$] \label{lem:fromY*htoYh}
    Let the assumption \eqref{eq:assum_boundary} be satisfied. For each $v_h\in Y_h$ there is a $w_h\in Y_{*h}$ such that
    \begin{align*}
        \norm{\nabla(\y-w_h)}_{L^2(\Omega)^{d\times d}} \le c\left(
        \norm{\y-v_h}_{H^1(\Omega)^d} +
        \norm{\u-\u_h}_{H^{\frac12}(\Gamma)^d} \right)
    \end{align*}
    for all $\y\in H^1(\Omega)^{d}$ with boundary data $\u=\y|_\Gamma$.
\end{lemma}

\begin{proof}
Let $S_h:Y_h^\partial\to Y_h$ be the discrete harmonic extension operator defined in \eqref{eq:discrharmextension}. 
For given $v_h\in Y_h$, the function $w_h$ is chosen as 
\begin{align}\label{eq:w_h}
  w_h=v_h+S_h(\u_h-v_h|_\Gamma).
\end{align}
Obviously, due to the definition of the discrete harmonic extension operator $S_h$, there holds
\[
    w_h|_{\Gamma}=v_h|_\Gamma+S_h(\u_h-v_h|_\Gamma)|_\Gamma=v_h|_\Gamma+\u_h-v_h|_\Gamma=\u_h
\]
such that $w_h\in Y_{*h}$.

The idea is used in a similar fashion in \cite{BCD:04}, where the non-homogeneous Dirichlet boundary condition for the Poisson problem is treated. 
    By means of the construction of $w_h$ in \eqref{eq:w_h} we get
    \[
        \norm{\nabla(\y-w_h)}_{L^2(\Omega)^{d\times d}}\le \norm{\nabla(\y-v_h)}_{L^2(\Omega)^{d\times d}}+\norm{\nabla S_h(\u_h-v_h|_\Gamma})_{L^2(\Omega)^{d\times d}}.
    \]
    Using Lemma~\ref{lem:discreteharmonicstabil} and a standard trace theorem, see e.\,g.\ \cite{Ding:96}, we obtain
    \begin{align}
        \norm{\nabla S_h(\u_h-v_h|_\Gamma)}_{L^2(\Omega)^{d\times d}}
        &\le c_1\norm{\u_h-v_h|_\Gamma}_{H^{\frac12}(\Gamma)^d}\\
        &\le c_1\left(\norm{\u_h-\u}_{H^{\frac12}(\Gamma)^d}+\norm{\u-v_h|_\Gamma}_{H^{\frac12}(\Gamma)^d}\right)\notag\\
        &\le c_1\left(\norm{\u_h-\u}_{H^{\frac12}(\Gamma)^d}+c_2\norm{\y-v_h}_{H^{1}(\Omega)^d}\right),\label{eq:S_hstab}
    \end{align}
    where $c_1$ denotes the constant from assumption \eqref{eq:assum_boundary} and $c_2$ the constant from the trace theorem.
\end{proof}

The proof of Theorem \ref{th:bestapprox} ends by combining the estimates \eqref{eq:bestapproximationy} and \eqref{eq:bestapproximationp} with Lemma \ref{lem:fromY*htoYh} and \eqref{eq:vissol}.

\subsection{Proof of Theorem \ref{lem:L2velocity}}\label{sec:proofH1-s}

We start with an error estimate for a dual problem.

\begin{lemma}
Let $(v,q)\in Y_0\times Q$ be the weak solution of the dual Stokes problem \eqref{eq:stokes_dual_neu} with right hand sides $f$ and $g$ as required for \eqref{eq:aprioridauge}. 
Furthermore, let $v_h\in Y_{0h}$ and $q_h\in Q_h$ be finite element solutions to the dual Stokes problem \eqref{eq:stokes_dual_neu}
defined via
\begin{subequations}\label{eq:stokes_dual_neu_weak_discrete}
\begin{align}
  (\nabla v_h,\nabla \varphi_h)-(\nabla\cdot \varphi_h,q_h) &=\langle f,\varphi_h\rangle \quad\forall \varphi_h\in Y_{0h},\\
  (\nabla\cdot v_h,\psi_h)&=(g,\psi_h) \quad\forall \psi_h\in Q_h. 
\end{align}
\end{subequations}
Then there holds the error estimate
\begin{align}\nonumber
\norm{\nabla &(v-v_h)}_{L^2(\Omega)^{d\times d}} + \norm{q-q_h}_{L^2(\Omega)} \\ \label{eq:errordual} &\le ch^s
\begin{cases}
  \norm{f}_{H^{s-1}(\Omega)^d}+\norm{g}_{H^{s}(\Omega)}&\text{if }s\in[0,\xi)\cap [0,1),\ s\ne\frac12,\\
  \norm{f}_{L^2(\Omega)^d}+\norm{g}_{H^{1}(\Omega)}+
  \norm{g}_{V_{-1}^{0,2}(\Omega)}&\text{if }s=1\text{ and }\xi>1
  \end{cases}
\end{align}
with $\xi\in\R$ from Subsection \ref{sec:regularity}.
\end{lemma}
\begin{proof}
Since $v\in H^{1+s}(\Omega)^d$ and $q\in H^s(\Omega)$ we have the standard error estimate
\[
    \norm{\nabla (v-v_h)}_{L^2(\Omega)^{d\times d}} + 
    \norm{q-q_h}_{L^2(\Omega)} \le ch^s \left( 
    \norm{v}_{H^{1+s}(\Omega)^d}+\norm{q}_{H^s(\Omega)} 
    \right),
\]
which can be deduced from \cite[Theorem 4.21, Theorem 4.25 and Lemma 3.60]{John:16} and Lemma \ref{lem:intQh} in combination with \eqref{eq:approxYh0} and \eqref{eq:approxQh}, see also the proof of Theorem \ref{th:bestapprox} in Section \ref{sec:proofapprox} and the proof of Theorem \ref{th:approxweak} for a similar reasoning.
With \eqref{eq:aprioridauge} we conclude the desired estimate.
\end{proof}

We now derive an intermediate error estimate for the velocity and the pressure highlighting the structure of the error.
\begin{lemma}\label{weak_error_first_part}
  Let $(\y,p)\in Y_*\times Q$ be the weak solution  according to \eqref{eq:weakneu} and $(\y_h,p_h)\in Y_{*h}\times Q_h$ be the discrete solution according to \eqref{eq:u_hp_h}. Let $f$ and $g$ be functions such that the right hand side of \eqref{eq:aprioridauge} is bounded. Then there holds
  \begin{align}
    \nu\langle\y&-\y_h,f\rangle+(p-p_h,g)\nonumber\\
    &\le
     c\left(h^s \left(\nu\norm{\nabla(y-y_h)}_{L^2(\Omega)^{d\times d}} + \norm{p-p_h}_{L^2(\Omega)}\right)
     +\nu\norm{\u-\u_h}_{\tilde H^{t'}(\Gamma)^d}\right)\nonumber\\
    &\qquad\times\begin{cases}
  \norm{f}_{H^{s-1}(\Omega)^d}+\norm{g}_{H^{s}(\Omega)}&\text{if }s\in[0,\xi)\cap [0,1),\ s\ne\frac12,\\
  \norm{f}_{L^2(\Omega)^d}+\norm{g}_{H^{1}(\Omega)}+
  \norm{g}_{V_{-1}^{0,2}(\Omega)}&\text{if }s=1\text{ and } \xi>1
  \end{cases}\label{eq:femdual}
\end{align}
with $\xi\in\R$ from Subsection \ref{sec:regularity}. If in addition
\[
\nabla \cdot \tilde v_h\equiv 0 \quad \forall \tilde v_h\in Y_{0h}^\div=\{v_h^\div\in Y_{0h}:
  (\nabla\cdot v_h^\div,q_h)=0\ \forall q_h\in Q_h \}
\]
then for $s\in[0,\xi)\cap [0,1]$, $s\ne\frac12$, there even holds
\begin{align*}
    \nu\langle\y&-\y_h,f\rangle\le
     c\nu\left(h^s \norm{\nabla(y-y_h)}_{L^2(\Omega)^{d\times d}}
     +\norm{\u-\u_h}_{\tilde H^{t'}(\Gamma)^d}\right)\norm{f}_{H^{s-1}(\Omega)^d}.
\end{align*}
\end{lemma}

\begin{proof}
Let $(v,q)\in Y_0\times Q$ be the weak solution of the dual Stokes problem \eqref{eq:stokes_dual_neu} with right hand sides $f$ and $-g$ instead of $f$ and $g$. For all $f\in L^2(\Omega)^d$ integration by parts leads to
\begin{align*}
     \nu(\y-\y_h,f) &= 
     \nu(\y-\y_h,-\Delta v +\nabla q) \\ &=
     \nu(\nabla(\y-\y_h),\nabla v)-\nu(\nabla\cdot(\y-\y_h),q)+
     \nu\langle \y-\y_h,\partial_nv-qn\rangle_\Gamma \\ &=
     \nu(\nabla(\y-\y_h),\nabla v)-\nu(\nabla\cdot(\y-\y_h),q)+
     \nu\langle \u-\u_h,\partial_nv-qn\rangle_\Gamma
  \end{align*}
and by using the dense embedding of $L^2(\Omega)^d$ in $H^{s-1}(\Omega)^d$ to
\begin{align*}
     \nu\langle\y-\y_h,f\rangle &= \nu(\nabla(\y-\y_h),\nabla v)-\nu(\nabla\cdot(\y-\y_h),q)+
     \nu\langle \u-\u_h,\partial_nv-qn\rangle_\Gamma
  \end{align*}
for all $f\in H^{s-1}(\Omega)^d$.
Moreover, there holds
\[
     (p-p_h,g) = -(p-p_h,\nabla\cdot v).
\]
Let $(v_h,q_h)$ be the solution of \eqref{eq:stokes_dual_neu_weak_discrete} with right hand sides $f$ and $-g$ instead of $f$ and $g$. We choose $v_h$ as test function in \eqref{eq:weak1} and \eqref{eq:u_h} to obtain
\[
    0=\nu(\nabla(\y-\y_h),\nabla v_h)-(\nabla\cdot v_h,p-p_h).
\]
Furthermore, we choose
$q_h$ as test function in \eqref{eq:weak2} and \eqref{eq:p_h} to get
\[
    0=(\nabla\cdot(\y-\y_h),q_h). 
\]
Together, all four equations lead to
\begin{align}
    \nu\langle\y-\y_h,f\rangle&+(p-p_h,-g)\nonumber\\
    &=\nu(\nabla(\y-\y_h),\nabla (v-v_h))-(\nabla\cdot (v-v_h),p-p_h) \nonumber \\ &\quad
     -\nu(\nabla\cdot(\y-\y_h),q-q_h)+\nu\langle \u-\u_h,\partial_nv-qn\rangle_{\Gamma} \nonumber \\
    &\le
     \norm{\nabla (v-v_h)}_{L^2(\Omega)^{d\times d}} \left(\nu
     \norm{\nabla(\y-\y_h)}_{L^2(\Omega)^{d\times d}}+\norm{p-p_h}_{L^2(\Omega)} 
     \right) \nonumber \\ &\quad+\nu\norm{\nabla(\y-\y_h)}_{L^2(\Omega)^{d \times d}}
     \norm{q-q_h}_{L^2(\Omega)} \nonumber\\
     &\quad+
     \nu\norm{\u-\u_h}_{H^{-s'+\frac12}(\Gamma)^d} 
     \norm{\partial_nv-qn}_{H^{s'-\frac12}(\Gamma)^d}.\label{eq:dualestimate}
\end{align}
By using a standard trace estimate, see e.\,g.\ \cite{Ding:96},  and $s'\in[\frac 12,s]\cap[\frac 12,1)$, we conclude by means of \eqref{eq:aprioridauge}
\begin{equation}\label{eq:regvqbdry}
\norm{\partial_nv-qn}_{H^{s'-\frac12}(\Gamma)^d}\le c(
\norm{f}_{H^{s-1}(\Omega)^d}+\norm{g}_{H^{s}(\Omega)}).
\end{equation}
Note that we introduced $s'$ with the restriction $s'<1$ here, since the normal vector is discontinuous such that the norm on the left hand side is not bounded if $s'=1$.
To conclude the first desired estimate, we use \eqref{eq:dualestimate}, \eqref{eq:errordual} and \eqref{eq:regvqbdry}. If $g\equiv0$ and
\[
    \nabla \cdot \tilde v_h\equiv 0 \quad \forall \tilde v_h\in Y_{0h}^\div,
\]
we obtain $\nabla\cdot (v-v_h)\equiv 0$, hence
\[
    (\nabla\cdot (v-v_h),p-p_h)=0,
\]
such that the second estimate of the assertion follows as above.
\end{proof}

\begin{proof}[Proof of Theorem \ref{lem:L2velocity}]
We now distinguish between the estimates for the velocity and the pressure.
We start with the proof of the estimate for the velocity.
Since we assume $1-s\in[0,\frac12)$, we get $H^{1-s}(\Omega)'=H^{1-s}_0(\Omega)'=H^{s-1}(\Omega)$. Consequently, we have
\begin{align}\label{eq:y-yh sup}
\norm{\y-\y_h}_{H^{1-s}(\Omega)^d}
&= \sup_{\varphi\in H^{s-1}(\Omega)^d\setminus\{0\}}
\frac{\langle \y-\y_h,\varphi\rangle}
     {\norm{\varphi}_{H^{s-1}(\Omega)^d}}.
\end{align}
Choosing $f\equiv\varphi$ and $g\equiv0$ in \eqref{eq:femdual} 
and applying the estimates \eqref{eq:est_energynorm} or \eqref{eq:bestapproxestimate}, respectively, and \eqref{eq:xxx5} yields
\[
    \nu\norm{\y-\y_h}_{H^{1-s}(\Omega)^d}\le ch^{s+\min(t-\frac12,k)} \left(\nu\norm{\y}_{H^{t+\frac12}(\Omega)^d}+
    \norm{p}_{H^{t-\frac12}(\Omega)}+\nu\norm{\u}_{\tilde H^{t'}(\Gamma)^d}\right)
\]
for $s\in(\frac12,\xi)\cap(\frac12,1]$, which is the desired estimate for the velocity having in mind that $(y,p)=(\bar y,\nu \bar p)$. 
We now come to the estimates for the pressure. The definition of the $H^{s}(\Omega)'$-norm and the fact that $p-p_h$ belongs to $L^2_0(\Omega)$, yield
\[
    \norm{p-p_h}_{H^{s}(\Omega)'}=\sup_{\varphi\in H^s(\Omega)\setminus\{0\}}
    \frac{( p-p_h,\varphi)}{\norm{\varphi}_{H^s(\Omega)}}=
    \sup_{\varphi\in H^s(\Omega)\setminus\{0\}}
    \frac{(p-p_h,\varphi-\frac{1}{\abs{\Omega}}(\varphi,1))}{\norm{\varphi}_{H^s(\Omega)}}.
\]
We observe that
\begin{align*}
    \norm{\varphi-\abs{\Omega}^{-1}(\varphi,1)}_{H^s(\Omega)}^2&=\abs{\varphi-\abs{\Omega}^{-1}(\varphi,1)}_{H^s(\Omega)}^2+\norm{\varphi-\abs{\Omega}^{-1}(\varphi,1)}_{L^2(\Omega)}^2\\
    &=\abs{\varphi}_{H^s(\Omega)}^2+\norm{\varphi-\abs{\Omega}^{-1}(\varphi,1)}_{L^2(\Omega)}^2
    \le \abs{\varphi}_{H^s(\Omega)}^2+\norm{\varphi}_{L^2(\Omega)}^2
\end{align*}
such that
\begin{equation}\label{eq:varphimean}
    \norm{\varphi-\abs{\Omega}^{-1}(\varphi,1)}_{H^s(\Omega)}\le\norm{\varphi}_{H^s(\Omega)}.
\end{equation}
Hence, $\varphi-\frac{1}{\abs{\Omega}}(\varphi,1)$ belongs to $H^s(\Omega)\cap L^2_0(\Omega)$ for any $s\in[0,1]$ if $\varphi$ belongs to $H^s(\Omega)$. However, it does not belong to $V^{0,2}_{-1}(\Omega)$ as no constant function unequal to zero belongs to $V^{0,2}_{-1}(\Omega)$. 
As a consequence, by choosing $f\equiv0$ and $g\equiv\varphi-\frac{1}{\abs{\Omega}}(\varphi,1)$ in \eqref{eq:femdual} and by using the estimates \eqref{eq:est_energynorm} or \eqref{eq:bestapproxestimate}, respectively, and \eqref{eq:xxx5}, we obtain
\[
    \norm{p-p_h}_{H^{s}(\Omega)'}\le ch^{s+\min(t-\frac12,k)} \left(\nu\norm{\y}_{H^{t+\frac12}(\Omega)^d}+
    \norm{p}_{H^{t-\frac12}(\Omega)}+\nu\norm{\u}_{\tilde H^{t'}(\Gamma)^d}\right)
\]
for $s\in(\frac12,\xi)\cap(\frac12,1)$. 
If $\xi>1$ and $s=1$, we get
\[
    \norm{p-p_h}_{(H^{1}(\Omega)\cap L^2_0(\Omega)\cap V_{-1}^{0,2}(\Omega))'}=\sup_{\varphi\in (H^{1}(\Omega)\cap L^2_0(\Omega)\cap V_{-1}^{0,2}(\Omega))\setminus\{0\}}
    \frac{( p-p_h,\varphi)}{\norm{\varphi}_{H^1(\Omega)}+\norm{\varphi}_{V_{-1}^{0,2}(\Omega)}}.
\]
Choosing $f\equiv0$ and $g\equiv\varphi$ in \eqref{eq:femdual} yields similarly to above
\begin{multline*}
    \norm{p-p_h}_{(H^{1}(\Omega)\cap L^2_0(\Omega)\cap V_{-1}^{0,2}(\Omega))'}\\ \le ch^{1+\min(t-\frac12,k)} \Big(\nu\norm{\y}_{H^{t+\frac12}(\Omega)^d}+
    \norm{p}_{H^{t-\frac12}(\Omega)}
    +\nu\norm{\u}_{\tilde H^{t'}(\Gamma)^d}\Big).
\end{multline*}
The desired estimates for the pressure now follow from the observation $(y,p)=(\bar y,\nu \bar p)$. 
Finally, if $\u-\u_h\equiv 0$ then the estimate \eqref{eq:femdual} does not explicitly depend on $\norm{\u}_{\tilde H^{t'}(\Gamma)^d}$, and $\norm{\u}_{\tilde H^{t'}(\Gamma)^d}$ does not appear in the estimates of the assertion.
\end{proof}

\section{\label{sec:veryweakformulation}Very weak formulation}

\subsection{\label{sec:4.1}Main results}

The weak formulation
\begin{align}
  (\y,p)&\in Y_*\times Q = \{v\in H^1(\Omega)^d:v=\u\text{ a.\,e.~on }\Gamma\}\times L^2_0(\Omega):\nonumber \\
  &\nu(\nabla \y,\nabla v)-\nu(\nabla\cdot \y,q)-(\nabla\cdot v,p)=0 \quad
  \forall (v,q)\in H^1_0(\Omega)^d\times L^2_0(\Omega),\label{eq:weak}
\end{align}
is well defined for $\u\in H^{\frac12}(\Gamma)^d$, see Lemma \ref{lem:2.1}. We recall that $L^2(\Omega)=L^2_0(\Omega)\oplus \mathcal{P}_0(\Omega)$, where $\mathcal{P}_0(\Omega)$ denotes the space of constant functions in $\Omega$. If $\u$ additionally fulfills
\begin{equation}\label{eq:compatibledata}
	\langle \u, n \rangle_\Gamma=0
\end{equation}
such that $(\nabla\cdot \y,q)=0$ for all $q\in\mathcal{P}_0(\Omega)$,
then the pair $(\y,p)$ satisfies
\[
	\nu(\nabla \y,\nabla v)-\nu(\nabla\cdot \y,q)-(\nabla\cdot v,p)=0 \quad
	\forall (v,q)\in H^1_0(\Omega)^d\times L^2(\Omega)
\]
such that it solves \eqref{eq:stokesclassical} in the sense of distributions, see Remark \ref{remark:origsys}.
A further partial
integration leads to 
\[
    \nu(y,-\Delta v +\nabla q)-(\nabla\cdot v,p)=\nu\langle \u,qn-\partial_n v\rangle_\Gamma
\]
provided $v$ and $q$ are smooth enough. This motivates to consider the very weak formulation
\[
  (\y,p)\in \mathcal{ Y}:\quad a((\y,p),(v,q)) = 
  \nu\langle \u,qn-\partial_n v\rangle_\Gamma \quad \forall (v,q)\in \mathcal{V}
\]
with the bilinear form $a:\mathcal{ Y}\times\mathcal{V}\to\R$,
\begin{align}\label{eq:bilinearform}
	a((\y,p),(v,q)):=\nu(\y,-\Delta v+\nabla q)-\langle \nabla\cdot v,p\rangle.
\end{align}
Appropriate spaces $\mathcal{ Y}$ and
$\mathcal{V}$ are
\[
    \mathcal{ Y}:=\mathcal{ Y}_1=L^2(\Omega)^d\times 
	(H^1(\Omega)\cap L_0^2(\Omega)\cap V_{-1}^{0,2}(\Omega))'
\]
and
\begin{alignat*}{2}
	\mathcal{V}&:=\mathcal{V}_1=\{(v,q)\in H^1_0(\Omega)^d\times L^2_0(\Omega)\colon
	&-\Delta v+\nabla q &\in L^2(\Omega)^d,\\ 
	&&\nabla\cdot v&\in H^1(\Omega)\cap L^2_0(\Omega)\cap V_{-1}^{0,2}(\Omega)\},
\end{alignat*}
see also the definition of the spaces $\mathcal{ Y}_s$ and $\mathcal{V}_s$, respectively, in \eqref{eq:Y_s} and \eqref{eq:V_s}.
We will see in Remark \ref{rem:t>1/2} that a solution of the very weak formulation is also a weak solution provided that the solution is smooth enough. Hence, the very weak solution is a more general notion of solution compared to the concept of a weak solution.

Let us now state the main theorem of this section. Its proof is postponed to Section~\ref{sec:vwf_proofs}.
\begin{theorem}[existence of very weak solution]\label{thm:vwf_wellposed}
  Let $\u\in H^{\frac12-s}(\Gamma)^d$ where $s\in[0,\xi)\cap[0,1)$ with $\xi$ from Subsection~\ref{sec:regularity}. Then there is a unique element $(\y,p)\in
  H^{1-s}(\Omega)^d\times \{v\in H^s(\Omega)'\colon \langle v,1\rangle =0\}$ satisfying
  \begin{align}\label{eq:veryweak1} a((\y,p),(v,q))=\nu\langle \u,qn-\partial_n v\rangle_\Gamma\quad 
    \forall(v,q)\in \mathcal{V}.
  \end{align}
  Moreover, there holds
  \[
  	\nu\norm{\y}_{H^{1-s}(\Omega)^d}+\norm{p}_{H^{s}(\Omega)'}\le c \nu\norm{\u}_{H^{\frac12-s}(\Gamma)^d}.
  \]
\end{theorem}

\begin{remark}[importance of $s<\xi$]
  Note that the condition $s<\xi$ is important in Theorem
  \ref{thm:vwf_wellposed} if $\xi<1$, see the discussion in
  \cite[Subsection 2.3]{ALN:24}. \qed
\end{remark}

Let us define the space $\mathcal{\tilde V}$ by
\begin{alignat*}{2}
	\mathcal{\tilde V}&:=\{(v,q)\in H^1_0(\Omega)^d\times L^2(\Omega)\colon 
	&-\Delta v+\nabla q &\in L^2(\Omega)^d,\\
	&&\nabla\cdot v&\in H^1(\Omega)\cap V_{-1}^{0,2}(\Omega)\}
\end{alignat*}
admitting $q\in\mathcal{P}_0(\Omega)$.
Analogously to the weak setting, we get the following results.

\begin{corollary}[from $\mathcal{V}$ to $\mathcal{\tilde V}$]\label{cor:vwf_wellposed}
	Let $\u\in H^{\frac12-s}(\Gamma)^d$ with $s\in[0,\xi)\cap[0,1)$ satisfy the compatibility condition~\eqref{eq:compatibledata}.
	Then there is a unique element $(\y,p)\in
	H^{1-s}(\Omega)^d\times \{v\in H^s(\Omega)'\colon \langle v,1\rangle =0\}$ satisfying
	\begin{align}\label{eq:veryweak}
		a((\y,p),(v,q))=\nu\langle \u,qn-\partial_n v\rangle_\Gamma\quad 
		\forall(v,q)\in \mathcal{\tilde V}.
	\end{align}
	Moreover, there holds
	\[
	\nu\norm{\y}_{H^{1-s}(\Omega)^d}+\norm{p}_{H^{s}(\Omega)'}\le c \nu\norm{\u}_{H^{\frac12-s}(\Gamma)^d}.
	\]
\end{corollary}
\begin{proof}
	By changing $\mathcal{V}$ to $\mathcal{\tilde V}$, we get an additional equation, which may not be fulfilled by the solution of \eqref{eq:veryweak1}.
	Indeed, if we choose $(v,q)\in\mathcal{\tilde V}$ satisfying $v\equiv0$ and $q\in \mathcal{P}_0(\Omega)$, we get
	$
			a((\y,p),(0,q))=0.
	$
	To also obtain $\langle \u,qn-\partial_n v\rangle_\Gamma=\langle \u,qn\rangle_\Gamma=0$ we require \eqref{eq:compatibledata} in addition.
\end{proof}

\begin{remark}[on the parameter $s$]\label{rem:Vt}
	In our notation, the parameter $s$ is the amount of regularity that
	the boundary datum misses for the existence of a weak solution, this
	means that for $s=0$ a weak solution exists but not for $s>0$. One could set
	$t=\frac12-s$ and reformulate Corollary \ref{cor:vwf_wellposed} such that 
    \[
	(\y,p)\in H^{t+\frac12}(\Omega)^d\times
	\{v\in H^{\frac12-t}(\Omega)'\colon \langle v,1\rangle =0\}
	\]
    for $\u\in H^t(\Gamma)^d$, $t\in(\frac12-\xi,\frac12]\cap(-\frac12,\frac12]$, satisfying \eqref{eq:compatibledata}. \qed	
\end{remark}

\begin{remark}[on weak solutions]\label{rem:t>1/2}
  The case $s=0$ in Corollary \ref{cor:vwf_wellposed} ($t=\frac12$ in
  Remark \ref{rem:Vt}) shows in particular that every weak solution is
  also a very weak solution. \qed
\end{remark}

\begin{lemma}[on solving \eqref{eq:stokesclassical} in the distributional sense] \label{lem:4.6}
The solution $(\y,p)$ of \eqref{eq:veryweak1} or \eqref{eq:veryweak}  satisfies $-\nu\Delta \y +\nabla p = 0$ in the distributional sense. If the boundary datum $u$ satisfies the compatibility condition \eqref{eq:compatibledata}, i.\,e., the solution of \eqref{eq:veryweak1} also solves \eqref{eq:veryweak}, then the solution additionally satisfies $\nabla \cdot \y = 0$ in the distributional sense. If $u\in H^\varepsilon(\Gamma)^d$ with $\varepsilon>0$ then the solution of \eqref{eq:veryweak1} or \eqref{eq:veryweak} satisfies the boundary condition $y=u$ almost everywhere on $\Gamma$.
\end{lemma}
\begin{proof}
Choosing $v$ from $C_0^\infty(\Omega)^d$ and $q\equiv 0$ in \eqref{eq:veryweak1} or \eqref{eq:veryweak} shows that the solution $(\y,p)$ satisfies $-\nu\Delta \y +\nabla p = 0$ in the distributional sense. Moreover, choosing $v\equiv0$ and $q$ from $C_0^\infty(\Omega)$ in \eqref{eq:veryweak} (which is possible in \eqref{eq:veryweak} as $C_0^\infty(\Omega)$ is dense in $L^2(\Omega)$, but not in \eqref{eq:veryweak1} as functions from $C_0^\infty(\Omega)$ do not belong to $L^2_0(\Omega)$ in general) shows that $v$ satisfies $\nabla \cdot \y = 0$ in the distributional sense. 

It remains to discuss the equation on the boundary. Using $s=\frac12-\varepsilon<\frac12$ in Theorem~\ref{thm:vwf_wellposed} or Corollary \ref{cor:vwf_wellposed}, we find that $\u$ and  $\y|_\Gamma$ belong to $H^\varepsilon(\Gamma)^d$, where we used the trace theorem for the latter result, see e.\,g.\ \cite{Ding:96}. Further, let $\u_\ell\in H^{\frac12}(\Gamma)^d$ be a sequence which converges to $\u$ in $H^\varepsilon(\Gamma)^d$ for $\ell$ tending to infinity and let $(\y_\ell,p_\ell)$ be corresponding weak solutions with $\y_\ell=\u_\ell$ almost everywhere on $\Gamma$ solving \eqref{eq:weak}. There holds
	\[
		\norm{\y-\u}_{H^\varepsilon(\Gamma)^d}\le \norm{\y-\y_\ell}_{H^\varepsilon(\Gamma)^d}+\norm{\u_\ell-\u}_{H^\varepsilon(\Gamma)^d}.
	\]
	Due to the trace theorem and the a priori estimate from Corollary \ref{cor:vwf_wellposed}, we further get
	\[
		\norm{\y-\y_\ell}_{H^\varepsilon(\Gamma)^d}\le c \norm{\y-\y_\ell}_{H^{\frac12+\varepsilon}(\Omega)^d}\le c \norm{\u-\u_\ell}_{H^\varepsilon(\Gamma)^d}.
	\]
	If we let $\ell$ tend to infinity, we obtain $\y=\u$ almost everywhere on $\Gamma$. 
\end{proof}

\begin{remark}[on the more general Stokes problem]
  The proof of Theorem \ref{thm:vwf_wellposed} is based on an inf-sup
  condition for the bilinear form and the application of the Babu\v
  ska--Lax--Milgram theorem. This means that the result can be proved
  for any right hand side $F\in \mathcal{V}'$. An important example is when
  the Stokes differential equations are non-homogeneous,
  \begin{align*}
    -\nu\Delta \y +\nabla p &= f\quad\text{in }\Omega,\\
    \nabla \cdot \y &=g\quad\text{in }\Omega,\\
    \y&=\u\quad\text{on }\Gamma.
  \end{align*}
  The linear functional on the right hand side of the very weak formulation is then
  \[
    \langle f,v\rangle-\langle g,q\rangle+\nu\langle \u,qn-\partial_n v\rangle_\Gamma
  \]
  such that the result of Theorem \ref{thm:vwf_wellposed} holds if the
  functional is well defined, i.\,e.\ for $(f,g)\in\mathcal{V}'$. 
  Corollary \ref{cor:vwf_wellposed} holds for $(f,g)\in\tilde{\mathcal{V}}'$ as well assuming an appropriate compatibility condition for $g$ in addition.\qed
\end{remark}

\subsection{\label{sec:vwf_proofs}Proof of Theorem \ref{thm:vwf_wellposed}}
Recall the definition of the spaces $\mathcal{ Y}_s$ and $\mathcal{V}_s$, respectively, in \eqref{eq:Y_s} and \eqref{eq:V_s}. With a slight abuse of notation we introduce for $s\in[0,1]\setminus\{\frac12\}$ the bilinear form $a:\mathcal{ Y}_s\times\mathcal{V}_s\to\R$,
\begin{align*}
	a((\y,p),(v,q)):=\nu\langle\y,-\Delta v+\nabla q\rangle-\langle \nabla\cdot v,p\rangle.
\end{align*}
Note that for $s=1$ the bilinear form coincides with that defined in \eqref{eq:bilinearform}.
Obviously, it satisfies
\begin{equation}\label{eq:bilinearformcont}
\abs{a((\y,p),(v,q))}\le c \norm{(\nu\y,p)}_{\mathcal{ Y}_s}\norm{(v,q)}_{\mathcal{V}_s}
\end{equation}
for all $(\y,p)\times(v,q)\in\mathcal{ Y}_s\times\mathcal{V}_s$. Next we show the inf-sup conditions.
\begin{lemma}[inf-sup conditions] \label{lem:infsup}
  The bilinear form $a:\mathcal{ Y}_s\times\mathcal{V}_s$, $s\in
  [0,1]\setminus\{\frac12\}$, satisfies 
  \begin{align}\label{eq:supY}
    \sup_{(v,q)\in \mathcal{V}_s\setminus\{0\}}\frac{a((\y,p),(v,q))}{\norm{(v,q)}_{\mathcal{V}_s}}&\ge c 
    \norm{(\nu\y,p)}_{\mathcal{ Y}_s}\quad \forall (\y,p)\in\mathcal{ Y}_s,\\ \label{eq:supV}
    \sup_{(\y,p)\in \mathcal{ Y}_s\setminus\{0\}}\frac{a((\y,p),(v,q))}{\norm{(\nu\y,p)}_{\mathcal{ Y}_s}}&\ge c 
    \norm{(v,q)}_{\mathcal{V}_s}\quad \forall (v,q)\in\mathcal{V}_s.
  \end{align}
\end{lemma}

\begin{proof}
  We start with the first condition. Let $(v,q)\in H^1_0(\Omega)^d\times L^2_0(\Omega)$ be the solution
  to \eqref{eq:stokes_dual_neu} with $(f,g)\in \mathcal{ Y}_s'$, see also \eqref{eq:dual_distri}. Then we obtain
  \begin{align*}
    \frac{\nu\langle f,\y\rangle + \langle -g, p\rangle}{\norm{(f,g)}_{\mathcal{ Y}_s'}}&=\frac{\nu\langle \y,-\Delta v +\nabla q\rangle-\langle \nabla \cdot v,p\rangle}{\norm{(-\Delta v +\nabla q,\nabla\cdot v)}_{\mathcal{ Y}_s'}}\le c \frac{a((\y,p),(v,q))}{\norm{(v,q)}_{\mathcal{V}_s}},
  \end{align*}
  where we used \eqref{eq:normVs} in the last step.
  Hence, there holds
  \begin{align*}
    \norm{(\nu\y,p)}_{\mathcal{ Y}_s}&=\sup_{(f,g)\in \mathcal{ Y}_s'\setminus\{0\}}\frac{\langle (\nu\y,p),(f,g) \rangle}{\norm{(f,g)}_{\mathcal{ Y}_s'}}=\sup_{(f,g)\in \mathcal{ Y}_s'\setminus\{0\}}\frac{\langle (\nu\y,p),(f,-g) \rangle}{\norm{(f,g)}_{\mathcal{ Y}_s'}}\\
    &=\sup_{(f,g)\in \mathcal{ Y}_s'\setminus\{0\}}\frac{\nu\langle f,\y\rangle + \langle -g, p\rangle}{\norm{(f,g)}_{\mathcal{ Y}_s'}}
    \le c \sup_{(v,q)\in \mathcal{V}_s\setminus\{0\}}\frac{a((\y,p),(v,q))}{\norm{(v,q)}_{\mathcal{V}_s}}.
  \end{align*}
  In case of the second condition, we get
  \begin{align*}
    \norm{(-\Delta v +\nabla q,\nabla\cdot v)}_{\mathcal{ Y}_s'}&=\sup_{(y,p)\in \mathcal{ Y}_s\setminus\{0\}}\frac{\langle (-\Delta v +\nabla q,\nabla\cdot v), (\nu\y,p) \rangle}{\norm{(\nu\y,p)}_{\mathcal{Y}_s}}\\
    &=\sup_{(\y,p)\in \mathcal{ Y}_s\setminus\{0\}}\frac{\nu\langle\y,-\Delta v +\nabla q\rangle + \langle \nabla\cdot v, p \rangle}{\norm{(\nu\y,p)}_{\mathcal{Y}_s}}\\
    &=\sup_{(\y,p)\in \mathcal{ Y}_s\setminus\{0\}}\frac{a((\y,p),(v,q))}{\norm{(\nu\y,p)}_{\mathcal{Y}_s}}.
  \end{align*}
  Hence, the second condition follows from \eqref{eq:normVs}.
\end{proof}
According to \eqref{eq:bilinearformcont} and the previous lemma (inf-sup conditions), we deduce by means of the Babu\v{s}ka--Lax--Milgram theorem, see
e.\,g.~\cite[Thm. 3.13]{John:16}, the
following result:
\begin{lemma}[general existence result]\label{lemma:BLM}
  Let $s\in[0,1]\setminus\{\frac12\}$. For any linear and bounded functional
  $F:\mathcal{V}_s\to\R$ there is a unique element $(\y,p)\in
  \mathcal{ Y}_s$ satisfying
  \begin{equation}\label{eq:veryweakF}
    a((\y,p),(v,q))=F(v,q) \quad \forall (v,q)\in\mathcal{V}_s.
  \end{equation}
  Moreover, there holds
  \[
  	\norm{(\nu\y,p)}_{\mathcal{ Y}_s}\le c \norm{F}_{\mathcal{V}_s'}.
  \]
\end{lemma}

\begin{remark}[solutions coincide for different $s$]\label{rem:allthesame}
  Note that $\mathcal{V}_1\hookrightarrow \mathcal{V}_s$ for $s\le 1$.
  Hence, the solution $(\y,p)\in\mathcal{ Y}_s$ with $s\in[0,1]\setminus\{\frac12\}$ from
  Lemma \ref{lemma:BLM} solves the variational equation
  \begin{align}\label{eq:veryweakx}
    a((\y,p),(v,q))=F(v,q) \quad \forall (v,q)\in\mathcal{V}_1=\mathcal{V}.
  \end{align}
  Since the solution is unique, the solutions from $\mathcal{ Y}_1$ and $\mathcal{ Y}_s$ coincide. \qed
\end{remark}

We now conclude with two corollaries, which form the basis for the main result in Theorem~\ref{thm:vwf_wellposed}. We distinguish between $s>\frac12$ and $s<\frac12$. The case $s=\frac12$ is then treated by an interpolation argument.

\begin{corollary}[$s>\frac12$]\label{cor:BLM1}
  Let $\u\in H^{\frac12-s}(\Gamma)^d$, $s\in [0,\xi)\cap (\frac12,1)$ with $\xi\in\R$ from Subsection \ref{sec:regularity}.
  Then there is a unique solution $(y,p)\in H^{1-s}(\Omega)^d\times (H^{s}(\Omega)\cap L_0^2(\Omega))'$ of the very weak formulation \eqref{eq:veryweak1} satisfying
  \[
 	 \nu\norm{\y}_{H^{1-s}(\Omega)^d}+\norm{p}_{\left(H^{s}(\Omega)\cap L^2_0(\Omega)\right)'}\le c \nu\norm{\u}_{H^{\frac12-s}(\Gamma)^d}.
  \]
\end{corollary}
\begin{proof}
  The functional $F$,
  \[
  	F(v,q):= \nu\langle \u,qn-\partial_n v\rangle_\Gamma,
  \]
  is obviously linear.  Let us discuss its boundedness in
  $\mathcal{V}_s'$. As $s-\frac12>0$, we obtain by means of the
  standard trace theorem  and the a priori
  estimate \eqref{eq:apriori} for $(v,q)$
  \begin{align}\label{eq:stern}
    \norm{qn-\partial_n v}_{H^{s-\frac12}(\Gamma)^d}\le
    c\left(\norm{q}_{H^s(\Omega)}+\norm{v}_{H^{1+s}(\Omega)^d}\right)\le
    c\norm{(v,q)}_{\mathcal{V}_s}.
  \end{align}
  Hence, there holds
  \[
  	\norm{F}_{\mathcal{V}_s'}\le c \nu\norm{\u}_{H^{\frac12-s}(\Gamma)^d}.
  \]
The assertion follows from Lemma \ref{lemma:BLM} and Remark \ref{rem:allthesame} where we use that $H^{1-s}_0(\Omega)^d\linebreak[3]=H^{1-s}(\Omega)^d$ for $s>\frac12$.
\end{proof}

\begin{remark}[reason to exclude $s=1$]
    The normal vector $n$ is piecewise constant and hence not in $H^{\frac12}(\Gamma)^d$. Therefore we exclude the case $s=1$ in \eqref{eq:stern} and consequently in Corollary \ref{cor:BLM1}. \qed
\end{remark}

 For $s < \frac{1}{2}$, we have  $qn - \partial_nv  \in H^{s - \frac{1}{2}}(\Gamma)^d$, which is a negative-order Sobolev space. In this case, we can not apply the trace theorem since the trace operator is only a linear and bounded operator from $H^s(\Omega)$ to $H^{s-\frac{1}{2}}(\Gamma)$ for $s \in (\frac{1}{2},\frac{3}{2})$. Furthermore, the solution $\y$ would have to lie in $H^{1-s}_0(\Omega)^d$, as this is the appropriate dual space. Since $1-s > \frac{1}{2}$, traces are well defined, and in particular this would imply $\y = 0 $ on $\Gamma$. However, this contradicts the actual boundary condition $ \y = \u$ on $\Gamma$. To resolve this issue, we now apply a lifting argument and split $ \y = \y_0 + E\u $, $\y_0 \in H^{1-s}_0(\Omega)^d$, where $E$ is the extension operator 
\begin{align}\label{eq:E}
  E\colon H^{\frac12-s}(\Gamma)^d \to H^{1-s}(\Omega)^d,
\end{align}  see e.\,g.\ \cite{Ding:96}. Accordingly, for $\u\in H^{\frac12-s}(\Gamma)^d$ with $s<\frac12$, instead of considering the very weak formulation \eqref{eq:veryweak1} directly, we seek for $(\y_0,p)\in \mathcal{ Y}$ satisfying
\begin{align}
  a((\y_0,p),(v,q)) &= \nu\langle \u,\, qn - \partial_n v \rangle_\Gamma
                       - a((E\u,0),(v,q))  \nonumber \\
  \label{eq:veryweakhom1}
  &= \nu\langle \u,\, qn - \partial_n v \rangle_\Gamma
     - \nu\langle E\u,\,-\Delta v + \nabla q \rangle \quad\forall (v,q)\in\mathcal{V}_1=\mathcal{V}.
\end{align}
The solution to the very weak formulation \eqref{eq:veryweak1} is then given by $(y,p)=(y_0+Eu,p)\in \mathcal{ Y}$.
\begin{lemma}[reformulation of the boundary term]\label{lem:tildeF=F-fast}
  Let $u\in H^{\frac12-s}(\Gamma)^d$, $s\in [0,\frac12)$, and  $(v,q)\in \mathcal{V}_1$. There holds
  \[
    \langle \u,qn-\partial_n v\rangle_\Gamma = \langle\nabla E\u,qI-\nabla v\rangle+ \langle E\u,-\Delta v +\nabla q\rangle.
  \]
\end{lemma}
\begin{proof}
  Let $\u_\ell\in H^{\frac12}(\Gamma)^d$ be a
  sequence which converges to $\u$ in $H^{\frac12-s}(\Gamma)^d$  for $\ell$ tending to infinity.  Then for
  $(v,q)\in \mathcal{V}_1$ there holds
  \[
    (E\u_\ell,-\Delta v +\nabla q)=\langle\nabla E\u_\ell,\nabla v-qI\rangle+
    \langle \u_\ell,qn-\partial_n v\rangle_\Gamma
  \]
  which is a direct consequence of the integration by parts formula.
  Hence we get for $(v,q)\in \mathcal{V}_1$
  \begin{align*}
    \langle\nabla &E\u,qI-\nabla v\rangle+ \langle E\u,-\Delta v +\nabla q\rangle
\\&=\langle\nabla E(\u-\u_\ell),qI-\nabla v\rangle+ 
    ( E(\u-\u_\ell),-\Delta v +\nabla q)+ 
    \langle \u_\ell-\u,qn-\partial_n v\rangle_\Gamma\\&\quad+
    \langle \u,qn-\partial_n v\rangle_\Gamma.
  \end{align*}
  We notice that
  \[
    \norm{\nabla E(\u-\u_\ell)}_{H^{-s}(\Omega)^{d\times d}}\le c 
    \norm{ E(\u-\u_\ell)}_{H^{1-s}(\Omega)^d}\le c 
    \norm{\u-\u_\ell}_{H^{\frac12-s}(\Gamma)^d},
  \]
  where the latter term converges to zero for $\ell$ tending to infinity.
  Thus, by using \eqref{eq:apriori} and $\mathcal{V}_1\hookrightarrow\mathcal{V}_s$, we obtain the assertion.
\end{proof}

As a consequence of Lemma \ref{lem:tildeF=F-fast}, the very weak formulation \eqref{eq:veryweakhom1} for $(y_0,p)\in\mathcal{ Y}$ can be further reformulated as 
\begin{align}\label{eq:veryweakhom}
    a((y_0,p),(v,q)) =  \nu\langle \nabla E u, qI - \nabla v \rangle \quad\forall (v,q)\in\mathcal{V}_1=\mathcal{V}.
\end{align}

We now prove that problem \eqref{eq:veryweakhom} admits a unique solution $(\y_0, p) \in \mathcal{Y}_s$ with $s<\frac12$.
\begin{lemma}[existence result for problem \eqref{eq:veryweakhom}]\label{lem:veryweakhom}
   Let $\u\in H^{\frac12-s}(\Gamma)^d$, $s\in [0,\frac12)$.
  Then there exists a unique element $(\y_0, p) \in H^{1-s}_0(\Omega)^d\times (H^{s}(\Omega)\cap L_0^2(\Omega))'$ of problem
  \eqref{eq:veryweakhom} satisfying
  \[
    \nu\|\y_0\|_{H^{1-s}(\Omega)^d} + \|p\|_{\left(H^{s}(\Omega)\cap L^2_0(\Omega)\right)'} \leq  c \nu\| \u \|_{H^{\frac{1}{2} -s}(\Gamma)^d}.
  \]
\end{lemma}
\begin{proof}
The functional $F$,
\[
    F(v,q):=\nu\langle \nabla E u, qI - \nabla v \rangle,
\]
is linear. Using \eqref{eq:apriori} we have 
\[
    \abs{F(v,q)}\le \nu\norm{\nabla E\u}_{H^{-s}(\Omega)^{d\times d}}
    \norm{qI-\nabla v}_{H^{s}(\Omega)^{d\times d}} \le c 
    \nu\norm{\nabla E\u}_{H^{-s}(\Omega)^{d\times d}} \norm{(v,q)}_{\mathcal{V}_s}.
\]
Moreover, there holds
\[
    \norm{\nabla E\u}_{H^{-s}(\Omega)^{d\times d}}\le c 
    \norm{ E\u}_{H^{1-s}(\Omega)^d}\le c 
    \norm{\u}_{H^{\frac12-s}(\Gamma)^d}.
\]
Consequently, we obtain
\begin{equation}\label{eq:aprioriy_0}
    \norm{F}_{\mathcal{V}_s'}\le c \nu\norm{\u}_{H^{\frac12-s}(\Gamma)^d}.
\end{equation}
By applying Lemma \ref{lemma:BLM}, there is an unique element $(\y_0,p)\in \mathcal{ Y}_s$ satisfying
\[
a((y_0,p),(v,q)) =  \nu\langle \nabla E u, qI - \nabla v \rangle \quad\forall (v,q)\in\mathcal{V}_s.
\]
As $\mathcal{V}_1\hookrightarrow \mathcal{V}_s$, it also satifies \eqref{eq:veryweakhom}. Using \eqref{eq:aprioriy_0} the a apriori estimate is proven by Lemma \ref{lemma:BLM} as well.
\end{proof}

Using the continuity of the extension operator $E$,
\[ \| E\u \|_{H^{1-s}(\Omega)^d} \leq c \|\u\|_{H^{\frac{1}{2}-s}(\Gamma)^d},\]
with $s\in[0,\frac12)$, and the considerations from above we directly conclude the following result.
\begin{corollary}[$s<\frac12$]\label{cor:BLM2}
Let $\u\in H^{\frac12-s}(\Gamma)^d$, $s\in[0,\frac12)$. Then there is a unique solution $(\y,p)\in
  H^{1-s}(\Omega)^d\times \left(H^{s}(\Omega)\cap L^2_0(\Omega)\right)'$ of the very weak formulation \eqref{eq:veryweak1} satisfying
  \[
    \nu\norm{\y}_{H^{1-s}(\Omega)^d}+\norm{p}_{\left(H^{s}(\Omega)\cap L^2_0(\Omega)\right)'}\le c \nu\norm{\u}_{H^{\frac12-s}(\Gamma)^d}.
  \]
\end{corollary}

To summarize so far, we have treated the cases $s\in [0,\xi)\cap (\frac12,1]$ and $s\in[0,\frac12)$ in Corollaries~\ref{cor:BLM1} and \ref{cor:BLM2}, respectively. It remains to consider the case $s=\frac12$. Unfortunately, this cannot not be done by the techniques from above as this is an exceptional case. Here we can apply the results for $s=\frac12-\varepsilon\in[0,\frac12)$ and for $s=\frac12+\varepsilon \in [0,\xi)\cap (\frac12,1]$ with $\varepsilon>0$ in combination with techniques from real interpolation in Banach spaces, see e.\,g.\ \cite[Chapter 14]{BrennerScott:08}, such that the whole range for $s\in[0,\xi)\cap [0,1]$ is now treated. 

To complete the proof of Theorem \ref{thm:vwf_wellposed}, we need to know that we can replace $p\in (H^s(\Omega)\cap L^2_0(\Omega))'$ with $p \in \{v\in H^s(\Omega)'\colon \langle v,1\rangle =0\}$, using the corresponding norms. This is shown in the following lemma.

\begin{lemma}[identification of dual spaces]\label{rem:equiHs}
    For each $p\in (H^s(\Omega)\cap L^2_0(\Omega))'$, $s\ge0$, there is exactly one $\tilde p \in \{v\in H^s(\Omega)'\colon \langle v,1\rangle =0\}$ satisfying
    \[
        \langle \tilde p ,\varphi\rangle = \langle p, \varphi\rangle \quad \forall \varphi \in H^s(\Omega)\cap L^2_0(\Omega)
    \]
    and
    \begin{equation}\label{eq:equiHs}
        \norm{p}_{\left(H^{s}(\Omega)\cap L^2_0(\Omega)\right)'} = \norm{\tilde p}_{H^{s}(\Omega)'}.
    \end{equation}
\end{lemma}
\begin{proof}
  For each $p\in (H^s(\Omega)\cap L^2_0(\Omega))'$, $s\ge0$, let us define $\tilde p \in H^s(\Omega)'$ by
  \[
    \langle \tilde p,\varphi \rangle:=\langle p, \varphi - \abs{\Omega}^{-1}(\varphi,1)\rangle,
  \]
  which is well-defined as $\varphi - \abs{\Omega}^{-1}(\varphi,1)\in H^s(\Omega)\cap L^2_0(\Omega)$ for $\varphi\in H^s(\Omega)$. By construction, we immediately obtain
  \[
    \tilde p \in \{v\in H^s(\Omega)'\colon \langle v,1\rangle =0\}\quad\text{and}\quad\langle \tilde p ,\varphi\rangle = \langle p, \varphi\rangle \quad \forall \varphi \in H^s(\Omega)\cap L^2_0(\Omega).
  \]
  To see \eqref{eq:equiHs}, we first note that there holds
  \begin{align*}
    \norm{p}_{(H^{s}(\Omega)\cap L^2_0(\Omega))'}=\norm{\tilde p}_{(H^{s}(\Omega)\cap L^2_0(\Omega))'}
    \le \norm{\tilde p}_{H^{s}(\Omega)'}
  \end{align*}
  as $H^{s}(\Omega)\cap L^2_0(\Omega)\subset H^{s}(\Omega)$. 
  For the opposite direction we obtain by means of \eqref{eq:varphimean}
  \begin{align*}
    \norm{\tilde p}_{H^{s}(\Omega)'}
    &=\sup_{\varphi \in H^s(\Omega)\setminus\{0\}}\frac{\langle \tilde p, \varphi\rangle }{\norm{\varphi }_{H^s(\Omega)}}
    = \sup_{\varphi \in H^s(\Omega)\setminus\{0\}}\frac{\langle p, \varphi- \abs{\Omega}^{-1}(\varphi,1)\rangle }{\norm{\varphi }_{H^s(\Omega)}}\\
    &\le \sup_{\varphi \in H^s(\Omega)\setminus\{0\}}\frac{\langle p, \varphi- \abs{\Omega}^{-1}(\varphi,1)\rangle }{\norm{\varphi - \abs{\Omega}^{-1}(\varphi,1)}_{H^s(\Omega)}}
    =\sup_{\varphi \in (H^s(\Omega)\cap L^2_0(\Omega))\setminus\{0\}}\frac{\langle p, \varphi\rangle }{\norm{\varphi}_{H^s(\Omega)}}\\
    &=\norm{p}_{\left(H^{s}(\Omega)\cap L^2_0(\Omega)\right)'},
  \end{align*}
  and the proof is complete.
\end{proof}

\section{\label{sec:NAveryweak}Numerical analysis for the very weak solution}

\subsection{\label{sec:5.1}Main results}

Consider the very weak solution $(\y,p)\in\mathcal{ Y}$ according to \eqref{eq:veryweak1} and the discrete solution $(\y_h,p_h)\in Y_{*h}\times Q_h$ according to \eqref{eq:u_hp_h}. Hence, we use the discretization of a weak formulation with all its properties (see e.\,g.~Remark \ref{rem:tilde}) to approximate the solution of a very weak formulation. However, we have to ensure that the approximations $\u_h$ of the Dirichlet boundary datum $\u$, which has low regularity, are well-posed. 
The approximation order $k=1$ is sufficient for this section, since the regularity of the very weak solution is low.
We derive approximation error estimates
and formulate the results in a way that neither $\langle\u,n\rangle_\Gamma=0$ nor $\langle\u_h,n\rangle_\Gamma=0$ are required.

\begin{theorem}[error estimates in weaker norms]\label{th:erresty-y_h}
  Let $\u\in H^{t'}(\Gamma)^d$ with some $t'\ge t$, $t\in[0,\frac12]$, $s\in(\frac12,\xi)\cap(\frac12,1]$ with $\xi\in\R$ from Subsection~\ref{sec:regularity}, and $s'\in[\frac12,s]\cap[\frac12,1)$. Assume that the discrete spaces fulfill assumptions \eqref{eq:inf-sup-discrete}, \eqref{eq:approxYhQh} and \eqref{eq:assum_boundary}, and the approximate boundary condition satisfies
  \begin{equation}\label{eq:assumvw2}
    \norm{\u-\u_h}_{H^{-s'+\frac12}(\Gamma)^d} \le
    ch^{s+t-\frac12} \norm{\u}_{H^{t'}(\Gamma)^d}
  \end{equation}
  and 
  \begin{equation}\label{eq:assumvw1}
    \norm{\u_h}_{H^{\frac12}(\Gamma)^d}\le c h^{t-\frac12}\norm{\u}_{H^{t}(\Gamma)^d}.
  \end{equation}
  Then there holds the error estimate
  \[
    \nu\norm{\y-\y_h}_{H^{1-s}(\Omega)^d}+\norm{p-p_h}_{H^{s}(\Omega)'} \le c\nu h^{s+t-\frac12} \norm{\u}_{H^{t'}(\Gamma)^d}
  \]
  for $s\in (\frac12,\min(\xi,1))$. If $\xi>1$ and $s=1$, there holds
  \[
    \nu\norm{\y-\y_h}_{L^2(\Omega)^d}+\norm{p-p_h}_{(H^{1}(\Omega)\cap L^2_0(\Omega)\cap V_{-1}^{0,2}(\Omega))'}\le c\nu h^{t+\frac12} \norm{\u}_{H^{t'}(\Gamma)^d}.
  \]
\end{theorem}
The proof is postponed to Subsection \ref{sec:reg_error}.

\begin{remark}[estimates in $L^2\times(H^1)'$] \label{rem:erresty-y_h}
  Since $H^{1-s}(\Omega)^d\hookrightarrow L^2(\Omega)^d$ and
  $H^{s}(\Omega)'\hookrightarrow H^1(\Omega)'$, Theorem
  \ref{th:erresty-y_h} implies in particular 
  \begin{align*}
    \nu\|\y-\y_h\|_{L^2(\Omega)^d}+\|p-p_h\|_{H^1(\Omega)'} &\le
    c\nu h^{s+t-\frac12}\|\u\|_{H^{t'}(\Gamma)^d}
  \end{align*}
  for $s\in (\frac12,\min(\xi,1))$. Note that $\y$ belongs to $H^{t+\frac12}(\Omega)^d$. Hence, one may expect an
  approximation order $t+\frac12$ for $\y$ in the $L^2(\Omega)^d$-norm.
  However, for $\xi < 1$, only the order $s + t - \tfrac{1}{2}$ with $s < \xi$ is obtained. We have
  now shown in Theorem \ref{th:erresty-y_h} that this order is
  also achieved in the $H^{1-s}(\Omega)^d$-norm. It does not increase
  for weaker norms, e.\,g.\ the $L^2(\Omega)^d$-norm. \qed
\end{remark}

\begin{remark}[optimality of estimates]
      For $\u\in L^2(\Gamma)^d$ and $\xi$ close to $\frac12$,
  one obtains a very small convergence order but it is sharp as seen in
  numerical tests, see Section~\ref{sec:test}.  \qed
\end{remark}

\begin{remark}[comparison of assumptions]
    Assumption \eqref{eq:assumvw2} is identical to assumption \eqref{eq:xxx5} of Theorem~\ref{lem:L2velocity} except that in Theorem~\ref{th:erresty-y_h} we assume $t\in[0,\frac12]$, whereas in Theorem~\ref{lem:L2velocity} we consider $t\ge\frac12$. \qed
\end{remark}

\begin{remark}[pressure robustness]
    In particular, we get from Theorem~\ref{th:erresty-y_h} that for $s\in (\frac12,\xi)\cap (\frac12,1]$ the estimate
    \begin{align*}
        \norm{\y-\y_h}_{H^{1-s}(\Omega)^d}\le c h^{s+t-\frac12} \norm{\u}_{H^{t'}(\Gamma)^d}
    \end{align*}
    is satisfied, which does not depend on the viscosity parameter $\nu$. Hence, the estimate is pressure robust independent of a particular choice of the finite elements.\qed
\end{remark}

\begin{example}[verification of the assumptions of Theorem~\ref{th:erresty-y_h}]\label{ex:veriassumvw}
We verify assumptions \eqref{eq:assumvw2} and \eqref{eq:assumvw1} for selected finite element pairings and approximations $u_h$. In detail, we discuss the MINI and the Bernardi--Raugel (SMALL) element,  where $k=1$, and the Taylor--Hood $\mathcal{P}_2/\mathcal{P}_1$ element, where $k=2$, see Example~\ref{ex:spaces}. For simplicity, we assume that the underlying triangulation is quasi-uniform. To construct $u_h$, we use the $L^2(\Gamma)^d$-projection, the Carstensen interpolant, or their appropriately modified versions, see Example~\ref{ex:prox} and Example~\ref{ex:proxmod}. Of course, for the modfied versions, we require $\langle u,n\rangle_{\Gamma}=0$, see Lemma~\ref{lem:tildeI_h}. Lagrange interpolation is not applicable since $u\in H^t(\Gamma)$ with $t\in[0,\frac12]$ only. We start with showing condition \eqref{eq:assumvw2}. 

The $L^2(\Gamma)^d$-projection satisfies \eqref{eq:assumvw2}  as follows: From Example~\ref{ex:prox} we deduce
\[
    \norm{\u-\u_h}_{H^{-s'+\frac12}(\Gamma)^d} \le
    ch^{s+t-\frac12} \norm{\u}_{\tilde H^{t+s-s'}(\Gamma)^d}
\]
for $-s'+\frac12\in[-1,1]$ and $t+s-s'\in[-s'+\frac12,k+1]$. The latter condition can be rewritten as $t-\frac12\in[-s,k+s'-s+\frac12]$. Hence, we get
\[
    \norm{\u-\u_h}_{H^{-s'+\frac12}(\Gamma)^d} \le
    ch^{s+\min(t-\frac12,k+s'-s+\frac12)} \norm{\u}_{\tilde H^{t+s-s'}(\Gamma)^d}
\]
for $s'\in[-\frac12,\frac32]$ and $t-\frac12\ge -s$. If we further restrict the ranges for the parameters to $s\in(\frac12,\xi)\cap(\frac12,1]$, $s'\in[\frac12,s]\cap[\frac12,1)$ and $t\in[0,\frac12]$, we obtain $t':=t+s-s'\ge t$ and $t-\frac12\le0\le k + s'-s+\frac12$ such that
\[
    \norm{\u-\u_h}_{H^{-s'+\frac12}(\Gamma)^d} \le
    ch^{s+t-\frac12} \norm{\u}_{\tilde H^{t'}(\Gamma)^d}.
\]

In case of the Carstensen interpolant we obtain from Example~\ref{ex:prox}
\[
    \norm{\u-\u_h}_{H^{-s'+\frac12}(\Gamma)^d} \le
    ch^{s+t-\frac12} \norm{\u}_{\tilde H^{t+s-s'}(\Gamma)^d}
\]
for $-s'+\frac12\in[-1,1]$ and $t+s-s'\in[-s'+\frac12,1]$. The latter condition can be rewritten as $t-\frac12\in[-s,s'-s+\frac12]$. Consequently, we get
\[
    \norm{\u-\u_h}_{H^{-s'+\frac12}(\Gamma)^d} \le
    ch^{s+\min(t-\frac12,s'-s+\frac12)} \norm{\u}_{\tilde H^{t+s-s'}(\Gamma)^d}
\]
for $s'\in[-\frac12,\frac32]$ and $t-\frac12\ge -s$. If we further restrict the ranges for the parameters to $s\in(\frac12,\xi)\cap(\frac12,1]$, $s'\in[\frac12,s]\cap[\frac12,1)$ and $t\in[0,\frac12]$, we obtain $t':=t+s-s'\ge t$ and $t-\frac12\le0\le s'-s+\frac12$ such that
\[
    \norm{\u-\u_h}_{H^{-s'+\frac12}(\Gamma)^d} \le
    ch^{s+t-\frac12} \norm{\u}_{\tilde H^{t'}(\Gamma)^d}.
\]
Hence, assumption \eqref{eq:assumvw2} is shown for the $L^2(\Gamma)^d$-projection and the Carstensen interpolant. It also holds for the appropriately modified versions according to Example~\ref{ex:proxmod}.

Next, we consider the validity of assumption \eqref{eq:assumvw1}. As we assume quasi-uniform meshes, we can use the inverse inequality to get
  \begin{align*}  
    \norm{\u_h}_{H^{\frac12}(\Gamma)^d}\le c h^{t-\frac12} \norm{\u_h}_{H^t(\Gamma)^d}
  \end{align*}
  for all $t\in[0,\frac12]$. Assumption \eqref{eq:assumvw1} is shown by using the stability result 
  \[
      \norm{\u_h}_{H^t(\Gamma)^d} \le \norm{\u}_{H^t(\Gamma)^d} + \norm{\u-\u_h}_{H^t(\Gamma)^d} \le c\norm{\u}_{H^t(\Gamma)^d}
  \]
  which follows for $\u_h=\Ih_h\u$ from \eqref{eq:assumptionregularizer} and for $\u_h=\tilde\Ih_h\u$ from \eqref{eq:assumptionregularizermod} for all $t\in[0,\frac12]$ if $[0,\frac12]\subset I_t$, see Subsection~\ref{sec:reg_error_neu} for the definition of $I_t$. Estimates \eqref{eq:assumptionregularizer} and \eqref{eq:assumptionregularizermod} can be used since $\u_h$ is defined by the $L^2(\Gamma)$-projection, the Carstensen interpolant, or their appropriately modified versions where $[0,1]\subset I_t$, see Example \ref{ex:prox} and Example~\ref{ex:proxmod}.\qed
\end{example}

\begin{example}[error estimates for selected finite element pairings]
We consider in detail the error estimates of Theorem~\ref{th:erresty-y_h} for the finite element pairings from Example~\ref{ex:veriassumvw}. For simplicity, we assume that the underlying triangulation is quasi-uniform. To construct $u_h$, we use the $L^2(\Gamma)^d$-projection, the Carstensen interpolant or their appropriately modified versions. These approximations $u_h$ satisfy \eqref{eq:assumvw2} and \eqref{eq:assumvw1} as discussed in Example~\ref{ex:veriassumvw}.
We distinguish between convex and non-nonconvex domains.

If the domain is convex, then we have $\xi>1\ge s$. Consequently, there holds
\begin{align*} 
     \nu \norm{\y-\y_h}_{H^{1-s}(\Omega)^d} &+ \norm{p-p_h}_{H^{s}(\Omega)'} \le c\nu h^{s+t-\frac12}\|\u\|_{H^{t'}(\Gamma)^d}
\end{align*}
for all $s\in(\frac12,1)$ and $t\in[0,\frac12]$ and
\begin{align*} 
     \nu\norm{\y-\y_h}_{L^2(\Omega)^d} &+ \norm{p-p_h}_{(H^{1}(\Omega)\cap L^2_0(\Omega)\cap V_{-1}^{0,2}(\Omega))'}\le c\nu h^{t+\frac12}\|\u\|_{H^{t'}(\Gamma)^d}
\end{align*}
for all $t\in[0,\frac12]$. The regularity of the Dirichlet boundary datum, which we need, is $\u\in \tilde H^{t'}(\Gamma)^d$ with $t'=t+\varepsilon$ for the (normal and modified) $L^2(\Gamma)^d$-projection and Carstensen interpolant, as we may only choose $s' = 1- \varepsilon$ with arbitrarily small $\varepsilon>0$. 

If $\Omega$ is non-convex, then $\xi<1$ and $s<\xi$. Hence, there only holds
\begin{align*} 
     \nu\norm{\y-\y_h}_{H^{1-s}(\Omega)^d} &+  \norm{p-p_h}_{H^{s}(\Omega)'} \le c\nu h^{s + t-\tfrac{1}{2}} \norm{\u}_{\tilde H^{t'}(\Gamma)^d}
\end{align*}
for all $s\in(\frac12,\xi)$ and $t\in[0,\frac12]$.
The regularity of the Dirichlet boundary datum is $\u\in \tilde H^{t'}(\Gamma)^d$ with $t'=t$ for the (normal and modified) $L^2(\Gamma)^d$-projection and Carstensen interpolant, as we may now choose $s' = s$.   \qed 
\end{example}

\subsection{\label{sec:reg_error}Proof of Theorem \ref{th:erresty-y_h}}

The novelty of this subsection is that the analysis of the discretization error has to be carried out on the basis of a very weak solution such that the weak formulation \eqref{eq:weakneu} cannot be used.
To this end, let us introduce the
intermediate functions $\y^h\in\{v\in H^1(\Omega)^d:v|_\Gamma=\u_h\}$,
$p^h\in Q$ such that
\begin{align}\label{eq:u^h_1}
  \nu(\nabla \y^h,\nabla v)-(\nabla\cdot v,p^h) - (\nabla\cdot \y^h,q)&=0 
  \quad\forall (v,q)\in Y_0\times Q. 
\end{align}
Note that the boundary condition satisfies $\u_h\in 
 Y_h^\partial\subset H^{\frac12}(\Gamma)^d$ such that the weak solution $(\y^h,p^h)$ exists. Hence, it may be denoted as a regularized solution. To prove Theorem~\ref{th:erresty-y_h} we next split the desired error term into two, the regularization error and the finite element error for the regularized solution.

\begin{lemma}[regularization error]\label{lem:regerror}
  Let $\u\in H^{t'}(\Gamma)^d$ with some $t'\ge t\ge0$, $s\in(\frac12,\xi)\cap(\frac12,1]$ with $\xi\in\R$ from Subsection~\ref{sec:regularity}, and $s'\in[\frac12,s]\cap[\frac12,1)$. Assume that the approximate boundary condition satisfies
  \begin{align}\label{eq:approxbound}
    \norm{\u-\u_h}_{H^{-s'+\frac12}(\Gamma)^d} \le
    ch^{s+t-\frac12} \norm{\u}_{H^{t'}(\Gamma)^d}.
  \end{align}
  Then the regularization error satisfies
  \begin{align}\label{eq:reg_error}
    \nu\norm{\y-\y^h}_{H^{1-s}(\Omega)^d}+\norm{p-p^h}_{H^{s}(\Omega)'} &\le c\nu h^{s+t-\frac12} \norm{\u}_{H^{t'}(\Gamma)^d}.
  \end{align}
\end{lemma}
\begin{proof}
  As mentioned in Remark \ref{rem:t>1/2}, the weak solution $(\y^h,p^h)$ is also a very weak solution. In particular, it satisfies 
  \begin{align*}
    a((\y-\y^h,p-p^h),(v,q))&= \nu\langle \u- \u_h,\partial_n v-qn\rangle_\Gamma \quad
    \forall (v,q)\in \mathcal{V}.
  \end{align*} 
  The a priori estimate of Theorem~\ref{thm:vwf_wellposed} and assumption \eqref{eq:approxbound} then imply
  \[
    \nu\norm{\y-\y^h}_{H^{1-s'}(\Omega)^d}+\norm{p-p^h}_{H^{s'}(\Omega)'} \le c\nu\|\u-\u_h\|_{H^{-s'+\frac12}(\Gamma)^d}\le c\nu h^{s+t-\frac12} \norm{\u}_{H^{t'}(\Gamma)^d}.
  \]
  Observing that $H^{1-s'}(\Omega)^d\hookrightarrow H^{1-s}(\Omega)^d$ and $H^{s'}(\Omega)'\hookrightarrow H^{s}(\Omega)'$ ends the proof.
\end{proof}

\begin{remark}[alternative definition of the very weak solution]
  From the proof we can conclude that 
  $\lim_{h\to0}(\nu\norm{\y-\y^h}_{H^{1-s'}(\Omega)^d}+\norm{p-p^h}_{H^{s'}(\Omega)'})=0$ whenever
  $\lim_{h\to0}\|\u-\u_h\|_{H^{-s'+\frac12}(\Gamma)^d}=0$. This means that the very weak solution $(\y,p)$
  can also be defined as the limit of the weak solution $(\y^h,p^h)$
  with regularized boundary datum $\u_h$.  \qed
\end{remark}

Observing that $(\y^h,p^h)$ is a weak solution and $(\y_h,p_h)$ the
finite element approximation with the same datum $\u_h\in
 Y_h^\partial$, we can use the finite element error estimates for weak
solutions from Section \ref{sec:3.1}.

\begin{lemma}[finite element error for the regularized solution]\label{lem:femveryweak}
  Let $\u\in H^{t}(\Gamma)^d$, $t\in[0,\frac12]$, and assume that the assumptions \eqref{eq:inf-sup-discrete}, \eqref{eq:approxYhQh} and \eqref{eq:assum_boundary} as well as the estimate
  \begin{equation}\label{eq:uh12}
    \norm{\u_h}_{H^{\frac12}(\Gamma)^d}\le c h^{t-\frac12}\norm{\u}_{H^{t}(\Gamma)^d}
  \end{equation}
  are satisfied. Then the error estimate 
  \begin{align*} 
     \nu\norm{\y^h-\y_h}_{H^{1-s}(\Omega)^d} +  \norm{p^h-p_h}_{H^{s}(\Omega)'}
     &\le c\nu h^{s+t-\frac12}\norm{\u}_{H^{t}(\Gamma)^d}
  \end{align*}
  holds for $s\in(\frac12,\min (\xi,1))$ with $\xi\in\R$ from Subsection  \ref{sec:regularity}. If $\xi>1$, there additionally holds
  \[
    \nu\norm{\y^h-\y_h}_{L^2(\Omega)^d}+\norm{p^h-p_h}_{(H^{1}(\Omega)\cap L^2_0(\Omega)\cap V_{-1}^{0,2}(\Omega))'}\le c\nu h^{t+\frac12} \norm{\u}_{H^{t}(\Gamma)^d}.
  \]
\end{lemma}

\begin{proof}
  A direct application of Theorem \ref{lem:L2velocity}, having in mind $(y,p)=(\bar y,\nu \bar p)$, results in the estimates
  \begin{align*} 
     \nu\norm{\y^h-\y_h}_{H^{1-s}(\Omega)^d} +  \norm{p^h-p_h}_{H^{s}(\Omega)'}
     &\le ch^{s}\left(\nu\norm{\y^h}_{H^{1}(\Omega)^d}+
    \norm{p^h}_{L^2(\Omega)}\right)
  \end{align*}
  for $s\in (\frac12,\min(\xi,1))$ and
  \[
    \nu\norm{\y^h-\y_h}_{L^2(\Omega)^d}+\norm{p^h-p_h}_{(H^{1}(\Omega)\cap L^2_0(\Omega)\cap V_{-1}^{0,2}(\Omega))'}\le c h \left(\nu\norm{\y^h}_{H^{1}(\Omega)^d}+
    \norm{p^h}_{L^2(\Omega)}\right)
  \]
  for $\xi>1$ and $s=1$. Using the regularity result from Lemma \ref{lem:2.1} or Theorem \ref{thm:vwf_wellposed} and assumption \eqref{eq:uh12}
  shows the assertion. 
\end{proof}
Theorem~\ref{th:erresty-y_h} is now a direct consequence of Lemma~\ref{lem:regerror} and Lemma~\ref{lem:femveryweak}.

\section{\label{sec:test}Test example}

In this section, we will present a series of numerical experiments to validate the theoretical results. The focus lies on verifying that the approximation orders presented in Sections \ref{sec:NumerikWeak} and \ref{sec:NAveryweak} are indeed achieved in practice. To this end, we consider a test example with a known analytical solution, for which the regularity can be prescribed as required. 

The main error estimates of this paper are formulated in
  Theorem \ref{th:approxweak} or Theorem~\ref{th:bestapproxindependent}, respectively, and in Theorem \ref{lem:L2velocity} for the weak
  solution and in Theorem~\ref{th:erresty-y_h} for the very weak
  solution. Together with the Remarks \ref{rem:erresty-y_h_kap3} and \ref{rem:erresty-y_h} we find the
  common error estimate
\begin{align}\label{eq:ee}
  \norm{\y-\y_h}_{L^2(\Omega)^d} &\le ch^{s+\min(t-\frac12,k)}
\end{align}
with $s=1$ if\/ $\Omega$ is convex and $s\in(\frac12,\xi)$ if $\Omega$
is non-convex, using $\xi\in(\frac12,1)$ from Subsection
\ref{sec:regularity}. 
In this estimate, the constant $c$ also depends on $y\in H^{t+\frac12}(\Omega)^d$, $p\in H^{t-\frac12}(\Omega)$, and $u \in H^{t'}(\Omega)^d$.
In order to confirm the estimate numerically we
consider two domains~$\Omega$ with maximal interior angle $ \omega =
\frac{2}{3} \pi$ in the convex case and $\omega = \frac{3}{2} \pi$ in
the non-convex case, see Figure \ref{fig:domains}.
\begin{figure}[tbp]\hspace*{0.02\textwidth}
  \includegraphics[width=0.38\textwidth]{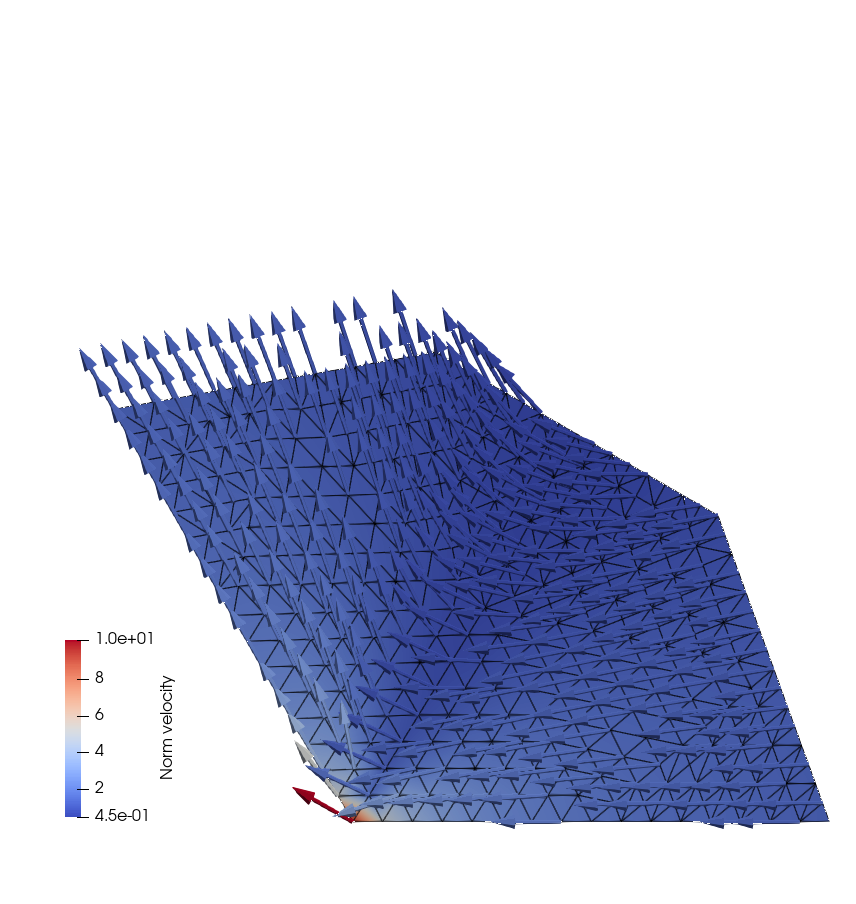}\hfill
  \includegraphics[width=0.58\textwidth]{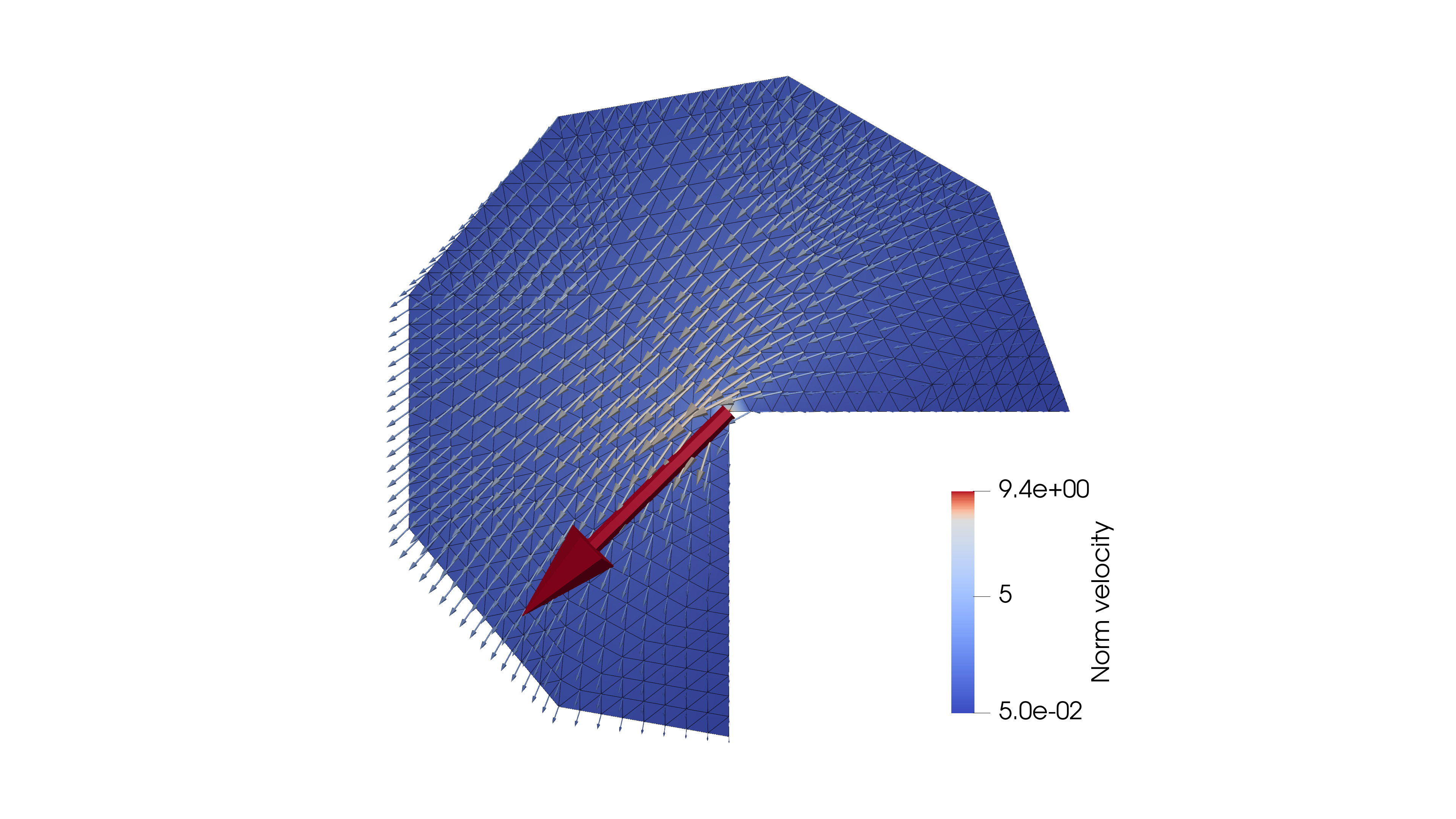}
  \caption{\label{fig:domains}Test domains with maximal angle $ \omega =\frac{2}{3} \pi$ and $\omega = \frac{3}{2} \pi$, resp., with velocity for $\alpha=-0.499$}
\end{figure}
In the latter case one obtains from \cite{dauge:89} that $\xi\approx
0.544$, see also Remark \ref{rem:xi}. In order to use various values
of the regularity parameter $t$ we use the following manufactured
solutions the Stokes problem \eqref{eq:stokesclassical}. Their origin
is explained in \cite{ALN:24}.  

Using polar coordinates $(r,\theta)$, $\theta\in[0,\omega]$, the functions
\[ \y=\left(\begin{array}{c}
 r^\alpha \Phi_1(\theta)  \\
 r^\alpha \Phi_2(\theta)
\end{array}\right), 
\qquad p = r^{\alpha-1} \Phi_p(\theta), \]
with 
\begin{align*}
\Phi_1(\theta) & = -\sin(\alpha\theta)\cos\omega
-\alpha\sin\theta\cos(\alpha(\omega-\theta)+\theta) \\ & \qquad
+\alpha\sin(\omega-\theta)\cos(\alpha\theta-\theta)
+\sin(\alpha(\omega-\theta)),
\\
\Phi_2(\theta) & = -\sin(\alpha\theta)\sin\omega
-\alpha\sin\theta\sin(\alpha(\omega-\theta)+\theta)
-\alpha\sin(\omega-\theta)\sin(\alpha\theta-\theta),
\\
\Phi_p(\theta) & = 2\alpha \,[ 
\sin((\alpha-1)\theta+\omega)
+\sin((\alpha-1)\theta-\alpha\omega) ]
\end{align*}
solve the homogeneous Stokes equations \eqref{eq:stokesclassical}, 
for any $\alpha\in\R$ and $\nu =1$.
Since $\y\in H^{t+\frac12}(\Omega)^d$ and $p\in H^{t-\frac12}(\Omega)$ for $t<\frac12+\alpha$ 
we can drive the regularity by choosing an appropriate value of
$\alpha>-1$.

In the first numerical experiment, we use the inf-sup stable Taylor--Hood element
$\mathcal{P}_2/\mathcal{P}_1$, so that $k=2$. The meshes are
uniformly refined. The boundary data are treated with the $L^2(\Gamma)$-projection, both into $Y_h^\partial$ as well into the space $\tilde Y_h^\partial$, where the projection is defined by 
using $w_h = y_0|_{\Gamma} \in Y_h^{\partial}$ with $y_0=\frac12\binom{x_1-\bar x_1}{x_2-\bar x_2}$ where $(\bar x_1,\bar x_2)$ is a
random interior point of the domain. Regardless of which method is chosen, we expect the approximation
order in \eqref{eq:ee} to be approximately
\[
  s+t-\frac12\approx\begin{cases}
    1+\alpha & \text{for the convex example domain,} \\
    0.5445+\alpha & \text{for the non-convex example domain,}
  \end{cases}
\]
for $t\le \frac52$. Tables \ref{tab:convex} and \ref{tab:non-convex}
\begin{table}[bt]
  \centering
  \caption{\label{tab:convex}Discretization error $e_h=\| y - y_h \|_{L^2(\Omega)^d}$ and 
      experimental convergence order for interior angle $\omega = \frac23 \pi$}
  \begin{tabular}{ccccccccc}\hline
           & \multicolumn{2}{c}{$\alpha = 0.5$} 
           & \multicolumn{2}{c}{$\alpha = 0.1$} 
           & \multicolumn{2}{c}{$\alpha = -0.1$} 
           & \multicolumn{2}{c}{$\alpha = -0.499$} \\
       $h$ &  $e_h$ & eoc    &  $e_h$ & eoc    &  $e_h$ & eoc    &  $e_h$ & eoc    \\ \hline
  $2^{-1}$  & 0.0104 &-     & 0.0079 & -      & 0.0190 & -      & 0.5429 & -      \\
  $2^{-2}$  & 0.0036 & 1.5160 & 0.0039 & 1.0199 & 0.0107 & 0.8302 & 0.3887 & 0.4818 \\
  $2^{-3}$  & 0.0012 & 1.4986 & 0.0018 & 1.0934 & 0.0057 & 0.8951 & 0.2750 & 0.4996 \\
  $2^{-4}$  & 0.0005 & 1.4996 & 0.0009 & 1.1000 & 0.0031 & 0.9001 & 0.1943 & 0.5010 \\
  $2^{-5}$  & 0.0002 & 1.4998 & 0.0004 & 1.0999 & 0.0017 & 0.9000 & 0.1373 & 0.5010 \\
  $2^{-6}$  & 5.7E-5 & 1.4999 & 0.0002 & 1.0999 & 0.0009 & 0.9000 & 0.0970 & 0.5010 \\ 
  $2^{-7}$  & 2.0E-5 & 1.5000 & 8.6E-5 & 1.1000 & 0.0005 & 0.9000 & 0.0685 & 0.5010 \\
  $2^{-8}$  & 7.1E-6 & 1.5000 & 4.9E-5 & 1.1000 & 0.0003 & 0.9000 & 0.0484 & 0.5010 \\
   $2^{-9}$  & 2.5E-6 & 1.5000 & 1.9E-5 & 1.1000 & 0.0001 & 0.9000 & 0.0342 & 0.5010 \\\hline
  expected &        & 1.5000 &        & 1.1000 &        & 0.9000 &        & 0.5010 \\ \hline
  \end{tabular}
  \caption{\label{tab:non-convex}Discretization error $e_h=\| y - y_h \|_{L^2(\Omega)^d}$ and 
      experimental convergence order for interior angle $\omega = \frac32 \pi$}
  \begin{tabular}{ccccccccc}\hline
           & \multicolumn{2}{c}{$\alpha = 0.5$} 
           & \multicolumn{2}{c}{$\alpha = 0.1$} 
           & \multicolumn{2}{c}{$\alpha = -0.1$} 
           & \multicolumn{2}{c}{$\alpha = -0.499$} \\
       $h$ &  $e_h$ & eoc    &  $e_h$ & eoc    &  $e_h$ & eoc    &  $e_h$ & eoc    \\ \hline
  $2^{-1}$  & 0.0697 &-     & 0.0480 & -      & 0.0546 & -      & 0.5510 & -      \\
  $2^{-2}$  & 0.0288 & 1.2753 & 0.0274 & 0.8095 & 0.0353 & 0.6291 & 0.5594 & 0.0216 \\
  $2^{-3}$  & 0.0116 & 1.3107 & 0.0156 & 0.8120 & 0.0231 & 0.6103 & 0.5561 & 0.0085 \\
  $2^{-4}$  & 0.0048 & 1.2628 & 0.0091 & 0.7712 & 0.0156 & 0.5662 & 0.5451 & 0.0287 \\
  $2^{-5}$  & 0.0021 & 1.2035 & 0.0055 & 0.7289 & 0.0109 & 0.5242 & 0.5311 & 0.0377 \\
  $2^{-6}$  & 0.0009 & 1.1493 & 0.0034 & 0.6963 & 0.0077 & 0.4928 & 0.5159 & 0.0418 \\ 
  $2^{-7}$  & 0.0004 & 1.1084 & 0.0021 & 0.6746 & 0.0056 & 0.4723 & 0.5005 & 0.0437 \\
  $2^{-8}$  & 0.0002 & 1.0815 & 0.00135 & 0.6614 & 0.0040 & 0.4600 & 0.4853 & 0.0446 \\
   $2^{-9}$  & 9.9E-5 & 1.0651 & 0.0009 & 0.6538 & 0.0030 & 0.4530 & 0.4703 & 0.0450 \\\hline
  expected &        & 1.0445 &        & 0.6445 &        & 0.4445 &        & 0.0445 \\ \hline
  \end{tabular}
\end{table}
show the computed error norms $e_h=\norm{\y-\y_h}_{L^2(\Omega)^d}$ for
various mesh sizes $h$ and the experimental approximation orders eoc
computed from two succesive meshes.  The results confirm the proven
error orders well, in the case of the convex domain even particularly
well. Both projections resulted in the same numerical results up to the eighth digit such that we present the results in the same tables.

Our numerical tests with the MINI element confirmed the expected error orders as well, gave no further insights, and are therefore not documented here.

In a further experiment, we investigate the pressure robustness of the discretization and the influence of the compatibility correction of the boundary data. Therefore, we consider the same manufactured solution as before. To keep $\y$ as the exact velocity solution for every value of $\nu > 0$ we rescale the pressure via $ \tilde p = \nu p$ so that $(\y,\tilde p)$ solves 
\begin{align*}
    -\nu\Delta y+\nabla \tilde p &= 0 \quad\text{in }\Omega, \\
    \nabla\cdot y&=0\quad \text{in }\Omega, \\
    y&=u\quad\text{on }\Gamma.
\end{align*}
In order to keep the comparability, the computations are carried out with the Taylor--Hood element $\mathcal{P}_4/\mathcal{P}_3$ on the original mesh and the Scott--Vogelius element $\mathcal{P}_4/\mathcal{P}_3^{disc}$ on the barycentrically refined mesh.
Moreover, the boundary data are again treated with the $L^2(\Gamma)$-projection both without and with the compatibility correction introduced above. Tables
\ref{tab: SV-hom-nocorrection} and \ref{tab: TH-hom-nocorrection} report the
results without the correction, tables \ref{tab: SV-hom-correction} and
\ref{tab: TH-hom-correction} report the results with the correction.

\begin{table}[t]
  \centering
    \caption{\label{tab: SV-hom-nocorrection}Discretization error
$e_h=\norm{y-y_h}_{L^2(\Omega)^d}$, divergence error
$\eta_h=\norm{\div(y_h)}_{L^2(\Omega)}$ for 
$\mathcal{P}_4/\mathcal{P}_3^{disc}$, $\omega=\frac32\pi$, $h = 2^{-4}$,
boundary datum without correction}
  \begin{tabular}{ccccccccccc}\hline
           & \multicolumn{2}{c}{$\alpha = 0.5$}
           & \multicolumn{2}{c}{$\alpha = 0.1$}
           & \multicolumn{2}{c}{$\alpha = -0.1$} 
           & \multicolumn{2}{c}{$\alpha = -0.499$}\\
       $\nu$ & $e_h$ & $\eta_h$    &  $e_h$ & $\eta_h$    & $e_h$ & $\eta_h$ & $e_h$ & $\eta_h$    \\ \hline
  $1$     & 0.0004 & 3.7E-8 & 0.0004 & 3.2E-7   & 0.0020 & 1.5E-6 & 0.4969 & 9.5E-5   \\
  $10^{-1}$ & 0.0004 & 3.7E-8 & 0.0004  & 3.2E-7 & 0.0020  & 1.5E-6 & 0.4969 & 9.5E-5 \\
  $10^{-2}$ & 0.0004 & 3.7E-8 & 0.0004  & 3.2E-7 & 0.0020  & 1.5E-6 & 0.4969 & 9.5E-5 \\
  $10^{-3}$ & 0.0004 & 3.7E-8 & 0.0004  & 3.2E-7 & 0.0020  & 1.5E-6 & 0.4969 & 9.5E-5 \\
  $10^{-4}$ & 0.0004 & 3.7E-8 & 0.0004  & 3.2E-7 & 0.0020  & 1.5E-6 & 0.4969 & 9.5E-5 \\
  $10^{-5}$ & 0.0004 & 3.7E-8 & 0.0004  & 3.2E-7 & 0.0020  & 1.5E-6 & 0.4969 & 9.5E-5 \\ \hline
  \end{tabular}

  
    \caption{\label{tab: TH-hom-nocorrection}Discretization error
$e_h=\norm{y-y_h}_{L^2(\Omega)^d}$, divergence error
$\eta_h=\norm{\div (y_h)}_{L^2(\Omega)}$ for 
$\mathcal{P}_4/\mathcal{P}_3$, $\omega=\frac32\pi$, $h = 2^{-4}$,
boundary datum without correction}
  \begin{tabular}{ccccccccccc}\hline
           & \multicolumn{2}{c}{$\alpha = 0.5$}
           & \multicolumn{2}{c}{$\alpha = 0.1$}
           & \multicolumn{2}{c}{$\alpha = -0.1$} 
           & \multicolumn{2}{c}{$\alpha = -0.499$}\\
       $\nu$ & $e_h$ & $\eta_h$    &  $e_h$ & $\eta_h$    & $e_h$ & $\eta_h$ & $e_h$ & $\eta_h$   \\ \hline
  $1$     & 0.0004 & 0.0717  & 0.0022 & 0.1833   & 0.0055 & 0.3696 & 0.5189 & 2.9523   \\
  $10^{-1}$   & 0.0004 & 0.0717 & 0.0022  & 0.1833 & 0.0055  & 0.3696 & 0.5189 & 2.9523 \\
  $10^{-2}$    & 0.0004 & 0.0717 &  0.0022  & 0.1833 & 0.0055  & 0.3696 & 0.5189 & 2.9523 \\
  $10^{-3}$   & 0.0004 & 0.0717 &  0.0022  & 0.1833 & 0.0055  & 0.3696 & 0.5189 & 2.9523 \\
  $10^{-4}$   & 0.0004 & 0.0717 &  0.0022  & 0.1833 & 0.0055  & 0.3696 & 0.5189 & 2.9523 \\
  $10^{-5}$    & 0.0004 & 0.0717 &  0.0022  & 0.1833 & 0.0055  & 0.3696 & 0.5189 & 2.9523 \\ \hline
  \end{tabular}

  \end{table}

\begin{table}[t]
  \centering
  \caption{\label{tab: SV-hom-correction}Discretization error
$e_h=\norm{y-y_h}_{L^2(\Omega)^d}$, divergence error
$\eta_h=\norm{\div (y_h)}_{L^2(\Omega)}$ for 
$\mathcal{P}_4/\mathcal{P}_3^{disc}$, $\omega=\frac32\pi$, $h = 2^{-4}$,
boundary datum with correction}
  \begin{tabular}{ccccccccccc}\hline
           & \multicolumn{2}{c}{$\alpha = 0.5$}
           & \multicolumn{2}{c}{$\alpha = 0.1$}
           & \multicolumn{2}{c}{$\alpha = -0.1$} 
           & \multicolumn{2}{c}{$\alpha = -0.499$}\\
       $\nu$ & $e_h$ & $\eta_h$    &  $e_h$ & $\eta_h$    & $e_h$ & $\eta_h$ & $e_h$ & $\eta_h$   \\ \hline
  $1$   & 0.0004 & 2.9E-11  & 0.0004 & 8.0E-12   & 0.0020 & 7.8E-12 & 0.4969 & 2.6E-11   \\
  $10^{-1}$   & 0.0004 & 3.4E-11 & 0.0004  & 9.3E-12 & 0.0020  & 8.6E-12 & 0.4969 & 2.4E-11 \\
  $10^{-2}$   & 0.0004 & 2.7E-11 & 0.0004  & 6.5E-12 & 0.0020  & 6.1E-12 & 0.4969 & 2.0E-11 \\
  $10^{-3}$   & 0.0004 & 1.4E-11 & 0.0004  & 4.8E-12 & 0.0020  & 3.8E-12 & 0.4969 & 1.2E-11\\
  $10^{-4}$   & 0.0004 & 2.0E-11 & 0.0004  & 5.4E-12 & 0.0020  & 5.9E-12 & 0.4969 & 2.4E-11 \\
  $10^{-5}$   & 0.0004 & 1.6E-11 & 0.0004  & 3.4E-12 & 0.0020  & 3.4E-12 & 0.4969 & 1.0E-11 \\ \hline
  \end{tabular}
  
    \caption{\label{tab: TH-hom-correction}Discretization error
$e_h=\norm{y-y_h}_{L^2(\Omega)^d}$, divergence error
$\eta_h=\norm{\div( y_h)}_{L^2(\Omega)}$ for 
$\mathcal{P}_4/\mathcal{P}_3$, $\omega=\frac32\pi$, $h = 2^{-4}$,
boundary datum with correction}
  \begin{tabular}{ccccccccccc}\hline
           & \multicolumn{2}{c}{$\alpha = 0.5$}
           & \multicolumn{2}{c}{$\alpha = 0.1$}
           & \multicolumn{2}{c}{$\alpha = -0.1$} 
           & \multicolumn{2}{c}{$\alpha = -0.499$}\\
       $\nu$ & $e_h$ & $\eta_h$    &  $e_h$ & $\eta_h$    & $e_h$ & $\eta_h$ & $e_h$ & $\eta_h$   \\ \hline
  $1$    & 0.0004 & 0.0717 & 0.0022 & 0.1833   & 0.0055 & 0.3696 & 0.5189 & 2.9523   \\
  $10^{-1}$  & 0.0004 & 0.0717 & 0.0022  & 0.1833 & 0.0055  & 0.3696 & 0.5189 & 2.9523 \\
  $10^{-2}$   & 0.0004 & 0.0717 &  0.0022  & 0.1833 & 0.0055  & 0.3696 & 0.5189 & 2.9523 \\
  $10^{-3}$   & 0.0004 & 0.0717 &  0.0022  & 0.1833 & 0.0055  & 0.3696 & 0.5189 & 2.9523 \\
  $10^{-4}$  & 0.0004 & 0.0717 &  0.0022  & 0.1833 & 0.0055  & 0.3696 & 0.5189 & 2.9523 \\
  $10^{-5}$  & 0.0004 & 0.0717 &  0.0022  & 0.1833 & 0.0055  & 0.3696 & 0.5189 & 2.9523 \\ \hline
  \end{tabular}
  \end{table}
All the tables show, for both the Taylor--Hood and the Scott--Vogelius element, the discretization error $e_h=\|y-y_h\|_{L^2(\Omega)^d}$ remains unchanged when varying the viscosity parameter $\nu$, with and without the correction. 
The correction does, however, reduce the divergence error $\eta_h=\|\operatorname{div}(y_h)\|_{L^2(\Omega)}$ for the Scott--Vogelius element, while
for Taylor--Hood the divergence error remains unchanged. 
We conclude that the compatibility correction improves the pointwise divergence-freedom of the discrete velocity where applicable (i.\,e. for Scott--Vogelius), but has no measurable
effect on the pressure robustness itself, for either element pair.

\appendix

\section{Counterexample}

In this section, we provide an example that the condition $\langle\u,n\rangle_\Gamma=0$ does not imply $\langle \Ih_h\u,n\rangle_\Gamma=0$.
\begin{example}[$\Ih_h:L^2(\Gamma)\to Y_h^\partial\setminus\tilde Y_h^\partial$]\label{bsp:counterexampleDGL}
Let $\Omega=(0,1)^2$. The function $\u:\Gamma\to\R^2$, 
\[
\u(x_1,x_2)=\begin{cases}(1, 0)^T & \text{if } \tfrac12\le x_1<1\text{ and }x_2=1 \\ (0,0)^T&\text{else,}\end{cases}
\]
satisfies $\u\in L^2(\Gamma)^2$ and 
\[
\langle\u,n\rangle_\Gamma=0.
\]
Let $\Ih_h:L^2(\Gamma)^2\to Y_h^\partial$ be the $L^2(\Gamma)$-projection or the Carstensen interpolant. We are going to show that
\[
\langle \Ih_h\u,n\rangle_\Gamma\not=0.
\]
To this end, we introduce the nodes
\[
a_1=(0,0)^T, \quad a_2=(1,0)^T, \quad a_3=(1,1)^T, \quad a_4=(0,1)^T, 
\]
the boundary elements
\[
E_1=(a_1,a_2)^T, \quad E_2=(a_2,a_3)^T, \quad E_3=(a_3,a_4)^T, \quad E_4=(a_4,a_1)^T, 
\]
and the triangulation 
\[
\mathcal{E}_h=\{E_1,E_2,E_3,E_4\}
\]
of the boundary. We also define the space $Y_h^\partial$ as the space of continuous, piecewise linear functions over the triangulation $\mathcal{E}_h$. Let $\varphi_j$, $j=1,\ldots,4$, denote scalar piecewise linear functions with $\varphi_j(a_i)=\delta_{ij}$ such that
\[
 Y_h^\partial=\left\{v_h=\sum_{j=1}^4(v_{1,j},v_{2,j})^T\varphi_j,\ v_{1,j},v_{2,j}\in\R\right\}.
\]
The $L^2(\Gamma)^2$-projection $\Ih_h\u=\sum_{j=1}^4(\u_{1,j},\u_{2,j})^T\varphi_j$ of $\u\in L^2(\Gamma)^2$ is defined via 
\[
  \langle\Ih_h\u,v_h\rangle_\Gamma=
  \langle\u,v_h\rangle_\Gamma\quad\forall v_h\in Y_h^\partial.
\]
This means that the coefficients $\u_{1,j},\u_{2,j}$ solve the systems
\[
  \begin{pmatrix}
  \langle\varphi_1,\varphi_1\rangle_\Gamma & \langle\varphi_2,\varphi_1\rangle_\Gamma & \langle\varphi_3,\varphi_1\rangle_\Gamma & \langle\varphi_4,\varphi_1\rangle_\Gamma \\
  \langle\varphi_1,\varphi_2\rangle_\Gamma & \langle\varphi_2,\varphi_2\rangle_\Gamma & \langle\varphi_3,\varphi_2\rangle_\Gamma & \langle\varphi_4,\varphi_2\rangle_\Gamma \\
  \langle\varphi_1,\varphi_3\rangle_\Gamma & \langle\varphi_2,\varphi_3\rangle_\Gamma & \langle\varphi_3,\varphi_3\rangle_\Gamma & \langle\varphi_4,\varphi_3\rangle_\Gamma \\
  \langle\varphi_1,\varphi_4\rangle_\Gamma & \langle\varphi_2,\varphi_4\rangle_\Gamma & \langle\varphi_3,\varphi_4\rangle_\Gamma & \langle\varphi_4,\varphi_4\rangle_\Gamma 
  \end{pmatrix}
  \begin{pmatrix}
  \u_{i,1} \\ \u_{i,2} \\ \u_{i,3} \\ \u_{i,4} 
  \end{pmatrix} =
  \begin{pmatrix}
  \langle \u_i,\varphi_1\rangle_\Gamma \\ \langle \u_i,\varphi_2\rangle_\Gamma \\ \langle \u_i,\varphi_3\rangle_\Gamma \\ \langle \u_i,\varphi_4 \rangle_\Gamma
  \end{pmatrix},\quad i=1,2.
\]
In detail, we obtain
\[
  \frac16\begin{pmatrix}
  4&1&0&1\\1&4&1&0\\0&1&4&1\\1&0&1&4 
  \end{pmatrix}
  \begin{pmatrix}
  \u_{1,1} \\ \u_{1,2} \\ \u_{1,3} \\ \u_{1,4} 
  \end{pmatrix} = \frac18
  \begin{pmatrix}
  0\\0\\3\\1 
  \end{pmatrix},\quad 
  \frac16\begin{pmatrix}
  4&1&0&1\\1&4&1&0\\0&1&4&1\\1&0&1&4 
  \end{pmatrix}
  \begin{pmatrix}
  \u_{2,1} \\ \u_{2,2} \\ \u_{2,3} \\ \u_{2,4} 
  \end{pmatrix} =
  \begin{pmatrix}
  0\\0\\0\\0 
  \end{pmatrix},
\]
such that
\[
  \Ih_h\u=\frac1{32}\left[
  \begin{pmatrix}1\\0\end{pmatrix}\varphi_1 +
  \begin{pmatrix}-5\\0\end{pmatrix}\varphi_2 +
  \begin{pmatrix}19\\0\end{pmatrix}\varphi_3 +
  \begin{pmatrix}1\\0\end{pmatrix}\varphi_4
  \right],
\]
but 
\[
  \langle \Ih_h\u,n\rangle_\Gamma=\frac1{32}\left(
  \int_{E_2}(-5\varphi_2+19\varphi_3) - \int_{E_4}(\varphi_1+\varphi_4)
  \right)=\frac{7-1}{32}=\frac3{16}\not=0.
\]
In case that $\Ih_h$ is the Carstensen interpolant, the coefficients are defined via
\[
\u_{i,j}=\frac{\langle\u_i,\varphi_j\rangle_\Gamma}{\langle1,\varphi_j\rangle_\Gamma}
\]
such that we obtain in the present case
\[
  \Ih_h\u=\frac18\left[
  \begin{pmatrix}0\\0\end{pmatrix}\varphi_1 +
  \begin{pmatrix}0\\0\end{pmatrix}\varphi_2 +
  \begin{pmatrix}3\\0\end{pmatrix}\varphi_3 +
  \begin{pmatrix}1\\0\end{pmatrix}\varphi_4
  \right]
\]
and as a result
\[
  \langle \Ih_h\u,n\rangle_\Gamma=\frac18\left(
  \int_{E_2}3\varphi_3 - \int_{E_4}\varphi_4
  \right)=\frac18\left(\frac32-\frac12\right)=\frac18\not=0.
\]
Hence both interpolants do not map into $\tilde Y_h^\partial$.\qed
\end{example}

\section*{Acknowledgment}

We would like to thank Hannes Meinlschmidt for the fruitful discussions on interpolation spaces and Tabea Tscherpel for bringing their paper \cite{EickmannScottTscherpel:25} to our attention.

\bibliographystyle{alphaurl}
\bibliography{stokes}

\end{document}